\documentclass[11pt]{amsart}

\usepackage{mathrsfs}
\usepackage{amsfonts}
\usepackage{amssymb}
\usepackage{amsxtra}
\usepackage{dsfont}
\usepackage{color}
\usepackage{graphicx}
\usepackage[english,polish]{babel}
\usepackage{enumerate}

\usepackage{newtxmath}
\usepackage{newtxtext}

\usepackage[margin=2.5cm, centering]{geometry}
\usepackage[colorlinks,citecolor=blue,urlcolor=blue,bookmarks=true]{hyperref}
\hypersetup{
pdfpagemode=UseNone,
pdfstartview=FitH,
pdfdisplaydoctitle=true,
pdfborder={0 0 0}, 
pdftitle={Asymptotic behaviour of Christoffel--Darboux kernel via three-term recurrence relation I},
pdfauthor={Grzegorz Świderski and Bartosz Trojan},
pdflang=en-US
}

\newcommand{\CC}{\mathbb{C}}
\newcommand{\ZZ}{\mathbb{Z}}

\newcommand{\NN}{\mathbb{N}}
\newcommand{\RR}{\mathbb{R}}
\newcommand{\PP}{\mathbb{P}}

\newcommand{\calR}{\mathcal{R}}
\newcommand{\calC}{\mathcal{C}}

\newcommand{\calO}{\mathcal{O}}

\newcommand{\calB}{\mathcal{B}}

\newcommand{\calA}{\mathcal{A}}

\newcommand{\calX}{\mathcal{X}}

\newcommand{\calD}{\mathcal{D}}

\newcommand{\frakB}{\mathfrak{B}}
\newcommand{\frakX}{\mathfrak{X}}
\newcommand{\frakA}{\mathfrak{A}}

\newcommand{\vphi}{\varphi}

\newcommand{\Id}{\operatorname{Id}}

\newcommand{\pl}[1]{\foreignlanguage{polish}{#1}}

\newcommand{\norm}[1]{\lVert {#1} \rVert}
\newcommand{\abs}[1]{\lvert {#1} \rvert}

\newcommand{\tr}{\operatorname{tr}}

\newcommand{\GL}{\operatorname{GL}}
\newcommand{\discr}{\operatorname{discr}}
\newcommand{\diag}{\operatorname{diag}}
\newcommand{\Dom}{\operatorname{Dom}}
\newcommand{\sinc}{\operatorname{sinc}}

\newcommand{\ud}{{\: \rm d}}
\newcommand{\ue}{\textrm{e}}
\newcommand{\supp}{\operatornamewithlimits{supp}}

\newtheorem{theorem}{Theorem}[section]
\newtheorem{conjecture}[theorem]{Conjecture}
\newtheorem{proposition}[theorem]{Proposition}
\newtheorem{lemma}[theorem]{Lemma}
\newtheorem{corollary}[theorem]{Corollary}
\newtheorem{claim}[theorem]{Claim}
\newtheorem*{theorem*}{Theorem}

\theoremstyle{definition}
\newtheorem{example}[theorem]{Example}
\newtheorem{remark}[theorem]{Remark}
\newtheorem{definition}[theorem]{Definition}

\usepackage[utf8]{inputenc}

\numberwithin{equation}{section}

\title[Asymptotic behaviour of Christoffel--Darboux kernel]
{Asymptotic behaviour of Christoffel--Darboux kernel via three-term recurrence relation I}

\author{Grzegorz Świderski}
\address{
	\pl{
	Grzegorz \'Swiderski \\
	Department of Mathematics \\
	KU Leuven \\
	Celestijnenlaan 200B box 2400 \\
	BE-3001 Leuven \\
	Belgium \&
	Mathematical Institute \\
	University of Wrocław \\
	pl. Grunwaldzki 2/4 \\
	50-384 Wrocław \\
	Poland}
}
\email{grzegorz.swiderski@kuleuven.be}

\author{Bartosz Trojan}
\address{
	\pl{
		Bartosz Trojan\\
      	Institute of Mathematics\\
        Polish Academy of Sciences\\
        ul. \'Sniadeckich 8\\
        00-696 Warszawa\\
        Poland}
}
\email{btrojan@impan.pl}



\begin{document}
\selectlanguage{english}

\begin{abstract}
	For Jacobi parameters belonging to one of the three classes: asymptotically periodic, periodically
	modulated and the blend of these two, we study the asymptotic behavior of the Christoffel functions and the scaling
	limits of the Christoffel--Darboux kernel. We assume regularity of Jacobi parameters in terms of the Stolz class.
	We emphasize that the first class only gives rise to measures with compact supports. 
\end{abstract}

\maketitle

\section{Introduction}
Let $\mu$ be a probability measure on the real line with infinite support such that for every $n \in \NN_0$,
\[
	\text{the moments} \quad \int_{\RR} x^n \ud \mu(x) \quad \text{are finite}.
\]
Let $L^2(\RR, \mu)$ be the Hilbert space of square-integrable functions equipped with the scalar product
\[
	\langle f, g \rangle = \int_\RR f(x) \overline{g(x)} \ud \mu(x).
\]
By performing the Gram--Schmidt orthogonalization process on the sequence of monomials $(x^n : n \in \NN_0)$ one obtains
the sequence of polynomials $(p_n : n \in \NN_0)$ satisfying
\begin{equation} \label{eq:1}
	\langle p_n, p_m \rangle = \delta_{nm}
\end{equation}
where $\delta_{nm}$ is the Kronecker delta. Moreover, $(p_n : n \in \NN_0)$ satisfies the following recurrence relation
\begin{equation}
	\label{eq:100}
	\begin{aligned} 
	p_0(x) &= 1, \qquad p_1(x) = \frac{x - b_0}{a_0}, \\
	x p_n(x) &= a_n p_{n+1}(x) + b_n p_n(x) + a_{n-1} p_{n-1}(x), \qquad n \geq 1
	\end{aligned}
\end{equation}
where
\[
	a_n = \langle x p_n, p_{n+1} \rangle, \qquad
	b_n = \langle x p_n, p_n \rangle, \qquad n \geq 0.
\]
Notice that for every $n$, $a_n > 0$ and $b_n \in \RR$. The pair $(a_n)$ and $(b_n)$ is called the Jacobi 
parameters.

Let $\PP_n$ be the orthogonal projection in $L^2(\RR, \mu)$ on the space of polynomials of degree at most $n$. Then
$\PP_n$ is given by the \emph{Christoffel--Darboux} kernel $K_n$, that is
\[
	\PP_n f(x) = \int_\RR K_n(x,y) \overline{f(y)} \ud \mu(y)
\]
where from \eqref{eq:1} one can verify that
\begin{equation}
	\label{eq:104}
	K_n(x, y) = \sum_{j=0}^n p_j(x) p_j(y).
\end{equation}
To motivate the study of Christoffel--Darboux kernels see surveys \cite{Lubinsky2016} and \cite{Simon2008}.

The asymptotic behavior of $K_n$ is well understood in the case when the measure $\mu$ has compact support. 
In this setup one of the most general results has been proven in \cite{Totik2009}. Namely, if $I$ is an open
interval contained in $\supp(\mu)$ such that $\mu$ is absolutely continuous on $I$ with continuous positive density
$\mu'$, then
\begin{equation} 
	\label{eq:97}
	\lim_{n \to \infty} \frac{1}{n} K_n \Big( x + \frac{u}{n}, x + \frac{v}{n} \Big) = 
	\frac{\omega'(x)}{\mu'(x)} \sinc \big( (u-v) \pi \omega'(x) \big)
\end{equation}
locally uniformly with respect to $x \in I$ and $u, v \in \RR$, provided that $\mu$ is \emph{regular} 
(see \cite[Definition 3.1.2]{StahlTotik1992}). In the formula \eqref{eq:97}, 
\[
	\sinc(x) = 
	\begin{cases}
		\frac{\sin(x)}{x} & \text{if } x \neq 0, \\
		1 & \text{otherwise,}
	\end{cases}
\]
and $\omega'$ denotes the density of the equilibrium measure corresponding to the support of $\mu$, see \eqref{eq:96} for details. In the case when $\supp(\mu)$
is a finite union of compact intervals, $\mu$ is regular provided that $\mu' > 0$ almost everywhere in the interior of 
$\supp(\mu)$. To give some historical perspective, let us also mention three earlier results. The case
$\supp(\mu)=[-1,1]$ and $u=v=0$ has been examined in \cite[Theorem 8]{Mate1991}, and its extension to a general compact
$\supp(\mu)$ has been obtained in \cite[Theorem 1]{Totik2000}. The extension to all $u,v \in \RR$
has been proven in \cite{Lubinsky2009}.

The best understood class of measures with unbounded support is the class of \emph{exponential weights}
(see the monograph \cite{Levin2001}). 
In \cite{Levin2009} and \cite[Theorem 7.4]{Levin2008} under a number of 
regularity conditions on the function $Q(x) = -\log \mu'(x)$, the following analogue of \eqref{eq:97} was shown
\begin{equation} \label{eq:98}
	\lim_{n \to \infty} 
	\frac{1}{\tilde{K}_n(x,x)} 
	\tilde{K}_n \Big( x + \frac{u}{\tilde{K}_n(x,x)}, x + \frac{v}{\tilde{K}_n(x,x)} \Big)  = 
	\sinc (u-v)
\end{equation}
locally uniformly with respect to $x,u,v \in \RR$ where $\tilde{K}_n(x,y) = \sqrt{\mu'(x) \mu'(y)} K_n(x,y)$. 
Unlike \eqref{eq:97}, the formula \eqref{eq:98} does not give any information if $u = v = 0$. It was recently proved
in \cite{Ignjatovic2017} that under some additional regularity on $Q$
\begin{equation} \label{eq:99}
	\lim_{n \to \infty} \frac{1}{\rho_n} K_n(x,x) = \frac{1}{2 \pi \mu'(x)}
\end{equation}
locally uniformly with respect to $x \in \RR$ where
\[
	\rho_n = \sum_{j=0}^n \frac{1}{a_j}.
\]
Let us comment that $\rho_n$ is comparable to $n$ if the sequences $(a_n)$ and $(a_n^{-1})$ are bounded. By combining
\eqref{eq:98} with \eqref{eq:99} one obtains
\begin{equation} \label{eq:102}
	\lim_{n \to \infty} \frac{1}{\rho_n} K_n \Big(x + \frac{u}{\rho_n}, x + \frac{v}{\rho_n} \Big) = 
	\frac{\omega'(0)}{\mu'(x)} \sinc \big( (u-v) \pi \omega'(0) \big)
\end{equation}
where $\omega'(x) = \big( \pi \sqrt{4 - x^2} \big)^{-1}$ is the density of the equilibrium measure for the
interval $[-2, 2]$.

Instead of taking the measure $\mu$ as the starting point one can consider polynomials $(p_n : n \in \NN_0)$ satisfying
the three-term recurrence relation \eqref{eq:100} with $a_n > 0$ and $b_n \in \RR$ for any $n \in \NN_0$.
Then the Favard's theorem (see, e.g. \cite[Theorem 5.10]{Schmudgen2017}) states that there is a probability measure $\mu$
such that $(p_n)$ is orthonormal in $L^2(\RR, \mu)$. The measure $\mu$ is unique, if and only if there is exactly one
measure with the same moments as $\mu$. It is always the case when the Carleman condition
\begin{equation}
	\label{eq:103}
	\sum_{n=0}^\infty \frac{1}{a_n} = \infty
\end{equation}
is satisfied (see, e.g. \cite[Corollary 6.19]{Schmudgen2017}). Moreover, the measure $\mu$ has compact support, if and
only if the Jacobi parameters are bounded.

In this article our starting point is the three-term recurrence relation. We study analogues of \eqref{eq:99} and
\eqref{eq:102} for three different classes of Jacobi parameters: asymptotically periodic, periodically
modulated and a blend of these two; for the definitions, see Sections~\ref{sec:class_asym}, \ref{sec:class_modul} and
\ref{sec:class_blend}, respectively. The first class only gives rise to measures with compact supports. 
The second class introduced in \cite{JanasNaboko2002} has the Jacobi parameters uniformly
unbounded in the sense that $\liminf a_n = \infty$. The third class has been studied in \cite{Janas2011} as an example of
unbounded Jacobi parameters corresponding to measures with absolutely continuous parts having supports equal 
a finite union of closed intervals.

To simplify the exposition in the introduction, we shall focus on the periodic modulations only. 
Before we formulate our results, let us state some definitions. Let $N$ be a positive integer. We say that sequences 
$(a_n), (b_n)$ are \emph{$N$-periodically modulated} if there are two $N$-periodic sequences $(\alpha_n : n \in \ZZ)$
and $(\beta_n : n \in \ZZ)$ of positive and real numbers, respectively, such that
\begin{enumerate}[(a)]
	\item
	$\begin{aligned}[b]
	\lim_{n \to \infty} a_n = \infty
	\end{aligned},$
	\item
	$\begin{aligned}[b]
	\lim_{n \to \infty} \bigg| \frac{a_{n-1}}{a_n} - \frac{\alpha_{n-1}}{\alpha_n} \bigg| = 0
	\end{aligned},$
	\item
	$\begin{aligned}[b]
	\lim_{n \to \infty} \bigg| \frac{b_n}{a_n} - \frac{\beta_n}{\alpha_n} \bigg| = 0
	\end{aligned}.$
\end{enumerate}
The crucial r\^ole is played by the \emph{$N$-step transfer matrices} defined by
\[
	X_n(x) = B_{n+N-1}(x) B_{n+N-2}(x) \cdots B_n(x)
	\qquad \text{where} \qquad
	B_j(x) = 
	\begin{pmatrix}
		0 & 1 \\
		-\frac{a_{j-1}}{a_j} & \frac{x - b_j}{a_j}
	\end{pmatrix},
\]
and
\[
	\frakX_n(x) = \frakB_{n+N-1}(x) \frakB_{n+N-2} \cdots \frakB_n(x)
	\qquad\text{where}\qquad
	\frakB_j(x) = 
	\begin{pmatrix}
		0 & 1 \\
		-\frac{\alpha_{j-1}}{\alpha_j} & \frac{x - \beta_j}{\alpha_j}
	\end{pmatrix}.
\]
The name is justified by the following property
\[
	\begin{pmatrix}
		p_{n+N-1}(x) \\
		p_{n+N}(x)
	\end{pmatrix}
	=
	X_n(x)
	\begin{pmatrix}
		p_{n-1}(x) \\
		p_n(x)
	\end{pmatrix}.
\]
We are interested in the class of Jacobi matrices associated to slowly oscillating sequences introduced in
\cite{Stolz1994}. Let $r$ be a positive integer. We say that the sequence $(x_n : n \in \NN)$ of vectors 
from a normed space $V$ belongs to $\calD_r (V)$, if it is bounded and for each $j \in \{1, \ldots, r\}$,
\[
	\sum_{n = 1}^\infty \big\| \Delta^j x_n \big\|^\frac{r}{j} < \infty
\]
where
\begin{align*}
	\Delta^0 x_n &= x_n, \\
	\Delta^j x_n &= \Delta^{j-1} x_{n+1} - \Delta^{j-1} x_n, \qquad j \geq 1.
\end{align*}
If $X$ is the real line with an Euclidean norm we abbreviate 
$\calD_{r} = \calD_{r}(X)$. Given a compact set $K \subset \CC$ and a normed space $R$, by 
$\calD_{r}(K, R)$ we denote the case when $X$ is the space of all continuous mappings from $K$ to $R$ equipped with
the supremum norm.

Our first result is the following theorem, see Theorem \ref{thm:6}. By $\GL(2, \RR)$ we denote $2 \times 2$ 
real invertible matrices equipped with the spectral norm. For a matrix 
\[
	X = 
	\begin{pmatrix}
	x_{11} & x_{12} \\
	x_{21} & x_{22}
	\end{pmatrix}
\]
we set $[X]_{i, j} = x_{i, j}$. Lastly, a discriminant of $X$ is $\discr X = (\tr X)^2 - 4 \det X$.
\begin{theorem}
	\label{thm:A}
	Let $N$ and $r$ be positive integers and $i \in \{0, 1, \ldots, N-1\}$. Suppose that $K$ is a 
	compact interval with non-empty interior contained in
	\[
		\Lambda =
		\left\{
		x \in \RR :
		\lim_{j \to \infty} \discr X_{jN+i}(x) \text{ exists and is negative}
		\right\}.
	\]
	Assume that
	\[
	        \lim_{j \to \infty} \frac{a_{(j+1)N+i-1}}{a_{jN+i-1}} = 1
	\]
	and
	\[
		\big(X_{jN+i} : j \in \NN \big) \in \calD_r \big(K, \GL(2, \RR) \big).
	\]
	Suppose that $\calX$ is the limit of $\big(X_{jN+i} : j \in \NN \big)$. If
	\[
		\sum_{j=1}^\infty \frac{1}{a_{jN+i-1}} = \infty,
	\]
	then for $x \in K$
	\[
		\bigg( \sum_{j=1}^n \frac{1}{a_{jN+i-1}} \bigg)^{-1} \sum_{j=0}^n p_{jN+i}^2(x)
		= 
		\frac{|[\calX(x)]_{2, 1}|}{\pi \mu'(x) \sqrt{-\discr \calX(x)}} + E_{n}(x)
	\]
	where
	\[
		\lim_{n \to \infty} \sup_{x \in K} \big| E_{n}(x) \big| = 0.
	\]
\end{theorem}
Let us remark that in Theorem~\ref{thm:2} we have obtained the quantitative bounds on $E_n$ in the $\calD_1$ setting.

Theorem~\ref{thm:A} is an important step in proving the analogues of \eqref{eq:99} and 
\eqref{eq:102}. The following theorem (see Theorem \ref{thm:7}) provides the analogue of \eqref{eq:99} for periodic
modulations. Similar results are obtained also for the remaining classes, that is 
for asymptotically periodic in Theorem~\ref{thm:10}, and for the blend in Theorem~\ref{thm:11}.
\begin{theorem} \label{thm:B}
	Let $(a_n)$ and $(b_n)$ be $N$-periodically modulated Jacobi parameters. 
	Suppose that there is $r \geq 1$ such that for every $i \in \{0, 1, \ldots, N-1\}$,
	\[
		\bigg( \frac{a_{kN+i-1}}{a_{kN+i}} : k \in \NN\bigg),
		\bigg( \frac{b_{kN+i}}{a_{kN+i}} : k \in \NN\bigg), 
		\bigg( \frac{1}{a_{kN+i}} : k \in \NN\bigg) \in \calD_r,
	\]
	and the Carleman condition \eqref{eq:103} is satisfied. If $\abs{\tr \frakX_0(0)} < 2$, then
	\[
		\frac{1}{\rho_n} K_n(x, x) = \frac{\omega'(0)}{\mu'(x)} + E_n(x)
	\]
	where
	\[
		\lim_{n \to \infty} |E_n(x)| = 0
	\]
	locally uniformly with respect to $x \in \RR$, where $\omega$ is the equilibrium measure of
	\[
		\big\{ x \in \RR : |\tr \frakX_0(x)| \leq 2 \big\}
	\]
	and
	\[
		\rho_n = \sum_{j = 0}^n \frac{\alpha_j}{a_j}.
	\]
\end{theorem}
Again, in $\calD_1$ setup, we obtained the quantitative bound on $E_n$, see Theorem~\ref{thm:8} (periodic modulations) and
Theorem~\ref{thm:9} (asymptotically periodic). 

We emphasize that Theorem~\ref{thm:B} solves \cite[Conjecture 1]{Ignjatovic2016} for a larger class of Jacobi parameters
than it was originally stated, see Section~\ref{sec:ign_conj} for details.

Lastly, we provide the analogue of \eqref{eq:102} for periodic modulations (see Theorem \ref{thm:1}).
In view of \cite[Corollary 7]{SwiderskiTrojan2019}, the asymptotically periodic case follows from \cite{Totik2009}.
For the blend, see Theorem~\ref{thm:15}.
\begin{theorem}
	\label{thm:C}
	Suppose that the hypotheses of Theorem~\ref{thm:B} are satisfied for $r=1$. Then
	\[
		\lim_{n \to \infty} \frac{1}{\rho_n} K_n \Big( x + \frac{u}{\rho_n}, x + \frac{v}{\rho_n} \Big)
		=
		\frac{\omega'(0)}{\mu'(x)} \sinc \big( (u-v) \pi \omega'(0) \big)
	\]
	locally uniformly with respect to $x,u,v \in \RR$.
\end{theorem}
Let us mention that the hypotheses of Theorem \ref{thm:B} and Theorem~\ref{thm:C} for $N=1$ are satisfied by Hermite
polynomials, Meixner--Pollaczek polynomials and Freud weights, see \cite[Section 5]{PeriodicII} for detailed proofs.

Let us present some ideas of the proofs. The basic strategy commonly used is to exploit the Christoffel--Darboux formula,
that is
\[
	K_n(x, y) = 
	a_n 
	\left\{
	\begin{aligned}
		\frac{p_{n+1}(x) p_n(y) - p_{n}(x) p_{n+1}(y)}{x-y}, &\qquad\text{if } x \neq y, \\
		p_n(x) p_{n+1}'(x) - p_n'(x) p_{n+1}(x), &\qquad\text{otherwise.}
	\end{aligned}
	\right.
\]
However, it requires the precise asymptotic of the polynomials as well as its derivatives in terms of both $n$ and $x$.
Unfortunately, for the classes of Jacobi parameters we are interested in they are not available. In the recent article
\cite{SwiderskiTrojan2019}, we managed to obtain the asymptotic of $(p_n(x) : n \in \NN_0)$ locally uniformly with respect
to $x$. Based on it we develop a method to study $K_n(x, y)$. Namely, we use the formula \eqref{eq:104},
which leads to the need of estimation of the oscillatory sums of a form
\[
	\sum_{k=0}^n \frac{\gamma_k}{\sum_{j=0}^n \gamma_j}
	\sin \Big( \sum_{j=0}^n \theta_j(x_n) + \sigma(x_n) \Big)
	\sin \Big( \sum_{j=0}^n \theta_j(y_n) + \sigma(y_n) \Big)
\]
where 
\[
	x_n = x + \frac{u}{\rho_n}, \qquad\text{and}\qquad
	y_n = x + \frac{v}{\rho_n}.
\]
To deal with the sums we prove two auxiliary results (see Lemma \ref{lem:9} and Theorem \ref{thm:3}) that are 
valid for sequences not necessarily belonging to $\calD_r$.

The organization of the article is as follows. In Section~\ref{sec:def} we present basic definitions used
in the article. In Section~\ref{sec:classes} we collected the definitions and basic properties of
the three classes of sequences. Section~\ref{sec:christoffel} is devoted to the general $\calD_r$ setting.
In particular, we present there the proofs of Theorem \ref{thm:A} and \ref{thm:B}, and we provide a
solution of Ignjatović conjecture. In Section~\ref{sec:CD_for_D1} we study the case $\calD_1$ where 
we derive quantitative bound on the error in the asymptotic of the polynomials. We also provide the quantitative
versions of Theorems \ref{thm:A} and \ref{thm:B}. Finally, we prove Theorem~\ref{thm:C}.

\subsection*{Notation}
By $\NN$ we denote the set of positive integers and $\NN_0 = \NN \cup \{0\}$. Throughout the whole article, we write $A \lesssim B$ if there is an absolute constant $c>0$ such that
$A\le cB$. Moreover, $c$ stands for a positive constant whose value may vary from occurrence to occurrence.

\subsection*{Acknowledgment}
The first author was partially supported by the Foundation for Polish Science (FNP) and by long term structural funding -- Methusalem grant of the Flemish Government.

\section{Definitions} \label{sec:def}
Given two sequences $a = (a_n : n \in \NN_0)$ and $b = (b_n : n \in \NN_0)$ of positive and real numbers, respectively,
and $k \in \NN$, we define $k$th associated \emph{orthonormal} polynomials as
\[
	\begin{gathered}
		p^{[k]}_0(x) = 1, \qquad p^{[k]}_1(x) = \frac{x - b_k}{a_k}, \\
		a_{n+k-1} p^{[k]}_{n-1}(x) + b_{n+k} p^{[k]}_n(x) + a_{n+k} p^{[k]}_{n+1}(x) = 
			x p^{[k]}_n(x), \quad (n \geq 1),
	\end{gathered}
\]
For $k=0$ we usually omit the superscript. A sequence $(u_n : n \in \NN_0)$ is a generalized eigenvector associated to
$x \in \CC$, if for all $n \geq 1$,
\[
	\begin{pmatrix}
		u_n \\
		u_{n+1}
	\end{pmatrix} 
	=
	B_n(x) 
	\begin{pmatrix}
		u_{n-1} \\
		u_{n}
	\end{pmatrix}
\]
where
\[
	B_n(x) = 
	\begin{pmatrix}
		0 & 1 \\
		-\frac{a_{n-1}}{a_n} & \frac{x - b_n}{a_n}
	\end{pmatrix}.
\]
Let $A$ be the closure in $\ell^2$ of the operator acting on sequences having finite support by the matrix
\begin{equation}
	\label{eq:5}
	\begin{pmatrix}
		b_0 & a_0 & 0   & 0      &\ldots \\
		a_0 & b_1 & a_1 & 0       & \ldots \\
		0   & a_1 & b_2 & a_2     & \ldots \\
		0   & 0   & a_2 & b_3   &  \\
		\vdots & \vdots & \vdots  &  & \ddots
	\end{pmatrix}.
\end{equation}
The operator $A$ is called \emph{Jacobi matrix}. If the Carleman condition
\begin{equation}
	\label{eq:52}
	\sum_{n = 0}^\infty \frac{1}{a_n} = \infty
\end{equation}
is satisfied then the operator $A$ has the unique self-adjoint extension (see e.g. \cite[Corollary 6.19]{Schmudgen2017}).
Let us denote by $E_A$ its spectral resolution of the identity. Then for any Borel subset $B \subset \RR$, we set
\[
	\mu(B) = \langle E_A(B) \delta_0, \delta_0 \rangle_{\ell^2}
\]
where $\delta_0$ is the sequence having $1$ on the $0$th position and $0$ elsewhere. The polynomials $(p_n : n \in \NN_0)$
form an orthonormal basis of $L^2(\RR, \mu)$. 

In this article the central object is the \emph{Christoffel--Darboux kernel} defined as
\begin{equation}
	\label{eq:37}
	K_n(x, y) = \sum_{j=0}^n p_j(x) p_j(y).
\end{equation}
Given a compact set $K \subset \RR$ with non-empty interior, there is the unique probability measure 
$\omega_K$, called \emph{the equilibrium measure} corresponding to $K$, minimizing the energy
\begin{equation} \label{eq:96}
	I(\nu) = -\int_\RR \int_\RR \log{\abs{x-y}} \nu({\rm d} x) \nu({\rm d} y),
\end{equation}
among all probability measures $\nu$ supported on $K$. The measure $\omega_K$ is absolutely continuous 
in the interior of $K$ with continuous density, see \cite[Theorem IV.2.5, pp. 216]{Saff1997}.

\section{Classes of sequences} \label{sec:classes}
In this article we are interested in Jacobi matrices having entries in one of the three classes defined in terms of
periodic sequences. Let us start by fixing some notation. 

By $(\alpha_n : n \in \ZZ)$ and $(\beta_n : n \in \ZZ)$ we denote
$N$-periodic sequences of real and positive numbers, respectively. For each $k \geq 0$, let us define polynomials
$(\mathfrak{p}^{[k]}_n : n \in \NN_0)$ by relations
\[
	\begin{gathered}
		\mathfrak{p}_0^{[k]}(x) = 1, \qquad \mathfrak{p}_1^{[k]}(x) = \frac{x-\beta_k}{\alpha_k}, \\
		\alpha_{n+k-1} \mathfrak{p}^{[k]}_{n-1}(x) + \beta_{n+k} \mathfrak{p}^{[k]}_n(x) 
		+ \alpha_{n+k} \mathfrak{p}^{[k]}_{n+1}(x)
		= x \mathfrak{p}^{[k]}_n(x), \qquad n \geq 1.
	\end{gathered}
\]
Let
\[
	\frakB_n(x) = 
	\begin{pmatrix}
		0 & 1 \\
		-\frac{\alpha_{n-1}}{\alpha_n} & \frac{x - \beta_n}{\alpha_n}
	\end{pmatrix},
	\qquad\text{and}\qquad
	\frakX_n(x) = \prod_{j = n}^{N+n-1} \mathfrak{B}_j(x), \qquad n \in \ZZ
\]
where for a sequence of square matrices $(C_n : n_0 \leq n \leq n_1)$ we set
\[
	\prod_{k = n_0}^{n_1} C_k = C_{n_1} C_{n_1-1} \cdots C_{n_0}.
\]
By $\frakA$ we denote the Jacobi matrix corresponding to 
\begin{equation*}
	\begin{pmatrix}
		\beta_0 & \alpha_0 & 0   & 0      &\ldots \\
		\alpha_0 & \beta_1 & \alpha_1 & 0       & \ldots \\
		0   & \alpha_1 & \beta_2 & \alpha_2     & \ldots \\
		0   & 0   & \alpha_2 & \beta_3   &  \\
		\vdots & \vdots & \vdots  &  & \ddots
	\end{pmatrix}.
\end{equation*}

Let $\omega$ be the equilibrium measure corresponding to $\sigma_{\textrm{ess}}(\frakA)$. Since
$\sigma_{\textrm{ess}}(\frakA)$ is a finite union of compact intervals with non-empty interiors 
(see e.g. \cite[Theorem 5.2.4 and Theorem 5.4.2]{Simon2010Book}), $\omega$ is absolutely continuous. Moreover,
for $x$ in the interior of $\sigma_{\textrm{ess}}(\frakA)$ we obtain
\begin{equation}
	\label{eq:4}
	\omega'(x) = 
	\frac{1}{N} 
	\sum_{i = 0}^{N-1} 
	\frac{|[\frakX_i(x) ]_{2,1}|}{\pi \sqrt{-\discr \frakX_i(x)}}
	\cdot \frac{1}{\alpha_{i-1}}.
\end{equation}
Indeed, by \cite[Proposition 3]{PeriodicIII} (see also Lemma \ref{lem:4} below),
\[
	[\frakX_i(x)]_{2,1} = -\frac{\alpha_{i-1}}{\alpha_i} \mathfrak{p}^{[i+1]}_{N-1}(x).
\]
Since
\[
	\frakX_{i+1}(x) = \big(\mathfrak{B}_i(x) \big)\big(\frakX_i(x) \big)\big(\mathfrak{B}_i(x)\big)^{-1},
\]
we have
\[
	\discr \frakX_i(x) = \discr \frakX_0(x).
\]
Therefore,
\[
	\frac{1}{N}
	\sum_{i = 0}^{N-1}
	\frac{|[\frakX_i(x)]_{2,1}|}{\pi \sqrt{-\discr \frakX_i(x)}}
	\cdot \frac{1}{\alpha_{i-1}}
	=
	\frac{1}{\pi \sqrt{-\discr \frakX_0(x)}}
	\frac{1}{N}
	\sum_{i = 0}^{N-1}
	\frac{\big|\mathfrak{p}^{[i+1]}_{N-1}(x)\big|}{\alpha_i}.
\]
Let us recall that the first formula on page 214 of \cite{VanAssche1993} reads 
\[
	\sum_{i = 0}^{N-1} \frac{\big|\mathfrak{p}^{[i+1]}_{N-1}(x)\big|}{\alpha_i} =
	\big|\tr \frakX_0'(x) \big|,
\]
hence
\begin{equation}
	\label{eq:82}
	\frac{1}{N}
	\sum_{i = 0}^{N-1}
	\frac{|[\frakX_i(x)]_{2,1}|}{\pi \sqrt{-\discr \frakX_i(x)}}
	\cdot \frac{1}{\alpha_{i-1}}
	=
	\frac{\big|\tr \frakX_0'(x) \big|}{\pi N \sqrt{-\discr \frakX_0(x)}}.
\end{equation}
Now, in view of \cite[the formula (3.2)]{VanAssche1993} we obtain \eqref{eq:4}.

\subsection{Asymptotically periodic} \label{sec:class_asym}
\begin{definition}
	The Jacobi matrix $A$ associated to $(a_n : n \in \NN_0)$ and $(b_n : n \in \NN_0)$ has 
	\emph{asymptotically $N$-periodic entries}, if there are two $N$-periodic sequences
	$(\alpha_n : n \in \ZZ)$ and $(\beta_n : n \in \ZZ)$ of positive and real numbers, respectively, such that
	\begin{enumerate}[(a)]
		\item 
		$\begin{aligned}[b]
			\lim_{n \to \infty} \big|a_n - \alpha_n\big| = 0
		\end{aligned}$,
		\item
		$\begin{aligned}[b]
			\lim_{n \to \infty} \big|b_n - \beta_n\big| = 0
		\end{aligned}$.
	\end{enumerate}
\end{definition}
Let us recall the following lemma.
\begin{lemma}[{\cite[Proposition 3]{PeriodicIII}}]
\label{lem:4}
	Let $(p_n : n \in \NN_0)$ be a sequence of orthonormal polynomials associated to $(a_n : n \in \NN_0)$ and
	$(b_n : n \in \NN_0)$. Then for all $n \geq 1$ and $k \geq 1$,
	\[
		\prod_{j=k}^{n+k-1} B_j(x) =
		\begin{pmatrix}
			-\frac{a_{k-1}}{a_k} p_{n-2}^{[k+1]}(x) & p_{n-1}^{[k]}(x) \\
			-\frac{a_{k-1}}{a_k} p_{n-1}^{[k+1]}(x) & p_{n}^{[k]}(x)
		\end{pmatrix}.
	\]
\end{lemma}

\begin{proposition} 
	\label{prop:1}
	Suppose that $A$ has asymptotically $N$-periodic entries. Then for each $i \in \{ 0, 1, \ldots, N-1 \}$ and
	$n \geq 0$,
	\[
		\lim_{k \to \infty} \bigg\| \prod_{j=kN+i}^{kN+i+n} B_j(x) - 
		\prod_{j=i}^{i+n} \mathfrak{B}_j(x) \bigg\| = 0,
	\]
	locally uniformly with respect to $x \in \CC$.
\end{proposition}
\begin{proof}
	Since for each $i \in \{ 0, 1, \ldots, N-1 \}$, we have
	\[
		\lim_{k \to \infty} \big\| B_{kN+i}(x) - \mathfrak{B}_i(x) \big\| = 0
	\]
	uniformly on compact subsets of $\CC$, the conclusion follows by the continuity of $B_n$.
\end{proof}

For a Jacobi matrix having asymptotically $N$-periodic entries, we set
\[
	X_n(x) = \prod_{j = n}^{N+n-1} B_j(x).
\]
Let us denote by $\calX_i$ the limit of $(X_{jN+i} : j \in \NN)$. Then, by Proposition \ref{prop:1}, we conclude that
$\calX_i = \frakX_i$ for all $i \in \{0, 1, \ldots, N-1\}$.
\begin{proposition}
	\label{prop:2}
	Suppose that a Jacobi matrix $A$ has asymptotically $N$-periodic entries. Then for each $i \in \{0, 1, \ldots, N-1 \}$,
	and $n \geq 0$, 
	\begin{subequations} 
		\begin{align}
		\label{eq:60a}
		\lim_{k \to \infty} 
		p_n^{[kN+i]}(x) 
		&= \mathfrak{p}_{n}^{[i]}(x),\\
		\label{eq:60b}
		\lim_{k \to \infty} 
		\frac{a_{kN+i}}{\alpha_i} \big( p_n^{[kN+i]} \big)'(x) 
		&= \big( \mathfrak{p}_{n}^{[i]} \big)'(x),\\ 
		\label{eq:60c}
		\lim_{k \to \infty} 
		\bigg(\frac{a_{kN+i}}{\alpha_i} \bigg)^2 \big( p_n^{[kN+i]} \big)''(x) 
		&= \big( \mathfrak{p}_{n}^{[i]} \big)''(x),
		\end{align}
	\end{subequations}
	locally uniformly with respect to $x \in \CC$.
\end{proposition}
\begin{proof}
	Lemma \ref{lem:4} together with Proposition \ref{prop:1} easily gives \eqref{eq:60a}. Since
	\[
		\lim_{k \to \infty} \frac{a_{kN+i}}{\alpha_i} = 1,
	\]
	the uniform convergence in \eqref{eq:60a} entails \eqref{eq:60b} and \eqref{eq:60c}.
\end{proof}

\begin{corollary} 
	\label{cor:4}
	Suppose that a Jacobi matrix $A$ has asymptotically $N$-periodic entries. Then for each $i \in \{0, 1, \ldots, N-1\}$,
	\begin{subequations} 
		\begin{align}
		\label{eq:61a}
		\lim_{k \to \infty} \tr \big( X_{kN+i} (x) \big)
		&= \tr \big( \frakX_0 (x) \big), \\
		\label{eq:61b}
		\lim_{k \to \infty} \frac{a_{kN+i}}{\alpha_i} \tr \big( X_{kN+i}'(x) \big)
		&= \tr \big( \frakX_0'(x) \big), \\
		\label{eq:61c}
		\lim_{k \to \infty} \bigg( \frac{a_{kN+i}}{\alpha_i} \bigg)^2 \tr \big( X_{kN+i}''(x) \big) 
		&= \tr \big( \frakX_0''(x) \big), 
		\end{align}
	\end{subequations}
	locally uniformly with respect to $x \in \CC$.
\end{corollary}
\begin{proof}
	Since for $x \in \CC$,
	\[
		\frakX_{i+1}(x) = 
		\big( \mathfrak{B}_i (x) \big)
		\big( \frakX_i(x) \big)
		\big( \mathfrak{B}_i(x)\big)^{-1},
	\]
	we have
	\begin{equation} 
		\label{eq:24}
		\tr\big(\frakX_{i}^{(s)}(x)\big) = 
		\tr\big(\frakX_0^{(s)}(x)\big), \qquad\text{for all }s \geq 0.
	\end{equation}
	By Lemma \ref{lem:4},
	\[
		\tr\big(X_{kN+i}(x)\big) = p_{N}^{[kN+i]}(x) -\frac{a_{kN+i-1}}{a_{kN+i}} p_{N-2}^{[kN+i+1]}(x),
	\]
	which together with Proposition \ref{prop:2}, implies that
	\[
		\lim_{k \to \infty} \tr\big( X_{kN+i}(x) \big) = \tr \big( \frakX_{i}(x)\big)
	\]
	uniformly on compact subsets of $\CC$. Hence, by \eqref{eq:24} we easily obtain \eqref{eq:61a}.

	For $s \in \{1, 2 \}$ we write
	\begin{align*}
		&
		\frac{a_{kN+i-1}}{a_{kN+i}} \Big( \frac{a_{kN+i}}{\alpha_i} \Big)^s 
		\big( p_{N-2}^{[kN+i+1]} \big)^{(s)}(x) \\
		&\qquad\qquad= 
		\frac{a_{kN+i-1}}{a_{kN+i}} 
		\Big( \frac{a_{kN+i}}{\alpha_i} \Big)^s \Big( \frac{\alpha_{i+1}}{a_{kN+i+1}} \Big)^s 
		\Big( \frac{a_{kN+i+1}}{\alpha_{i+1}} \Big)^s \big( p_{N-2}^{[kN+i+1]} \big)^{(s)}(x),
	\end{align*}
	hence, by Proposition \ref{prop:2}, we obtain
	\[
		\lim_{k \to \infty} 
		\frac{a_{kN+i-1}}{a_{kN+i}} \Big( \frac{a_{kN+i}}{\alpha_i} \Big)^s \big( p_{N-2}^{[kN+i+1]} \big)^{(s)}(x)
		=
		\frac{\alpha_{i-1}}{\alpha_i} \big( \mathfrak{p}_{N-2}^{[i+1]} \big)^{(s)}(x).
	\]
	Therefore, by Lemma \ref{lem:4},
	\begin{align*}
		\lim_{k \to \infty} \Big( \frac{a_{kN+i}}{\alpha_i} \Big)^s 
		\tr \big( X_{kN+i}^{(s)}(x) \big) 
		&= 
		\big( \mathfrak{p}_{N}^{[i]} \big)^{(s)}(x) - 
		\frac{\alpha_{i-1}}{\alpha_i} \big( \mathfrak{p}_{N-2}^{[i+1]} \big)^{(s)}(x) \\
		&= \tr \big(\frakX_i^{(s)}(x)\big),
	\end{align*}
	which finishes the proof.
\end{proof}

\subsection{Periodic modulations} \label{sec:class_modul}
\begin{definition}
	\label{def:2}
	We say that the Jacobi matrix $A$ associated to $(a_n : n \in \NN_0)$ and $(b_n : n \in \NN_0)$ has 
	\emph{$N$-periodically modulated entries,} if there are two $N$-periodic sequences
	$(\alpha_n : n \in \ZZ)$ and $(\beta_n : n \in \ZZ)$ of positive and real numbers, respectively, such that
	\begin{enumerate}[(a)]
	\item
	$\begin{aligned}[b]
	\lim_{n \to \infty} a_n = \infty
	\end{aligned},$
	\item
	$\begin{aligned}[b]
	\lim_{n \to \infty} \bigg| \frac{a_{n-1}}{a_n} - \frac{\alpha_{n-1}}{\alpha_n} \bigg| = 0
	\end{aligned},$
	\item
	$\begin{aligned}[b]
	\lim_{n \to \infty} \bigg| \frac{b_n}{a_n} - \frac{\beta_n}{\alpha_n} \bigg| = 0
	\end{aligned}.$
	\end{enumerate}
\end{definition}
Suppose that $A$ is a Jacobi matrix with $N$-periodically modulated entries. 
Observe that, by setting 
\[
	\tilde{a}_n = \frac{a_n}{\alpha_n}, \qquad\text{and}\qquad \tilde{b}_n = \frac{b_n}{\alpha_n},
\]
we obtain
\[
	\lim_{n \to \infty} \frac{\tilde{a}_{n-1}}{\tilde{a}_n} = 1, 
	\qquad\text{and}\qquad
	\frac{\tilde{b}_n}{\tilde{a}_n} = \frac{b_n}{a_n}. 
\]
Hence, $A$ is $N$-periodic modulation of the Jacobi matrix corresponding to the sequences $(\tilde{a}_n : n \in \NN_0)$
$(\tilde{b}_n : n \in \NN_0)$ in the usual sense.

\begin{proposition} 
	\label{prop:5}
	If a Jacobi matrix $A$ has $N$-periodically modulated entries and $i \in \NN$, then
	\[
		\lim_{n \to \infty} \frac{\alpha_{n+i}}{\alpha_n}\frac{a_{n}}{a_{n+i}} = 1.
	\]
	In particular,
	\[
		\lim_{n \to \infty} \frac{a_{n}}{a_{n+N}} = 1.
	\]
\end{proposition}
\begin{proof}
	Since
	\[
		\lim_{n \to \infty} \bigg| \frac{a_{n}}{a_{n+1}} - \frac{\alpha_{n}}{\alpha_{n+1}} \bigg| = 0,
	\]
	one has
	\[
		\lim_{n \to \infty}
		\bigg|
		\frac{a_n}{a_{n+i}} - \frac{\alpha_n}{\alpha_{n+i}}
		\bigg|
		=
		\lim_{n \to \infty}
		\bigg|
		\prod_{k=0}^{i-1} \frac{a_{n+k}}{a_{n+k+1}} - 
		\prod_{k=0}^{i-1} \frac{\alpha_{n+k}}{\alpha_{n+k+1}}
		\bigg| = 0.
	\]
	Hence, for some $c > 0$,
	\begin{align*}
		\lim_{n \to \infty} 
		\bigg| \frac{a_{n}}{a_{n+i}} \frac{\alpha_{n+i}}{\alpha_{n}} - 1 \bigg| &= 
		\lim_{n \to \infty} 
		\frac{\alpha_{n+i}}{\alpha_{n}} \bigg| \frac{a_{n}}{a_{n+i}} - \frac{\alpha_{n}}{\alpha_{n+i}} \bigg| \\
		&\leq c
		\lim_{n \to \infty}
		\bigg| \frac{a_{n}}{a_{n+i}} - \frac{\alpha_{n}}{\alpha_{n+i}} \bigg| = 0,
	\end{align*}
	and the proposition follows.
\end{proof}

\begin{proposition} 
	\label{prop:4}
	Suppose that $A$ has $N$-periodically modulated entries. Then for each $i \in \{ 0, 1, \ldots, N-1 \}$ and
	$n \geq 0$,
	\[
		\lim_{k \to \infty} \bigg\| \prod_{j=kN+i}^{kN+i+n} B_j(x) - 
		\prod_{j=i}^{i+n} \mathfrak{B}_j(0) \bigg\| = 0,
	\]
	locally uniformly with respect to $x \in \CC$.
\end{proposition}
\begin{proof}
	Since for each $i \in \{ 0, 1, \ldots, N-1 \}$, we have
	\[
		\lim_{k \to \infty} \big\| B_{kN+i}(x) - \mathfrak{B}_i(0) \big\| = 0
	\]
	uniformly on compact subsets of $\CC$, the conclusion follows by the continuity of $B_n$.
\end{proof}
For a Jacobi matrix with $N$-periodically modulated entries, we set
\[
	X_n(x) = \prod_{j = n}^{N+n-1} B_j(x).
\]
Let us denote by $\calX_i$ the limit of $(X_{jN+i} : j \in \NN)$. Then, by Proposition \ref{prop:4}, we have
$\calX_i(x) = \frakX_i(0)$ for all $i \in \{0, 1, \ldots, N-1\}$ and $x \in \CC$.
\begin{proposition}
	\label{prop:3}
	Suppose that a Jacobi matrix $A$ has $N$-periodically modulated entries. Then for each $i \in \{0, 1, \ldots, N-1 \}$,
	and $n \geq 0$, 
	\begin{subequations} 
		\begin{align}
		\label{eq:10a}
		\lim_{k \to \infty} 
		p_n^{[kN+i]}(x) 
		&= \mathfrak{p}_{n}^{[i]}(0),\\
		\label{eq:10b}
		\lim_{k \to \infty} 
		\frac{a_{kN+i}}{\alpha_i} \big( p_n^{[kN+i]} \big)'(x) 
		&= \big( \mathfrak{p}_{n}^{[i]} \big)'(0),\\ 
		\label{eq:10c}
		\lim_{k \to \infty} 
		\bigg(\frac{a_{kN+i}}{\alpha_i} \bigg)^2 \big( p_n^{[kN+i]} \big)''(x) 
		&= \big( \mathfrak{p}_{n}^{[i]} \big)''(0),
		\end{align}
	\end{subequations}
	locally uniformly with respect to $x \in \CC$.
\end{proposition}
\begin{proof}
	By Lemma \ref{lem:4} and Proposition \ref{prop:4} we obtain that for every $i \geq 0$ and $n \geq 0$,
	\begin{equation} 
		\label{eq:14}
		\lim_{k \to \infty} p^{[kN+i]}_n(x) = \mathfrak{p}_{n}^{[i]}(0)
	\end{equation}
	uniformly on compact subsets of $\CC$, which is \eqref{eq:10a}. Next, let us recall that
	(see e.g. \cite[Proposition 2]{PeriodicIII})
	\begin{equation}
		\label{eq:62}
		p_n'(x) = \frac{1}{a_0} 
		\sum_{m=0}^{n-1} \left( p_m(x) p_{n-1}^{[1]}(x) - p_n(x) p_{m-1}^{[1]}(x) \right) p_m(x),
	\end{equation}
	therefore for every $n \in \NN$,
	\begin{align*}
		&
		\big( p_n^{[kN+i]} \big)'(x) \\
		&\qquad=  
		\frac{1}{a_{kN+i}} \sum_{m=0}^{n-1} 
		\left( p_m^{[kN+i]}(x) p_{n-1}^{[kN+i+1]}(x) - 
		p_n^{[kN+i]}(x) p_{m-1}^{[kN+i+1]}(x) \right) p_m^{[kN+i]}(x).
	\end{align*}
	Hence,
	\begin{equation} 
		\label{eq:9}
		\begin{aligned}
		&
		\frac{a_{kN+i}}{\alpha_i} \big( p_n^{[kN+i]} \big)'(x) \\
		&\qquad=
		\frac{1}{\alpha_i} \sum_{m=0}^{n-1} 
		\left( p_m^{[kN+i]}(x) p_{n-1}^{[kN+i+1]}(x) - 
		p_n^{[kN+i]}(x) p_{m-1}^{[kN+i+1]}(x) \right)
		p_m^{[kN+i]}(x).
		\end{aligned}
	\end{equation}
	By \eqref{eq:14} the right-hand side of \eqref{eq:9} tends to
	\[
		\frac{1}{\alpha_i} \sum_{m=0}^{n-1} 
		\left( \mathfrak{p}_m^{[i]}(0) \mathfrak{p}_{n-1}^{[i+1]}(0) - 
		\mathfrak{p}_n^{[i]}(0) \mathfrak{p}_{m-1}^{[i+1]}(0) \right) \mathfrak{p}_m^{[i]}(0).
	\]
	Therefore, by \eqref{eq:62} we conclude \eqref{eq:10b}. For the proof of \eqref{eq:10c}, let us observe that
	from \eqref{eq:9} we get
	\begin{equation} 
		\label{eq:12}
		\begin{aligned}
		&
		\Big( \frac{a_{kN+i}}{\alpha_i} \Big)^2 \big(p_{n}^{[kN+i]}\big)''(x) \\
		&\quad=
		\frac{1}{\alpha_i} 
		\sum_{m=0}^{n-1} 
		\left( p_m^{[kN+i]}(x) p_{n-1}^{[kN+i+1]}(x) - 
		p_n^{[kN+i]}(x) p_{m-1}^{[kN+i+1]}(x) \right)
		\frac{a_{kN+i}}{\alpha_i} \big( p_{m}^{[kN+i]} \big)'(x) \\
		&\quad\phantom{=}+
		\frac{1}{\alpha_i} \sum_{m=0}^{n-1} 
		\frac{a_{kN+i}}{\alpha_i} \left( p_m^{[kN+i]} p_{n-1}^{[kN+i+1]} -
		p_n^{[kN+i]} p_{m-1}^{[kN+i+1]} \right)'(x) 
		p_{m}^{[kN+i]}(x).
		\end{aligned}
	\end{equation}
	Therefore, by \eqref{eq:10b} and \eqref{eq:14}, the first sum in \eqref{eq:12} approaches to
	\begin{equation} 
		\label{eq:20}
		\frac{1}{\alpha_i} \sum_{m=0}^{n-1} 
		\left(\mathfrak{p}_m^{[i]}(0) \mathfrak{p}_{n-1}^{[i+1]}(0) - 
		\mathfrak{p}_n^{[i]}(0) \mathfrak{p}_{m-1}^{[i+1]}(0) \right) 
		\big( \mathfrak{p}_{m}^{[i]} \big)'(0).
	\end{equation}
	In view of \eqref{eq:10b}, for each $m \in \NN_0$,
	\begin{align}
		\nonumber
		\lim_{k \to \infty} \frac{a_{kN+i}}{\alpha_i} \big( p_{m}^{[kN+i+1]} \big)'(x) 
		&=
		\lim_{k \to \infty} 
		\frac{a_{kN+i}}{\alpha_i} \frac{\alpha_{i+1}}{a_{kN+i+1}} 
		\frac{a_{kN+i+1}}{\alpha_{i+1}} \big( p_{m}^{[kN+i+1]} \big)'(x) \\
		\label{eq:36}
		&=
		\big(\mathfrak{p}_m^{[i+1]} \big)'(0).
	\end{align}
	Hence, by \eqref{eq:10b}, \eqref{eq:36} and \eqref{eq:14}, we obtain
	\begin{align*}
		&
		\lim_{k \to \infty}
		\frac{a_{kN+i}}{\alpha_i} 
		\left( p_m^{[kN+i]} p_{n-1}^{[kN+i+1]} - 
		p_n^{[kN+i]} p_{m-1}^{[kN+i+1]} \right)'(x) \\
		&\qquad\qquad
		=
		\left(
		\mathfrak{p}_m^{[kN+i]} \mathfrak{p}_{n-1}^{[kN+i+1]} -
		\mathfrak{p}_n^{[kN+i]} \mathfrak{p}_{m-1}^{[kN+i+1]}
		\right)'(0).
	\end{align*}
	Consequently, the second sum in \eqref{eq:12} tends to
	\begin{equation} 
		\label{eq:21}
		\frac{1}{\alpha_i} 
		\sum_{m=0}^{n-1} 
		\left( \mathfrak{p}_m^{[i]} \mathfrak{p}_{n-1}^{[i+1]} - 
		\mathfrak{p}_n^{[i]} \mathfrak{p}_{m-1}^{[i+1]} \right)'(0) 
		\mathfrak{p}_{m}^{[i]}(0).
	\end{equation}
	Finally, putting \eqref{eq:20} and \eqref{eq:21} into \eqref{eq:12}, by \eqref{eq:62}, we obtain
	\eqref{eq:10c}. This completes the proof.
\end{proof}

By reasoning analogous to the proof of Corollary \ref{cor:4} we obtain the following corollary.
\begin{corollary}
	\label{cor:5}
	Suppose that a Jacobi matrix $A$ has $N$-periodically modulated entries. Then for each $i \in \{0, 1, \ldots, N-1\}$,
	\begin{subequations} 
		\begin{align}
		\label{eq:16a}
		\lim_{k \to \infty} \tr \big( X_{kN+i} (x) \big)
		&= \tr \big( \frakX_0(0) \big), \\
		\label{eq:16b}
		\lim_{k \to \infty} \frac{a_{kN+i}}{\alpha_i} \tr \big( X_{kN+i}'(x) \big)
		&= \tr \big( \frakX_0'(0) \big), \\
		\label{eq:16c}
		\lim_{k \to \infty} \bigg( \frac{a_{kN+i}}{\alpha_i} \bigg)^2 \tr \big( X_{kN+i}''(x) \big) 
		&= \tr \big( \frakX_0''(0) \big), 
		\end{align}
	\end{subequations}
	locally uniformly with respect to $x \in \CC$.
\end{corollary}

\subsection{A blend of bounded and unbounded parameters} \label{sec:class_blend}
\begin{definition}
	\label{def:3}
	The Jacobi matrix $A$ associated with sequences $(a_n : n \in \NN_0)$ and $(b_n : n \in \NN_0)$ is 
	a \emph{$N$-periodic blend} if there are an asymptotically $N$-periodic Jacobi matrix $\tilde{A}$ associated with
	sequences $(\tilde{a}_n : n \in \NN_0)$ and $(\tilde{b}_n : n \in \NN_0)$, and a sequence of positive
	numbers $(\tilde{c}_n : n \in \NN_0)$, such that
	\begin{enumerate}
		\item
		$\begin{aligned}[b]
			\lim_{n \to \infty} \tilde{c}_n = \infty, \qquad\text{and}\qquad
			\lim_{m \to \infty} \frac{\tilde{c}_{2m+1}}{\tilde{c}_{2m}} = 1
		\end{aligned}$,
		\item
		$\begin{aligned}[b]
			a_{k(N+2)+i} = 
			\begin{cases}
				\tilde{a}_{kN+i} & \text{if } i \in \{0, 1, \ldots, N-1\}, \\
				\tilde{c}_{2k} & \text{if } i = N, \\
				\tilde{c}_{2k+1} & \text{if } i = N+1,
			\end{cases}
		\end{aligned}$
		\item
		$\begin{aligned}[b]
			b_{k(N+2)+i} = 
			\begin{cases}
				\tilde{b}_{kN+i} & \text{if } i \in \{0, 1, \ldots, N-1\}, \\
				0 & \text{if } i \in \{N, N+1\}.
			\end{cases}
		\end{aligned}$
	\end{enumerate}
\end{definition}

\begin{proposition}
	\label{prop:7}
	For $i \in \{1, 2, \ldots,N-1\}$,
	\[
		\lim_{j \to \infty} B_{j(N+2)+i} (x) = \mathfrak{B}_i(x)
	\]
	locally uniformly with respect to $x \in \RR$. Moreover,
	\[
		\lim_{j \to \infty} B_{((j+1)(N+2)}(x) B_{j(N+2)+N+1}(x) B_{(j(N+2)+N}(x)
		=
		\calC(x),
	\]
	locally uniformly with respect to $x \in \RR$, where
	\[
		\calC(x) =
		\begin{pmatrix}
			0 & -1 \\
			\frac{\alpha_{N-1}}{\alpha_0} & -\frac{2x - \beta_0}{\alpha_0}
		\end{pmatrix}.
	\]
\end{proposition}
\begin{proof}
	The argument is contained in the proof of \cite[Corollary 9]{SwiderskiTrojan2019}.
\end{proof}
For a Jacobi matrix being $N$-periodic blend, we set
\[
	X_n(x) = \prod_{j = n}^{N+n+1} B_j(x).
\]
In view of Proposition \ref{prop:7}, for $i \in \{1, 2, \ldots, N\}$, the sequence $(X_{j(N+2)+i} : j \in \NN)$
converges to $\calX_i$ locally uniformly on $\CC$ where
\begin{equation}
	\label{eq:69}
	\calX_i(x) =
	\bigg( \prod_{j = 1}^{i-1} \frakB_j(x) \bigg) \calC(x)
	\bigg( \prod_{j = i}^{N-1} \frakB_j(x) \bigg).
\end{equation}
We set
\[
	\Lambda = \big\{x \in \RR : \big|\tr \calX_1(x) \big| < 2 \big\}.
\]
\begin{theorem}
	\label{thm:4}
	There are non-empty open and disjoint intervals $(I_j : 1 \leq j \leq N)$ such that
	\[
		\Lambda = \bigcup_{j = 1}^N I_j.
	\]
	Moreover, for $x \in \Lambda$,
	\[
		\omega'(x) = 
		\frac{1}{N \pi \sqrt{-\discr \calX_1(x)}} 
		\bigg( \sum_{i = 1}^{N-1} \frac{|[\calX_i(x)]_{2,1}|}{\alpha_{i-1}}
		+
		2 \frac{|[\calX_N(x)]_{2,1}|}{\alpha_{N-1}} \bigg)
	\]
	where $\omega$ is the equilibrium measure corresponding to $\overline{\Lambda}$.
\end{theorem}
\begin{proof}
	Let us begin with the case $N = 1$. Then
	\[
		\calX_1(x) = 
		\begin{pmatrix}
			0 & 1 \\
			-1 & \frac{x - \beta_0/2}{\alpha_0/2}
		\end{pmatrix}.
	\]
	Therefore, by \eqref{eq:4} one obtains
	\[
		\omega'(x) = \frac{1}{\pi \sqrt{-\discr \calX_1(x)}} 
		\frac{|[\calX_1(x)]_{2,1}|}{\alpha_0/2},
	\]
	which ends the proof for $N=1$.

	Suppose that $N \geq 2$. For $k \in \ZZ$ and $i \in \{0, 1, \ldots, N-1\}$, we set
	\[
		\tilde{\alpha}_{kN+i} = 
		\alpha_i \cdot 
		\begin{cases}
			\frac{1}{\sqrt{2}} & \text{if } i \in \{0, N-1\}, \\
			1 & \text{otherwise},
		\end{cases}
		\qquad\text{and}\qquad
		\tilde{\beta}_{kN+i}=
		\beta_i \cdot
		\begin{cases}
			\frac{1}{2} & \text{if } i = 0, \\
			1 & \text{otherwise.}
		\end{cases}
	\]
	Let
	\[
		\tilde{\mathfrak{B}}_j(x) = 
		\begin{pmatrix}
			0 & 1 \\
			-\frac{\tilde{\alpha}_{j-1}}{\tilde{\alpha}_j} & \frac{x-\tilde{\beta}_j}{\tilde{\alpha}_j}
		\end{pmatrix}
		\qquad\text{and}\qquad
		\tilde{\frakX}_n (x) = \prod_{j = n}^{N+n-1} \tilde{\mathfrak{B}}_j(x).
	\]
	By \eqref{eq:4},
	\begin{equation}
		\label{eq:27}
		\tilde{\omega}'(x) = \frac{1}{N \pi \sqrt{-\discr \tilde{\frakX}_1(x)}} 
		\sum_{i = 0}^{N-1} 
		\frac{|[\tilde{\frakX}_i(x)]_{2,1}|}{\tilde{\alpha}_{i-1}}
	\end{equation}
	where $\tilde{\omega}$ is the equilibrium measure corresponding to the closure of
	\[
		\tilde{\Lambda} = \big\{x \in \RR : \discr \tilde{\frakX}_1(x) < 0\big\}.
	\]
	In particular, $\tilde{\Lambda}$ is the union of $N$ non-empty open and disjoint intervals, see
	\cite[Theorem 5.4.2]{Simon2010Book}. Notice that
	\[
		-
		\begin{pmatrix}
			1 & 0 \\
			0 & \sqrt{2}
		\end{pmatrix}
		\tilde{\mathfrak{B}}_0
		\begin{pmatrix}
			\frac{1}{\sqrt{2}} & 0 \\
			0 & 1
		\end{pmatrix}
		=
		\calC.
	\]
	We next show the following claim.
	\begin{claim}
		\label{clm:6}
		If $N \geq 2$, then
		\begin{align}
			\label{eq:78a}
			\calX_1 &= -
					\begin{pmatrix}
					\frac{1}{\sqrt{2}} & 0 \\
					0 & 1
				\end{pmatrix}
				\tilde{\frakX}_1
				\begin{pmatrix}
					\sqrt{2} & 0 \\
					0 & 1
				\end{pmatrix}, \\
			\label{eq:78b}
			\calX_j &= -\tilde{\frakX}_j, \qquad\text{for}\qquad i = 2, 3, \ldots, N-1,\\
			\label{eq:78c}
			\calX_N &= 
			-
			\begin{pmatrix}
				\sqrt{2} & 0 \\
				0 & 1
			\end{pmatrix}
			\tilde{\frakX}_0
			\begin{pmatrix}
				\frac{1}{\sqrt{2}} & 0 \\
				0 & 1
			\end{pmatrix}.
		\end{align}
	\end{claim}
	For $N = 2$, the identities \eqref{eq:78a} and \eqref{eq:78c} can be checked by a direct computations. For $N \geq 3$,
	we first observe that
	\begin{align}
		\label{eq:76a}
		\mathfrak{B}_1 &= \tilde{\mathfrak{B}}_1
		\begin{pmatrix}
			\sqrt{2} & 0 \\
			0 & 1
		\end{pmatrix}, \\
		\label{eq:76b}
		\mathfrak{B}_j &= \tilde{\mathfrak{B}}_j, \qquad\text{for}\qquad i =2, 3, \ldots, N-2,\\
		\label{eq:76c}
		\mathfrak{B}_{N-1} &= 
		\begin{pmatrix}
			1 & 0 \\
			0 & \frac{1}{\sqrt{2}}
		\end{pmatrix}
		\tilde{\mathfrak{B}}_{N-1}.
	\end{align}
	Consequently,
	\[
		\calX_1 = 
		\calC
		\prod_{j = 1}^{N-1} \mathfrak{B}_j \\
		=
		-
		\begin{pmatrix}
			\frac{1}{\sqrt{2}} & 0 \\
			0 & 1
		\end{pmatrix}
		\tilde{\frakX}_1
		\begin{pmatrix}
			\sqrt{2} & 0 \\
			0 & 1
		\end{pmatrix}.
	\]
	Similarly, one can show \eqref{eq:78b} and \eqref{eq:78c}.

	Now, using \eqref{eq:78a}, we easily get
	\[
		\discr \calX_1 = \discr \tilde{\frakX}_1,
	\]
	which implies that $\Lambda = \tilde{\Lambda}$. Hence, $\tilde{\omega} = \omega$. Moreover, by Claim \ref{clm:6}, we
	obtain
	\begin{equation}
		\label{eq:85}
		\begin{aligned}
		|[\tilde{\frakX}_0(x)]_{2,1}|
		&=
		\sqrt{2} |[\calX_N(x)]_{2,1}|, \\
		|[\tilde{\frakX}_1(x) ]_{2,1}|
		&=
		\frac{1}{\sqrt{2}} |[\calX_1(x)]_{2,1}|, \\
		|[\tilde{\frakX}_i(x) ]_{2,1}| &= |[\calX_i(x)]_{2,1}|
		\qquad\text{for}\qquad i = 2, 3, \ldots, N-1
		\end{aligned}
	\end{equation}
	Therefore, for $x \in \Lambda$, formula \eqref{eq:27} gives
	\begin{align*}
		\omega'(x)&=
		\frac{1}{N \pi \sqrt{-\discr \tilde{\frakX}_1(x)}}
		\sum_{i = 0}^{N-1} \frac{|[\tilde{\frakX}_i(x)]_{2,1}|}{\tilde{\alpha}_{i-1}}\\
		&=
		\frac{1}{N \pi \sqrt{-\discr \calX_1(x)}}
		\bigg(2 \frac{|[\calX_N(x)]_{2,1}\big|}{\alpha_{N-1}} + 
		\sum_{i = 1}^{N-1} \frac{|[\calX_i(x) ]_{2,1}|}{\alpha_{i-1}}
		\bigg)
	\end{align*}
	which finishes the proof.
\end{proof}

\begin{corollary}
	\label{cor:6}
	\[
		\sum_{i = 1}^{N-1} \frac{|[\calX_i(x)]_{2,1}|}{\alpha_{i-1}} + 2 \frac{|[\calX_N(x)]_{2,1}|}{\alpha_{N-1}} 
		= \abs{\tr \calX_1'(x)}. 
	\]
\end{corollary}
\begin{proof}
	In view of \cite[formula (3.2)]{VanAssche1993}, 
	\[
		\sum_{i = 0}^{N-1} \frac{|[\tilde{\frakX}_i(x)]_{2,1}|}{\tilde{\alpha}_{i-1}} 
		= \big| \tr \tilde{\frakX}_1'(x) \big|.
	\]
	By Claim \ref{clm:6}, 
	\[
		\tr \tilde{\frakX}_1'(x) = \tr \calX_1'(x),
	\]
	which together with \eqref{eq:85} concludes the proof.
\end{proof}

\begin{remark}
	We want to emphasize that Theorem \ref{thm:4} says that $\Lambda$ is the disjoint union of
	\emph{exactly} $N$ non-empty open intervals. This should be compared with the discussion at beginning of 
	Section~5 in \cite{Janas2011}.
\end{remark}

\begin{proposition}
	\label{prop:9}
	Suppose that a Jacobi matrix $A$ is $N$-periodic blend. Then for each $i \in \{1, 2, \ldots, N\}$
	and $n \geq 0$, 
	\begin{subequations} 
		\begin{align}
		\label{eq:90a}
		\lim_{k \to \infty} 
		p_n^{[k(N+2)+i]}(x) 
		&= [\calX_i(x)]_{2, 2}, \\
		\label{eq:90b}
		\lim_{k \to \infty} 
		\frac{a_{k(N+2)+i}}{\alpha_i} \big( p_n^{[k(N+2)+i]} \big)'(x) 
		&= \big[\calX_i'(x) \big]_{2,2},\\ 
		\label{eq:90c}
		\lim_{k \to \infty} 
		\bigg(\frac{a_{k(N+2)+i}}{\alpha_i} \bigg)^2 \big( p_n^{[k(N+2)+i]} \big)''(x) 
		&= \big[ \calX_i ''(x) \big]_{2, 2},
		\end{align}
	\end{subequations}
	locally uniformly with respect to $x \in \CC$.
\end{proposition}
\begin{proof}
	Lemma \ref{lem:4} together with Proposition \ref{prop:1} easily gives \eqref{eq:90a}. Since
	\[
		\lim_{k \to \infty} \frac{a_{k(N+2)+i}}{\alpha_i} = 1,
	\]
	the uniform convergence in \eqref{eq:90a} entails \eqref{eq:90b} and \eqref{eq:90c}.
\end{proof}

\begin{corollary} 
	\label{cor:3}
	Suppose that a Jacobi matrix $A$ is $N$-periodic blend. Then for each $i \in \{1, 2, \ldots, N\}$,
	\begin{subequations} 
		\begin{align}
		\label{eq:91a}
		\lim_{k \to \infty} \tr \big( X_{k(N+2)+i}(x) \big)
		&= \tr \big( \calX_1 (x) \big), \\
		\label{eq:91b}
		\lim_{k \to \infty} \frac{a_{k(N+2)+i}}{\alpha_i} \tr \big( X_{k(N+2)+i}'(x) \big)
		&= \tr \big( \calX_1'(x) \big), \\
		\label{eq:91c}
		\lim_{k \to \infty} \bigg( \frac{a_{k(N+2)+i}}{\alpha_i} \bigg)^2 \tr \big( X_{k(N+2)+i}''(x) \big) 
		&= \tr \big( \calX_1''(x) \big), 
		\end{align}
	\end{subequations}
	locally uniformly with respect to $x \in \CC$.
\end{corollary}

\section{Christoffel functions for $\calD_r$, $r \geq 1$} 
\label{sec:christoffel}
\subsection{General case}
In this section we determine the asymptotic behavior of the Christoffel--Darboux kernel \eqref{eq:37} on the diagonal.
We start by showing the following lemma. In its proof we use the following well-known fact, called the Stolz--Ces\'aro
theorem: if $(c_n : n \in \NN)$ is a sequence of real numbers strictly increasing and approaching infinity, and
$(f_n : n \in \NN)$ is a sequence of real functions on a compact set $K$, such that
\[
	\lim_{n \to \infty} \frac{f_{n+1}(x) - f_n(x)}{c_{n+1} - c_n} = L(x)
\]
uniformly with respect to $x \in K$, then
\[
	\lim_{n \to \infty} \frac{f_n(x)}{c_n} = L(x)
\]
uniformly with respect to $x \in K$. The classical scalar version (see, e.g. \cite[Section 3.1.7]{Muresan2009}) quickly generalizes to the form above.

\begin{lemma} 
	\label{lem:9}
	Let $(\gamma_n : n \geq 0)$ be a sequence of positive numbers and $(\theta_n : n \geq 0)$ be a sequence of continuous
	functions on some compact set $K \subset \RR^d$ with values in $(0, 2 \pi)$. Suppose that there is a function
	$\theta: K \rightarrow (0, 2\pi)$ such that
	\[
		\lim_{n \to \infty} \theta_n(x) = \theta(x)
	\]
	uniformly with respect to $x \in K$. Then there is $c > 0$ such that for all $x \in K$ and $n \in \NN$,
	\begin{equation}
		\label{eq:81}
		\bigg|
		\sum_{k=0}^n \gamma_k
		\exp \Big( i \sum_{j=0}^k \theta_j(x) \Big)
		\bigg|
		\leq
		c
		\bigg(
		1
		+\sum_{k = 0}^{n-1} 
		\Big(\big|\gamma_{k+1} - \gamma_k \big|
		+ \gamma_{k+1} \big|\theta_{k+1}(x) - \theta(x) \big|\Big)
		\bigg).
	\end{equation}
	In particular, if
	\[
		\sum_{k=0}^\infty \gamma_k = \infty, \qquad\text{and}\qquad
        \lim_{n \to \infty} \frac{\gamma_{n-1}}{\gamma_n} = 1,
	\]
	then
	\begin{equation}
		\label{eq:49}
		\lim_{n \to \infty}
		\sum_{k=0}^n \frac{\gamma_k}{{\sum_{j=0}^n \gamma_j}}
		\exp \bigg( i \sum_{j=0}^k \theta_j(x) \bigg)
		=0
	\end{equation}
	uniformly with respect to $x \in K$.
\end{lemma}
\begin{proof}
	Let us observe that
	\[
		\sum_{k=0}^n \gamma_k \exp\Big(i (k+1) \theta(x) + i \eta_k(x)\Big) =
		\sum_{k=0}^n \gamma_k \ue^{i \eta_k(x)} \big(s_{k+1}(x) - s_k(x)\big)
	\]
	where
	\[
		s_k(x) = \sum_{j=0}^k \ue^{i j \theta(x)} =
	    \ue^{i k \theta(x)/2}
	    \frac{\sin \big( (k+1) \theta(x) / 2 \big)}{\sin \big( \theta(x)/2 \big)},
	\]
	and
	\[
		\eta_k(x) = \sum_{j = 0}^k \big(\theta_j(x) - \theta(x)\big).
	\]
	Hence, by the summation by parts we get
	\begin{equation} 
		\label{eq:71}
		\begin{aligned}
			&\sum_{k=0}^n \gamma_k \exp\Big(i (k+1) \theta(x) + i \eta_k(x)\Big) \\
			&\qquad=
			\gamma_n \ue^{i \eta_n(x)} s_{n+1}(x) - \gamma_0 \ue^{i \eta_0(x)} s_0(x) 
			- 
			\sum_{k=0}^{n-1} \Big(\gamma_{k+1} \ue^{i \eta_{k+1}(x)} - 
			\gamma_k \ue^{i \eta_{k}(x)}\Big) s_{k+1}(x).
		\end{aligned}
	\end{equation}
	Since 
	\[
		\sup_{k \in \NN} \sup_{x \in K} | s_k(x) | < \infty,
	\]
	the first term in \eqref{eq:71} is bounded by a constant multiple of $\gamma_n$. Moreover,
	\[
		\big|
		\gamma_n \ue^{i \eta_n(x)} s_{n+1}(x)
		\big|
		\leq
		c \Big(\sum_{k = 0}^{n-1} \big|\gamma_{k+1} - \gamma_k\big| + 1\Big)
	\]
	because
	\[
		\gamma_n \leq \sum_{k = 0}^{n-1} \big|\gamma_{k+1} - \gamma_k\big| + \gamma_0.
	\]
	Similarly we treat the second term in \eqref{eq:71}. Lastly, the third term in \eqref{eq:71} can be bounded by
	a constant multiple of
	\[
		\sum_{k=0}^{n-1} |\gamma_{k+1} - \gamma_k| 
		+ \sum_{k=0}^{n-1} \gamma_{k+1} | \eta_{k+1}(x) - \eta_k(x) |
	\]
	which together with the identity
	\[
		\eta_{k+1}(x) - \eta_{k}(x) = \theta_{k+1}(x) - \theta(x),
	\]
	entails that \eqref{eq:71} is bounded by a constant multiple of
	\[
		1
		+
		\sum_{k=0}^{n-1} |\gamma_{k+1} - \gamma_k|
		+
		\sum_{k=0}^{n-1} \gamma_{k+1} | \theta_{k+1}(x) - \theta(x) |
	\]
	proving \eqref{eq:81}.

	Lastly, we observe that by the Stolz--Ces\'aro theorem,
	\[
		\lim_{n \to \infty} \frac{\sum_{k=0}^{n-1} |\gamma_{k+1} - \gamma_k|}{\sum_{k=0}^n \gamma_k}
		=
		\lim_{n \to \infty} \frac{|\gamma_n - \gamma_{n-1}|}{\gamma_n} = 0.
	\]
	Similarly, we obtain
	\begin{align*}
		\lim_{n \to \infty} \sum_{j=0}^{n-1} \frac{\gamma_{j+1}}{\sum_{k=0}^n \gamma_k} 
		\sup_{x \in K}| \theta_{j+1}(x) - \theta(x) | 
		&=
		\lim_{n \to \infty} \frac{\gamma_n \cdot \sup_{x \in K} |\theta_n(x) - \theta(x) |}{\gamma_n} \\
		&=
		\lim_{n \to \infty} \sup_{x \in K} |\theta_n(x) - \theta(x)| = 0,
	\end{align*}
	thus
	\[
		\lim_{n \to \infty} \sum_{j=0}^{n-1} \frac{\gamma_{j+1}}{\sum_{k=0}^n \gamma_k} 
		| \theta_{j+1}(x) - \theta(x) |
		= 0
	\]
	uniformly with respect to $x \in K$ proving \eqref{eq:49}, and the lemma follows.
\end{proof}

The following theorem has been proved in \cite{SwiderskiTrojan2019}.
\begin{theorem}{\cite[Theorem C]{SwiderskiTrojan2019}}
	\label{thm:5}
	Let $N$ and $r$ be positive integer and $i \in \{0, 1, \ldots, N-1\}$. Suppose that $K$ is a compact interval
	contained in
	\[
		\Lambda =
		\left\{
		x \in \RR :
		\lim_{j \to \infty} \discr X_{jN+i}(x) \text{ exists and is negative}
		\right\}.
	\]
	Assume that 
	\[
		\lim_{j \to \infty} \frac{a_{(j+1)N+i-1}}{a_{jN+i-1}} = 1
	\]
	and
	\[
		\big(X_{jN+i} : j \in \NN \big) \in \calD_r\big(K, \GL(2, \RR)\big).
	\]
	Suppose that $\calX$ is the limit of $(X_{jN+i} : j \in \NN)$. Then there is a probability measure $\nu$ such
	that $(p_n : n \in \NN_0)$ are orthonormal in $L^2(\RR, \nu)$, which is purely absolutely continuous with continuous
	and positive density $\nu'$ on $K$. Moreover, there are $M > 0$ and a real continuous function
	$\eta: K \rightarrow \RR$ such that for all $k \geq M$,
	\[
		\lim_{k \to \infty}
		\sup_{x \in K}
		\left|
		\sqrt{a_{(k+1)N+i-1}} p_{kN+i}(x)
		-
		\sqrt{\frac{2 |[\calX(x)]_{2,1} |}
		{\pi \nu'(x) \sqrt{-\discr \calX(x)}}}
		\sin\Big(\sum_{j = M+1}^k \theta_{jN+i}(x) + \eta(x) \Big)
		\right|
		=0
	\]
	where 
	\begin{equation}
		\label{eq:51}
		\lim_{n \to \infty} \sup_{x \in K} \big|\theta_n(x) - \arccos\big(\tfrac{1}{2} \tr \frakX(x) \big) \big| = 0.
	\end{equation}
\end{theorem}
The proof of the following proposition is inspired by \cite{Dombrowski1978}.
\begin{proposition}
	\label{prop:8}
	Under the hypotheses of Theorem \ref{thm:5}, we have
	\begin{equation} 
		\label{eq:29}
		\inf_{n \in \NN} a_n > 0.
	\end{equation}
\end{proposition}
\begin{proof}
	First, let us consider the case when the moment problem for $\nu$ is indeterminate, that is
	\begin{equation}
		\label{eq:54}
		\sum_{n=0}^\infty \frac{1}{a_n} < \infty.
	\end{equation}
	But \eqref{eq:54} implies that
	\[
		\lim_{n \to \infty} a_n = \infty,
	\]
	which easily gives \eqref{eq:29}. 

	Assume now that the moment problem for $\nu$ is determinate. By Theorem \ref{thm:5}, the measure $\nu$ 
	has non-trivial absolutely continuous part. To obtain a contradiction, suppose that there is a strictly
	increasing sequence $(L_j : j \in \NN_0)$ such that
	\[
		\lim_{j \to \infty} a_{L_j} = 0.
	\]
	Without loss of generality we may assume that
	\begin{equation}
		\label{eq:11}
		\sum_{j=0}^\infty a_{L_j} < \infty.
	\end{equation}
	The Jacobi matrix \eqref{eq:5} can be written in the following block-form
	\[
		\calA = 
  		\begin{pmatrix}
			J_0 & a_{L_0} P_0 & 0   & 0      &\ldots \\
			a_{L_0}P_0^t & J_1 & a_{L_1} P_1 & 0       & \ldots \\
    		0   & a_{L_1}P_1^t & J_2 & a_{L_2} P_2     & \ldots \\
     		0   & 0   & a_{L_2} P_2^t & J_3   &  \\
			\vdots & \vdots & \vdots  &  & \ddots
  		\end{pmatrix}
	\]
	where $J_i$ are some finite Jacobi matrices with dimensions
	\[
		\begin{cases}
			L_0 & \text{if } i = 0, \\
			L_i - L_{i-1} & \text{otherwise,}
		\end{cases}
	\]
	and $P_i$ are rectangular matrices
	\[
		P_i = 
		\begin{pmatrix}
		0 & 0 & \ldots & 0 \\
		0 & 0 & \ldots & 0 \\
		\vdots & \vdots & \ddots & \vdots \\
		1 & 0 & \ldots & 0
		\end{pmatrix}.
	\]
	Let 
	\[
		\calB = \diag(J_0, J_1, J_2, \ldots).
	\]
	By $A$ and $B$ we denote the restrictions of $\calA$ and $\calB$ to the maximal domains, respectively, that is
	\[
		\Dom(A) = \big\{ x \in \ell^2 : \calA x \in \ell^2 \big\}, \qquad\text{and}\qquad
		\Dom(B) = \big\{ x \in \ell^2 : \calB x \in \ell^2 \big\}.
	\]
	The determinacy of the moment problem for $\nu$ is equivalent to $A$ being self-adjoint. Moreover,
	\begin{equation}
		\label{eq:22}
		\nu(\:\cdot\:) = \langle E_A(\: \cdot\:) \delta_0, \delta_0 \rangle_{\ell^2}
	\end{equation}
	where $E_A$ is the spectral resolution of $A$ and $\delta_0$ is the sequence having $1$ on the zero position and
	zero elsewhere. In view of \eqref{eq:11}, the operator $A - B$ is self-adjoint and belongs to the trace class.
	Hence, by the Kato--Rosenblum theorem (see, e.g., \cite[Theorem 9.29]{Schmudgen2012})
	\[
		\sigma_{\textrm{ac}}(A) = \sigma_{\textrm{ac}}(B).
	\]
	Since $B$ is unitary equivalent to an operator acting by the multiplication by a real-valued sequence,
	$B$ has only discrete spectrum. Therefore,
	\[
		\sigma_{\textrm{ac}}(A) = \emptyset,
	\]
	and consequently, by \eqref{eq:22}, the measure $\nu$ has no non-trivial absolutely continuous part, which leads to
	the contradiction.
\end{proof}

For $i \in \{0, 1, \ldots, N-1\}$ and $n \in \NN$ we set
\[
	K_{i; n}(x, y) = \sum_{j = 0}^n p_{jN+i}(x) p_{jN+i}(y), 
	\qquad x, y \in \RR,
\]
and
\[
	\rho_{i; n} = \sum_{j = 0}^n \frac{1}{a_{jN+i}}.
\]
Let us recall that the Carleman condition \eqref{eq:52} implies that there is the unique probability measure $\mu$ such
that $(p_n : n \in \NN_0)$ are orthonormal in $L^2(\RR, \mu)$.
\begin{theorem}
	\label{thm:6}
	Let $N$ and $r$ be positive integers and $i \in \{0, 1, \ldots, N-1\}$. Suppose that $K$ is a compact interval
	with non-empty interior contained in
	\[
		\Lambda =
		\left\{
		x \in \RR :
		\lim_{j \to \infty} \discr X_{jN+i}(x) \text{ exists and is negative}
		\right\}.
	\]
	Assume that
	\[
	        \lim_{j \to \infty} \frac{a_{(j+1)N+i-1}}{a_{jN+i-1}} = 1
	\]
	and
	\[
		\big(X_{jN+i} : j \in \NN \big) \in \calD_r\big(K, \GL(2, \RR) \big).
	\]
	Suppose that $\calX$ is the limit of $\big(X_{jN+i} : j \in \NN \big)$. If
	\[
		\sum_{j = 0}^\infty \frac{1}{a_{jN+i-1}} = \infty,
	\]
	then
	\begin{equation}
		\label{eq:3}
		K_{i; n}(x, x) = 
		\frac{|[\calX(x)]_{2, 1}|}{\pi \mu'(x) \sqrt{-\discr \calX(x)}}
		\rho_{i-1; n+1} + E_{i; n}(x)
	\end{equation}
	where
	\[
		\lim_{n \to \infty} \frac{1}{\rho_{i-1; n+1}} \sup_{x \in K} \big| E_{i; n}(x) \big| = 0.
	\]
\end{theorem}
\begin{proof}
	Fix a compact interval $K$ with non-empty interior contained in $\Lambda$. In view of Theorem \ref{thm:5}, there is
	$M > 0$ such that for all $k \geq M$,
	\[
		a_{(k+1)N+i-1} p^2_{k N+i}(x) = 
		\frac{2 |[\calX(x)]_{2,1}|}{\pi \mu'(x) \sqrt{-\discr \calX(x)}}
		\sin^2 \Big( \eta(x) + \sum_{j=M+1}^k \theta_{jN+i}(x) \Big) + E_{kN+i}(x)
	\]
	where 
	\[
		\lim_{k \to \infty} \sup_{x \in K}{\abs{E_{kN+i}(x)}} = 0.
	\]
	Since $2 \sin^2(x) = 1 - \cos(2x)$, we have
	\begin{align*}
		\sum_{k = M}^n p_{kN+i}^2(x)
		&=
		\frac{|[ \calX(x)]_{2,1} | }{\pi \mu'(x) \sqrt{-\discr \calX(x)}}
		\sum_{k = M}^n \frac{1}{a_{(k+1)N+i-1}}
		\bigg(1 - \cos \Big( 2 \eta(x) + 2 \sum_{j=M+1}^k \theta_{j N+i}(x) \Big)\bigg) \\
		&\phantom{=}+ \sum_{k = M}^n \frac{1}{a_{(k+1)N+i-1}}
		E_{kN+i}(x).
	\end{align*}
	Notice that the Stolz--Ces\`aro theorem gives
	\[
		\lim_{n \to \infty} 
		\frac{1}{\rho_{i-1; n+1}} \sum_{j = M}^n \frac{1}{a_{(j+1)N+i-1}} E_{jN+i}(x)
		=
		\lim_{n \to \infty} E_{nN+i}(x) = 0
	\]
	uniformly with respect to $x \in K$. Moreover, by Lemma \ref{lem:9}
	\[
		\frac{1}{\rho_{i-1; n+1}}
		\sum_{k=M}^N \frac{1}{a_{(k+1)N+i-1}}  \cos \Big( 2 \eta(x) + 2 \sum_{j=M+1}^k \theta_{j N+i}(x) \Big)
		=0.
	\]
	Finally, since there is $c > 0$ such that
	\[
		\sup_{x \in K}{\sum_{k=0}^{M-1} p^2_{kN+i}(x)} \leq c,
	\]
	we conclude that
	\begin{equation}
		\label{eq:30}
		\lim_{n \to \infty}
		\left|
		\frac{1}{\rho_{i-1; n+1}}
		K_{i; n}(x, x)
		-
		\frac{| [ \calX(x) ]_{2,1}| }{\pi \mu'(x) \sqrt{-\discr \calX(x)}}
		\right|
		=0,
	\end{equation}
	which completes the proof.
\end{proof}

\begin{remark}
	\label{rem:2}
	In view of \cite[Proposition 7]{SwiderskiTrojan2019}, for each compact set $K \subset \RR$ with non-empty interior 
	there is $c > 0$
	such that for all $n \in \NN$,
	\begin{align*}
		&
		\sum_{j = n}^\infty \sup_{x \in K} \big\| X_{(j+1)N+i}(x) - X_{jN+i}(x) \big\| \\
		&\qquad\qquad
		\leq
		c
		\sum_{j = n}^\infty
		\Bigg(
		\bigg|\frac{a_{(j+1)N+i-1}}{a_{(j+1)N+i}} - \frac{a_{jN+i-1}}{a_{jN+i}} \bigg| +
		\bigg|\frac{b_{(j+1)N+i}}{a_{(j+1)N+i}} - \frac{b_{jN+i}}{a_{jN+i}}\bigg| +
		\bigg|\frac{1}{a_{(j+1)N+i}} - \frac{1}{a_{jN+i}} \bigg|\Bigg).
	\end{align*}
	Moreover, the right-hand side is comparable to a constant multiple of
	\[
		\sum_{j = n}^\infty 
		\sup_{x \in K} \big\|B_{(j+1)N+i}(x) - B_{jN+i}(x) \big\|.
	\]
\end{remark}

\subsection{Application to the classes}
We are now ready to prove the main theorem of this section.
\begin{theorem}
	\label{thm:7}
	Let $A$ be a Jacobi matrix with $N$-periodically modulated entries. Suppose that there is $r \geq 1$ such 
	that for every $i \in \{0, 1, \ldots, N-1\}$,
	\[
		\bigg( \frac{a_{kN+i-1}}{a_{kN+i}} : k \in \NN\bigg),
		\bigg( \frac{b_{kN+i}}{a_{kN+i}} : k \in \NN\bigg), 
		\bigg( \frac{1}{a_{kN+i}} : k \in \NN\bigg) \in \calD_r,
	\]
	and
	\begin{equation}
		\label{eq:32}
		\sum_{n = 0}^\infty \frac{1}{a_n} = \infty.
	\end{equation}
	If $\abs{\tr \frakX_0(0)} < 2$, then
	\[
		K_n(x, x) = \frac{\omega'(0)}{\mu'(x)} \rho_n + E_n(x)
	\]
	where
	\[
		\lim_{n \to \infty} \frac{1}{\rho_n} E_n(x) = 0
	\]
	locally uniformly with respect to $x \in \RR$, where $\omega$ is the equilibrium measure corresponding to
	$\sigma_{\textrm{ess}}(\frakA)$, with $\frakA$ being the Jacobi matrix associated to
	$(\alpha_n : n \in \NN_0)$ and $(\beta_n : n \in \NN_0)$, and
	\[
		\rho_n = \sum_{j = 0}^n \frac{\alpha_j}{a_j}.
	\]
\end{theorem}
\begin{proof}
	Let $K$ be a compact interval in $\RR$ with non-empty interior. Observe that, by Remark \ref{rem:2},
	for each $i \in \{0, 1, \ldots, N-1\}$ the sequence $(X_{jN+i} : j \geq 0)$ belongs to $\calD_r \big( K, \GL(2, \RR) \big)$.
	Moreover, by Proposition~\ref{prop:4} we have
	\[
		\lim_{j \to \infty} X_{jN+i}(x) = \mathfrak{X}_i(0),
	\]
	which together with $\discr \frakX_i(0) < 0$ implies that $\Lambda = \RR$. Since for each
	$n, n' \in \NN_0$, by Proposition \ref{prop:5}, we have
	\[
		\lim_{j \to \infty} \frac{a_{jN+n'}}{a_{jN+n}} = \frac{\alpha_{n'}}{\alpha_n},
	\]
	by the Stolz--Ces\`aro theorem
	\begin{align}
		\nonumber
		\lim_{j \to \infty} \frac{\rho_{i'; j}}{\rho_{jN+i}} 
		&=
		\lim_{j \to \infty} \frac{\frac{1}{a_{jN+i'}}}{\sum_{k = 1}^N \frac{\alpha_{i+k}}{a_{jN+i+k}}} \\
		\label{eq:6}
		&=
		\frac{1}{N\alpha_{i'}}.
	\end{align}
	Consequently, the Carleman condition \eqref{eq:32} implies that
	\[
		\lim_{k \to \infty} \sum_{j = 0}^k \frac{1}{a_{jN+i}} = \infty,
	\]
	for each $i \in \{0, 1, \ldots, N-1\}$. Thus, by Theorem \ref{thm:6}, we obtain
	\begin{equation}
		\label{eq:31}
		\lim_{n \to \infty}
		\sup_{x \in K}
		\left|
		\frac{1}{\rho_{i; n}} K_{i; n}(x, x)
		-
		\frac{|[\frakX_i(0)]_{2, 1}|}{\pi \mu'(x) \sqrt{-\discr \frakX_i(0)}}
		\cdot
		\frac{\alpha_i}{\alpha_{i-1}}
		\right|
		=0.
	\end{equation}
	Fix $i \in \{0, 1, \ldots, N-1\}$ and consider $n = kN + i$ for $k \in \NN_0$. We write
	\[
		K_{kN+i} (x, x) = \sum_{i' = 0}^{N-1} K_{i'; k}(x, x) + \sum_{i' = i+1}^{N-1} 
		\big(K_{i'; k-1}(x, x) - K_{i'; k}(x, x) \big).
	\]
	Observe that
	\[
		\sup_{x \in K}{ \big|K_{i'; k}(x, x) - K_{i'; k-1}(x, x) \big|}
		=
		\sup_{x \in K}{p_{kN+i'}^2(x) }
		\leq
		c,
	\]
	hence, by \eqref{eq:31},
	\begin{equation}
		\label{eq:53}
		\lim_{k \to \infty} 
		\frac{1}{\rho_{kN+i} } K_{kN+i}(x, x)
		=
		\sum_{i' = 0}^{N-1} 
		\frac{|[\frakX_{i'}(0)]_{2, 1}|}{\pi \mu'(x) \sqrt{-\discr \frakX_{i'}(0)}}
		\cdot
		\frac{\alpha_{i'}}{\alpha_{i'-1}}
		\cdot
		\lim_{k \to \infty} \frac{\rho_{i'; k}}{\rho_{kN+i}}.
	\end{equation}
	Since
	\[
		\frakX_{i+1}(x) = \big(\mathfrak{B}_i(x)\big) \big(\frakX_i(x) \big)\big(\mathfrak{B}_i(x)\big)^{-1},
	\]
	we immediately get
	\begin{equation}
		\label{eq:7}
		\discr \frakX_{i'}(x) = \discr \frakX_0(x).
	\end{equation}
	Finally, putting \eqref{eq:6} and \eqref{eq:7} into \eqref{eq:53}, we obtain
	\begin{align*}
		\lim_{k \to \infty} \frac{1}{\rho_{kN+i}} K_{kN+i}(x, x)
		&=
		\frac{1}{\pi \mu'(x)} \sum_{i' = 0}^{N-1}
		\frac{|[\frakX_{i'}(0)]_{2, 1}|}{\sqrt{-\discr \frakX_{i'}(0)}}
		\frac{1}{N \alpha_{i'-1}} \\
		&=
		\frac{1}{N \pi \mu'(x) \sqrt{-\discr \frakX_0(0)}}
		\sum_{i' = 0}^{N-1} \frac{|[\frakX_{i'}(0)]_{2, 1}|}{\alpha_{i'-1}}
	\end{align*}
	which together with \eqref{eq:4} completes the proof.
\end{proof}

The following theorem has essentially the same proof as Theorem \ref{thm:7}.
\begin{theorem}
	\label{thm:10}
	Let $A$ be a Jacobi matrix with asymptotically $N$-periodic entries. Suppose that there is $r \geq 1$ such 
	that for every $i \in \{0, 1, \ldots, N-1\}$,
	\begin{equation}
		\label{eq:63}
		\bigg( \frac{a_{kN+i-1}}{a_{kN+i}} : k \in \NN\bigg),
		\bigg( \frac{b_{kN+i}}{a_{kN+i}} : k \in \NN\bigg),
		\bigg( \frac{1}{a_{kN+i}} : k \in \NN\bigg) \in \calD_r.
	\end{equation}
	Let $K$ be a compact interval with non-empty interior contained in
	\[
		\Lambda = \big\{x \in \RR : \big| \tr \frakX_0(x)\big| < 2\big\}.
	\]
	Then
	\[
		K_n(x, x) = \frac{\omega'(x)}{\mu'(x)} \rho_n + E_n(x)
	\]
	with
	\[
		\lim_{n \to \infty} \frac{1}{\rho_n} \sup_{x \in K}{\big| E_n(x) \big|} = 0
	\]
	where $\omega$ is the equilibrium measure corresponding to
	$\sigma_{\textrm{ess}}(\frakA)$, with $\frakA$ being the Jacobi matrix associated to
	$(\alpha_n : n \in \NN_0)$ and $(\beta_n : n \in \NN_0)$, and
	\[
		\rho_n = \sum_{j = 0}^n \frac{\alpha_j}{a_j}.
	\]
\end{theorem}

\begin{remark}
	If the Jacobi matrix $A$ has asymptotically $N$-periodic entries, then the condition \eqref{eq:63} is equivalent
	to
	\[
		\big(a_n : n \in \NN_0\big), \big(b_n : n \in \NN_0\big) \in \calD_r.
	\]
	To see this, let us observe that
	\[
		\inf_{n \in\NN_0}{a_n} > 0.
	\]
	Hence, by \cite[Corollary 2]{SwiderskiTrojan2019}
	\[
		\bigg(\frac{1}{a_n} : n \in\NN_0\bigg) \in \calD_r,
	\]
	thus by \cite[Corollary 1]{SwiderskiTrojan2019}
	\[
		\bigg(\frac{a_{n-1}}{a_n} : n \in \NN\bigg), \bigg(\frac{b_n}{a_n} : n \in \NN_0 \bigg) \in \calD_r.
	\]
	Analogously, one can prove the opposite implication.
\end{remark}

\begin{remark}
	If the Jacobi matrix $A$ has asymptotically $N$-periodic entries, then
	\begin{equation} \label{eq:105}
		\lim_{n \to \infty} \frac{\rho_n}{n} = 1.
	\end{equation}
	Indeed, by the Stolz--Ces\`aro theorem, we have
	\[
		\lim_{n \to \infty} \frac{\rho_n}{n}
		=
		\lim_{n \to \infty} \frac{\alpha_n}{a_n} = 1.
	\]
	Let us recall that in this case $\supp(\mu)$ is compact and $\mu'$ is continuous and positive 
	in the interior of $\supp(\mu)$. Thus, in view of \eqref{eq:105}, Theorem~\ref{thm:10} follows 
	from \cite[Theorem 1]{Totik2000}. We want to point out that our approach is completely different.
\end{remark}

\begin{theorem}
	\label{thm:11}
	Let $A$ be a Jacobi matrix that is $N$-periodic blend. Suppose that there is $r \geq 1$ such 
	that for every $i \in \{0, 1, \ldots, N-1\}$,
	\[
		\bigg( \frac{1}{a_{j(N+2)+i}} : j \in \NN_0 \bigg),
		\bigg( \frac{b_{j(N+2)+i}}{a_{j(N+2)+i}} : j \in \NN_0\bigg) \in \calD_r,
	\]
	and
	\[
		\bigg( \frac{1}{a_{j(N+2)+N}} : j \in \NN_0 \bigg),
		\bigg( \frac{1}{a_{j(N+2)+N+1}} : j \in \NN_0 \bigg),
		\bigg( \frac{a_{j(N+2)+N}}{a_{j(N+2)+N+1}} : j \in \NN_0 \bigg) \in \calD_r.
	\]
	Let $K$ be a compact subset with non-empty interior contained in
	\[
		\Lambda = \big\{x \in \RR : \big| \tr \calX_1 (x) \big| < 2\big\}
	\]
	where $\calX_1$ is the limit of $(X_{j(N+2)+1} : j \in \NN_0)$. Then
	\[
		K_n(x, x) = \frac{\omega'(x)}{\mu'(x)}\rho_n + E_n(x)
	\]
	where $\omega$ is the equilibrium measure corresponding to $\overline{\Lambda}$, and
	\[
		\lim_{n \to \infty} \frac{1}{\rho_n} \sup_{x \in K}{\big| E_n(x) \big|} = 0,
	\]
	with
	\[
		\rho_n = 
		\sum_{i = 0}^{N-1} \sum_{\stackrel{1 \leq m \leq n}{m \equiv i \bmod (N+2)}} \frac{\alpha_i}{a_m}.
	\]
\end{theorem}
\begin{proof}
	Let $K$ be a compact set with non-empty interior contained in $\Lambda$. Fix $i \in \{1,2, \ldots, N\}$. Let us recall
	that the sequence $(X_{j(N+2)+i} : j \in \NN)$ converges to $\calX_i$. Moreover,
	we claim the following holds.
	\begin{claim}
		\label{clm:2}
		The sequence $(X_{j(N+2)+i} : j \in \NN)$ belongs to $\calD_r\big(K, \GL(2, \RR)\big)$.	
	\end{claim}
	For the proof, let us observe that if $i \in \{0, 1, \ldots, N-1\}$,
	\[
		\inf_{j \in \NN}{\frac{1}{a_{j(N+2)+i}}} > 0,
	\]
	thus, by \cite[Corollary 2]{SwiderskiTrojan2019},
	\[
		\big(a_{j(N+2)+i} : j \in \NN_0\big) \in \calD_r,
	\]
	and consequently, by \cite[Corollary 1(i)]{SwiderskiTrojan2019},
	\[
		\big(b_{j(N+2)+i} : j \in \NN_0\big) \in \calD_r.
	\]
	Moreover,
	\[
		\bigg(\frac{1}{a_{j(N+2)+N} \cdot a_{j(N+2)+N+1}} : j \in \NN_0\bigg) \in \calD_r.
	\]
	Since
	\[
		\lim_{j \to \infty} \frac{a_{j(N+2)+N}}{a_{j(N+2)+N+1}} = 1,
	\]
	by \cite[Corollary 1]{SwiderskiTrojan2019}, we get
	\[
		\bigg(\frac{a_{j(N+2)+N+1}}{a_{j(N+2)+N}} : j \in \NN_0 \bigg) \in \calD_r.
	\]
	Now, the conclusion follows from the proof of \cite[Corollary 9]{SwiderskiTrojan2019}.

	Since
	\[
		\lim_{k \to \infty} \sum_{j = 0}^k \frac{1}{a_{j(N+2)+i}} = \infty,
	\]
	and, in view of \eqref{eq:69},
	\[
		\discr \calX_i = \discr \calX_1,
	\]
	therefore, by Theorem \ref{thm:6}, we obtain
	\begin{equation}
		\label{eq:74}
		K_{i; k}(x, x) = 
		\frac{|[\calX_i(x)]_{2,1}|}{\pi \mu'(x) \sqrt{-\discr \calX_i(x)}} \rho_{i-1; k} +
		E_{i; k}(x)
	\end{equation}
	where
	\[
		\lim_{k \to \infty} \frac{1}{\rho_{i-1; k}} \sup_{x \in K}{\big|E_{i; k}(x) \big|} = 0.
	\]
	Next, we show
	\begin{equation}
		\label{eq:70}
		\lim_{k \to \infty}
		\sup_{x \in K}{
		\frac{1}{\rho_{N-1; k}}
		\big|
		K_{0; k}(x, x) - K_{N; k}(x, x)
		\big|} = 0.
	\end{equation}
	and
	\begin{equation}
		\label{eq:94}
		\lim_{k \to \infty} \sup_{x \in K}{\frac{1}{\rho_{N-1; k}} \big| K_{N+1; k}(x, x) \big|} = 0.
	\end{equation}
	First, we prove the following claim.
	\begin{claim}
		\label{clm:4}
		There is $c > 0$ such that for all $k \in \NN_0$,
		\begin{equation}
			\label{eq:66}
			\sup_{x \in K} {\big| p_{k(N+2)+N+1}(x) \big|} \leq c \frac{1}{a_{k(N+2)+N+1}},
		\end{equation}
		and
		\begin{equation}
			\label{eq:67}
			\sup_{x \in K} {\big| p_{k(N+2)+N}(x) + p_{(k+1)(N+2)}(x) \big|}
			\leq c \bigg(\bigg|1 - \frac{a_{k(N+2)+N}}{a_{k(N+2)+N+1}}\bigg| + \frac{1}{a_{k(N+2)+N+1}^2} \bigg).
		\end{equation}
	\end{claim}
	For the proof, let us notice that \cite[Theorem A]{SwiderskiTrojan2019} with $i = 1$, implies
	\[
		\sup_{x \in K}{a_{k(N+2)}\big(p_{k(N+2)}^2(x) + p_{k(N+2)+1}^2 (x)\big)} \leq c.
	\]
	Since $(a_{k(N+2)} : k \in \NN_0)$ is bounded, by \cite[Corollary 9]{SwiderskiTrojan2019} with $i = 1$,
	\begin{equation}
		\label{eq:64}
		\sup_{k \in \NN}{\sup_{x \in K} {\big| p_{k(N+2)+1}(x) \big|}} \leq c, 
	\end{equation}
	consequently we get
	\begin{equation}
		\label{eq:65}
		\sup_{k \in \NN}{\sup_{x \in K} {\big| p_{k(N+2)}(x) \big|}} \leq c.
	\end{equation}
	Analogous reasoning for $i = N$ shows that
	\begin{equation}
		\label{eq:68}
		\sup_{k \in \NN}{\sup_{x \in K}{\big|p_{k(N+2)+N}(x)\big|}} \leq c.
	\end{equation}
	Recall that we have the following recurrence relation
	\[
		p_{k(N+2)+N+1}(x) = \frac{x - b_{(k+1)(N+2)}}{a_{k(N+2)+N+1}} p_{(k+1)(N+2)}(x) 
		- \frac{a_{(k+1)(N+2)}}{a_{k(N+2)+N+1}} p_{(k+1)(N+2)+1}(x).
	\]
	Therefore, by \eqref{eq:64} and \eqref{eq:65}, we easily get \eqref{eq:66}. Moreover, we have
	\[
		p_{(k+1)(N+2)}(x) = \frac{x}{a_{k(N+2)+N+1}} p_{k(N+2)+N+1}(x) 
		- \frac{a_{k(N+2)+N}}{a_{k(N+2)+N+1}} p_{k(N+2)+N}(x), 
	\]
	thus
	\begin{align*}
		\big|p_{(k+1)(N+2)}(x) + p_{k(N+2)+N}(x) \big|
		&\leq
		\frac{\abs{x}}{a_{k(N+2)+N+1}} \big|p_{k(N+2)+N+1}(x)\big| \\
		&\phantom{\leq}+ \bigg|1 - \frac{a_{k(N+2)+N}}{a_{k(N+2)+N+1}} \bigg| \big|p_{k(N+2)+N}(x)\big|,
	\end{align*}
	which together with \eqref{eq:66} and \eqref{eq:68} entails \eqref{eq:67}.

	Now using Claim \ref{clm:4} together with \eqref{eq:65} and \eqref{eq:68}, we easily see that
	\begin{align*}
		\big|K_{0; k}(x, x) - K_{N; k}(x, x) \big|
		&=
		\Big|\sum_{j = 0}^k p_{j(N+2)}^2(x) - \sum_{j = 0}^k p_{j(N+2)+N}^2(x) \Big|\\
		&\leq
		c\sum_{j = 0}^k \frac{1}{a_{j(N+2)+N+1}^2} + \bigg|1 - \frac{a_{j(N+2)+N}}{a_{j(N+2)+N+1}} \bigg|.
	\end{align*}
	Since $(a_{k(N+2)+N-1} : k \in \NN_0)$ is bounded, by the Stolz--Ces\`aro theorem, we obtain
	\[
		\lim_{k \to \infty} \frac{1}{\rho_{N-1; k}} \sum_{j = 0}^k \frac{1}{a_{j(N+2)+N+1}^2}
		=
		\lim_{k \to \infty} \frac{a_{k(N+2)+N-1}}{a_{k(N+2)+N+1}^2}
		=0,
	\]
	and
	\[
		\lim_{k \to \infty} \frac{1}{\rho_{N-1; k}} \sum_{j = 0}^k \bigg|1 - \frac{a_{j(N+2)+N}}{a_{j(N+2)+N+1}} \bigg|
		=
		\lim_{k \to \infty} a_{k(N+2)+N-1} \bigg|1 - \frac{a_{k(N+2)+N}}{a_{k(N+2)+N+1}}\bigg|
		=0,
	\]
	which gives \eqref{eq:70}. To prove \eqref{eq:94}, we reason analogously. Namely, by Claim \ref{clm:4}, we have
	\[
		\big|K_{N+1; k}(x, x) \big|
		\leq
		c \sum_{j = 0}^k \frac{1}{a_{j(N+2)+N+1}^2},
	\]
	thus, by the Stolz--Ces\`aro theorem, we get
	\begin{align*}
		\lim_{k \to \infty} \sup_{x \in K}{\frac{1}{\rho_{N-1; k}} \big| K_{N+1; k}(x, x) \big|} 
		&\leq
		c
		\lim_{k \to \infty} \frac{1}{\rho_{N-1; k}} \sum_{j = 0}^k \frac{1}{a_{j(N+2)+N+1}^2} \\
		&=
		\lim_{k \to \infty} \frac{a_{k(N+2)+N-1}}{a_{k(N+2)+N+1}^2} 
		=0.
	\end{align*}
	Notice that, by \cite[Corollary 9]{SwiderskiTrojan2019} and \eqref{eq:66}, we have 
	\begin{equation}
		\label{eq:73}
		\sup_{n \in \NN_0}{\sup_{x \in K}{\big| p_n(x) \big|}} \leq c.
	\end{equation}
	Next we show the following statement.
	\begin{claim}
		\label{clm:7}
		For each $i \in \{0, 1, \ldots, N+1\}$ and $i' \in \{0, 1, \ldots, N-1\}$,
		\begin{equation}
			\label{eq:33}
			\lim_{k \to \infty} \frac{\rho_{i'; k}}{\rho_{k(N+2)+i}} = \frac{1}{N \alpha_{i'}}.
		\end{equation}
	\end{claim}
	For the proof, let us consider $i \in \{0, 1, \ldots, N-2\}$. By the Stolz--Ces\`aro theorem, we have
	\begin{align*}
		\lim_{k \to \infty} \frac{\rho_{i'; k}}{\rho_{k(N+2)+i}}
		&=
		\lim_{k \to \infty} \frac{\frac{1}{a_{k(N+2)+i'}}}
		{\frac{\alpha_{i+1}}{a_{k(N+2)+i+1}} + \ldots + \frac{\alpha_{N-1}}{a_{k(N+2)+N-1}} + 
		\frac{\alpha_0}{a_{(k+1)(N+2)}} + \ldots + \frac{\alpha_i}{a_{(k+1)(N+2)+i}}} \\
		&=
		\frac{1}{N\alpha_{i'}}.
	\end{align*}
	For $i = N-1$ we obtain
	\[
		\lim_{k \to \infty} \frac{\rho_{i'; k}}{\rho_{k(N+2)+N-1}}
		=
		\lim_{k \to \infty}
		\frac{\frac{1}{a_{k(N+2)+i'}}}
		{\frac{\alpha_0}{a_{(k+1)(N+2)}} + \ldots + \frac{\alpha_{N-1}}{a_{(k+1)(N+2)+N-1}}}
		=
		\frac{1}{N\alpha_{i'}}.
	\]
	Finally, we observe that
	\begin{align*}
		\rho_{(k+1)(N+2)+N-1} - \rho_{k(N+2)+N-1}
		&=
		\rho_{(k+1)(N+2)+N} - \rho_{k(N+2)+N} \\
		&=
		\rho_{(k+1)(N+2)+N+1} - \rho_{k(N+2)+N+1},
	\end{align*}
	which entails \eqref{eq:33} for $i \in \{N, N+1\}$.

	Now, writing
	\[
		K_{k(N+2)+i}(x, x) = \sum_{i' = 0}^{N+1} K_{i'; k}(x, x) 
		+ \sum_{i' = i+1}^{N+1} \big(K_{i'; k-1}(x,x) - K_{i'; k}(x, x) \big),
	\]
	by \eqref{eq:70}, \eqref{eq:94} and \eqref{eq:73}, we obtain
	\[
		K_{k(N+2)+i}(x, x)
		=
		\sum_{i' = 1}^{N-1} K_{i'; k}(x, x)
		+
		2
		K_{N; k}(x, x)
		+
		o\big(\rho_{k(N+2)+i}\big)
	\]
	uniformly with respect to $x \in K$. Using \eqref{eq:74} together with \eqref{eq:33}, we arrive at 
	\[
		K_n(x, x) 
		= \frac{1}{N \pi \mu'(x) \sqrt{-\discr \calX_1(x)}}
		\bigg(
		\sum_{i = 1}^{N-1} \frac{|[\calX_{i}(x) ]_{2,1}|}{\alpha_{i-1}}
		+
		2 \frac{|[\calX_{N}(x)]_{2,1}|}{\alpha_{N-1}}\bigg) \rho_n
		+
		o\big(\rho_n\big)
	\]
	uniformly with respect to $x \in K$. To finish the proof, it is sufficient to invoke Theorem \ref{thm:4}.
\end{proof}

\begin{remark}
	If $A$ is a Jacobi matrix that is $N$-periodic blend then
	\begin{equation}
		\label{eq:75}
		\lim_{n \to \infty} \frac{\rho_n}{n} = \frac{N}{N+2}.
	\end{equation}
	Indeed, for each $i \in \{0, 1, \ldots, N+1\}$, by Claim \ref{clm:7}, we have
	\[
		\lim_{k \to \infty} \frac{\rho_{k(N+2)+i}}{k(N+2)+i} 
		= \lim_{k \to \infty} \frac{\rho_{k(N+2)+i}}{\rho_{0; k}}
		\cdot \frac{\rho_{0; k}}{k(N+2)+i} = N \alpha_0 \cdot\lim_{k \to \infty} \frac{\rho_{0; k}}{k(N+2)+i}.
	\]
	Now, by the Stolz--Ces\`aro theorem, we obtain
	\[
		\lim_{k \to \infty} \frac{\rho_{0; k}}{k(N+2)+i} 
		= \lim_{k \to \infty} \frac{1}{(N+2) a_{kN}} 
		= \frac{1}{(N+2)\alpha_0},
	\]
	and the formula \eqref{eq:75} follows.
	
	Let us recall that $\supp(\mu)$ is \emph{not} compact. In \cite[Theorem 5]{Janas2011} there was
	provided examples of Jacobi parameters from this class such that the set of the accumulation points
	of $\supp(\mu)$ is equal to the compact set $\overline{\Lambda}$. From \cite[Corollary 9]{SwiderskiTrojan2019} the density $\mu'$ is continuous and positive in $\Lambda$. Hence, in view of \eqref{eq:75}, the hypothesis on the compactness of $\supp(\mu)$ from \cite[Theorem 1]{Totik2000} cannot be replaced by compactness of the set of its accumulation points.
\end{remark}

\subsection{Ignjatović conjecture} \label{sec:ign_conj}
In this section we show the relation between Theorem \ref{thm:7} and the conjecture due to Ignjatović
\cite[Conjecture 1]{Ignjatovic2016}.
\begin{conjecture}[Ignjatović, 2016]
	\label{conj:1}
	Suppose that
	\begin{itemize}
		\item[($\calC_1$)]
		$\begin{aligned}[b]
			\lim_{n \to \infty} a_n = \infty
		\end{aligned}$;
		\item[($\calC_2$)]
		$\begin{aligned}[b]
			\lim_{n \to \infty} \Delta a_n = 0
		\end{aligned}$;
		\item[($\calC_3$)]
		There exist $n_0, m_0$ such that $a_{n+m} > a_n$ holds for all $n \geq n_0$ and all $m \geq m_0$;
		\item[($\calC_4$)]
		$\begin{aligned}[b]
		\sum_{n=0}^\infty \frac{1}{a_n} = \infty
		\end{aligned}$;
		\item[($\calC_5$)]
		There exists $\kappa > 1$ such that $\sum_{n=0}^\infty \frac{1}{a_n^\kappa} < \infty$;
		\item[($\calC_6$)]
		$\begin{aligned}[b]
		\sum_{n=0}^\infty \frac{|\Delta a_n|}{a_n^2} < \infty
		\end{aligned}$;
		\item[($\calC_7$)]
		$\begin{aligned}[b]
			\sum_{n=0}^\infty \frac{\big|\Delta^2 a_n \big|}{a_n} < \infty
		\end{aligned}$.
	\end{itemize}
	If
	\[
		-2 < \lim_{n \to \infty} \frac{b_n}{a_n} < 2,
	\]
	then for any $x \in \RR$, the limit
	\[
		\lim_{n \to \infty} \bigg( \sum_{j=0}^n \frac{1}{a_j} \bigg)^{-1} \sum_{j=0}^n p_j^2(x),
	\]
	exists and is positive.
\end{conjecture}
In \cite[Corollary 3]{Ignjatovic2016} the conclusion was shown to hold in the case $b_n \equiv 0$. Later, in
\cite[Corollary 3]{PeriodicII}, the result was extended to a more general class of sequences with
\[
	\lim_{n \to \infty} \frac{b_n}{a_n} = 0,
\]
and it was shown that
\[
	\lim_{n \to \infty} \bigg( \sum_{j=0}^n \frac{1}{a_j} \bigg)^{-1} \sum_{j=0}^n p_j^2(x)
	=
	\frac{1}{2 \pi \mu'(x)}.
\]
Our results imply the following corollary.
\begin{corollary}
	\label{cor:1}
	Let $N$ be a positive integer and let $r \geq 1$. Suppose that for each $i \in \{0, 1, \ldots, N-1\}$,
	\[
		\bigg( \frac{a_{jN+i-1}}{a_{jN+i}} : j \in \NN \bigg), 
		\bigg( \frac{b_{jN+i}}{a_{jN+i}} : j \in \NN_0 \bigg), 
		\bigg( \frac{1}{a_{jN+i}} : j \in \NN_0 \bigg) \in \calD_r,
	\]
	and
	\[
		\lim_{n \to \infty} \frac{a_{n-1}}{a_n} = 1, \qquad
		\lim_{n \to \infty} \frac{b_n}{a_n} = q, \qquad
		\lim_{n \to \infty} a_n = \infty.
	\]
	If $|q|<2$ with 
	\begin{equation} \label{eq:91}
		q \notin \big\{2 \cos(j \tfrac{\pi}{N}) : j = 1, 2, \ldots, N-1 \big\},
	\end{equation}
	and the Carleman condition is satisfied, then
	\begin{equation}
		\label{eq:109}
		\lim_{n \to \infty} 
		\bigg( \sum_{j=0}^n \frac{1}{a_j} \bigg)^{-1} \sum_{j=0}^n p_j^2(x) 
		= \frac{1}{\pi \sqrt{4-q^2} \mu'(x)},
	\end{equation}
	locally uniformly with respect to $x \in \RR$.
\end{corollary}
\begin{proof}
Let
\[
	\alpha_n \equiv 1, \quad \beta_n \equiv q.
\]
By \eqref{eq:4},
\[
	\omega'(x) = \frac{1}{\pi \sqrt{4 - (x-q)^2}}, \qquad x \in (-2+q, 2+q)
\]
Hence, the conclusion follows from Theorem~\ref{thm:7} provided that we can show
\begin{equation} \label{eq:93}
	|\tr \frakX_0(0)| < 2.
\end{equation}
To do so, let us observe that
\[
	\frakX_0(0) = 
	\begin{pmatrix}
		0 & 1 \\
		-1 & -q
	\end{pmatrix}^N
	= 
	\begin{pmatrix}
		0 & 1 \\
		-1 & \frac{-q/2 - 0}{1/2}
	\end{pmatrix}^N.
\]
Hence, by Lemma~\ref{lem:4}
\[
	\frakX_0(0) = 
	\begin{pmatrix}
		-U_{N-2}(-\tfrac{q}{2}) & U_{N-1}(-\tfrac{q}{2}) \\
		-U_{N-1}(-\tfrac{q}{2}) & U_{N}(-\tfrac{q}{2})
	\end{pmatrix},
\]
where $(U_n : n \in \NN_0)$ is the sequence of the Chebyshev polynomials of the second kind (see \cite[formula (1.6)]{Mason2002}). By \cite[formula (1.7)]{Mason2002}
\[
	U_{N}(-\tfrac{q}{2}) - U_{N-2}(-\tfrac{q}{2}) = 2 \cos \big( N \arccos(-\tfrac{q}{2}) \big).
\]
Hence, \eqref{eq:91} implies \eqref{eq:93}. The proof is complete.
\end{proof}
\begin{remark}
The case when \eqref{eq:91} is violated is more complicated and demands stronger hypotheses. We 
refer to \cite{ChristoffelII} for more details.
\end{remark}

As it was shown in \cite[Section 4]{PeriodicII} conditions $(\calC_1)$--$(\calC_7)$ imply the hypotheses of
Corollary~\ref{cor:1} with $b_n \equiv 0$, $N=1$ and $r=1$. On the other hand, in Conjecture~\ref{conj:1} no regularity
assumptions on $(b_n)$ was imposed, whereas in Corollary~\ref{cor:1} we asked for
\[
	\bigg( \frac{b_{jN+i}}{a_{jN+i}} : j \in \NN_0 \Big) \in \calD_r,
\]
for each $i \in \{0, 1, \ldots, N-1\}$. The following example illustrates that some regularity assumption on $(b_n)$ is
necessary.
\begin{example}
Let
\[
	a_n = \sqrt{n+1}, \qquad\text{and}\qquad
	b_n = 
	\begin{cases}
		1 & \text{$n$ even} \\
		0 & \text{otherwise.}
	\end{cases}
\]
Then the measure $\mu$ is absolutely continuous on $\RR \setminus [0, 1]$ with continuous and positive
density, $0$ is not a mass point of $\mu$, and
\begin{equation}
	\label{eq:55}
	\lim_{n \to \infty} \bigg( \sum_{j=0}^n \frac{1}{a_j} \bigg)^{-1} \sum_{j=0}^n p_j^2(0) = \infty.
\end{equation}
For the proof, we set
\[
	X_n(x) = B_{n+1}(x) B_n(x),
\]
that is 
\begin{equation} 
	\label{eq:56}
	X_n(x) = 
	\begin{pmatrix}
		-\frac{a_{n-1}}{a_n} & \frac{x - b_n}{a_n} \\
		-\frac{a_{n-1}}{a_n} \frac{x - b_{n+1}}{a_{n+1}} 
		& -\frac{a_n}{a_{n+1}} + \frac{x - b_{n+1}}{a_{n+1}} \frac{x - b_n}{a_n}
	\end{pmatrix}.
\end{equation}
Since
\[
	a_{n+1} \big( X_n(x) + \Id\big) = 
	\begin{pmatrix}
		\frac{a_{n+1}}{a_n} ( a_n - a_{n-1} ) & \frac{a_{n+1}}{a_n} ( x - b_n ) \\
		-\frac{a_{n-1}}{a_n} (x - b_{n+1} ) & (a_{n+1} - a_n) + (x - b_{n+1}) \frac{x - b_n}{a_n} 
	\end{pmatrix},
\]
we obtain
\[
	\calR_0 = \lim_{n \to \infty} a_{2n+1} \big( X_{2n}(x) + \Id\big) = 
	\begin{pmatrix}
		0 & x-1 \\
		-x & 0
	\end{pmatrix}.
\]
Consequently,
\[
	\Lambda = \big\{ x \in \RR : \discr \calR_0(x) < 0 \big\} = \RR \setminus [0, 1].
\]
Therefore, by \cite[Theorem D]{PeriodicIII} and \cite[Proposition 11]{PeriodicIII}, the measure $\mu$ is purely absolutely
continuous on $\Lambda$ with positive continuous density proving the first assertion.

Next, let us observe that, by \eqref{eq:56},
\begin{align*}
	\begin{pmatrix}
		p_{2n}(0) \\
		p_{2n+1}(0)
	\end{pmatrix} 
	&=
	X_{2n-1}(0)
	\begin{pmatrix}
		p_{2n-2}(0) \\
		p_{2n-1}(0)
	\end{pmatrix} \\
	&=
	\begin{pmatrix}
		-\frac{a_{2n-2}}{a_{2n-1}} & 0 \\
		\frac{a_{2n-2}}{a_{2n-1}} \frac{1}{a_{2n}} & -\frac{a_{2n-1}}{a_{2n}}
	\end{pmatrix}
	\begin{pmatrix}
		p_{2n-2}(0) \\
		p_{2n-1}(0)
	\end{pmatrix}.
\end{align*}
Thus
\begin{equation} 
	\label{eq:57}
	p_{2n}(0) = (-1)^n \frac{a_0 a_2 \ldots a_{2n-2}}{a_1 a_3 \ldots a_{2n-1}} =
	(-1)^n \frac{\sqrt{(2n)!}}{2^n n!},
\end{equation}
and
\begin{equation}
	\label{eq:58}
	p_{2n+1}(0) 
	=
	-\frac{1}{a_{2n}} p_{2n}(0) - \frac{a_{2n-1}}{a_{2n}} p_{2n-1}(0).
\end{equation}
By \eqref{eq:57} and \eqref{eq:58}, $(p_{2n+1}(0) : n \in \NN_0)$ satisfies
\begin{equation}
	\label{eq:23}
	\left\{
	\begin{aligned}
	x_n &= -\frac{1}{\sqrt{2n+1}} (-1)^n \frac{\sqrt{(2n)!}}{2^n n!} - \sqrt{\frac{2n}{2n+1}} x_{n-1}, \qquad n \geq 1, \\
	x_0 &= -1.
	\end{aligned}
	\right.
\end{equation}
It can be verified that 
\[
	x_n = (-1)^{n+1} \frac{\sqrt{(n+1) (2n+2)!}}{(n+1)! 2^{n+1/2}}
\]
is the only solution of \eqref{eq:23}. Hence,
\[
	p^2_{2n+1}(0) = \frac{(n+1) (2n+2)!}{\big( (n+1)! \big)^2 2^{2n+1}}.
\]
Using the Stirling's formula, we can find that
\begin{equation} 
	\label{eq:59}
	\lim_{n \to \infty} \frac{p^2_{2n+1}(0)}{\sqrt{n+1}} = \frac{\sqrt{\pi}}{2}.
\end{equation}
Now, by the Stolz--Ces\`aro theorem
\begin{align*}
	\lim_{n \to \infty} 
	\bigg( \sum_{j=0}^{2n} \frac{1}{\sqrt{j+1}} \bigg)^{-1} \sum_{j=0}^{n-1} p_{2j+1}^2(0)
	&=
	\lim_{n \to \infty}
	\frac{p_{2n - 1}^2(x)}{\frac{1}{\sqrt{2n}} + \frac{1}{\sqrt{2n+1}}} \\
	&=
	\frac{2}{\sqrt{\pi}}
	\lim_{n \to \infty}
	\frac{\sqrt{n}}{\frac{1}{\sqrt{2n}} + \frac{1}{\sqrt{2n+1}}}
	=\infty
\end{align*}
proving \eqref{eq:55}. Moreover, by \eqref{eq:59}, the sequence $\big( p_n(0) : n \in \NN_0 \big)$ 
is not square summable. Hence, $0$ is not a mass point of $\mu$.
\end{example}

\begin{remark}
It turns out that in the previous example the measure $\mu$ on $(0,1)$ can have only mass points without any accumulation
points on this set (see \cite{Discrete}). Moreover, recently in \cite{ChristoffelII}, we have shown that in this case
\[
	\lim_{n \to \infty} 
	\bigg( \sum_{j=0}^n \frac{1}{a_j} \bigg)^{-1} \sum_{j=0}^n p_j^2(x) 
	= \frac{|2x-1|}{4\pi \sqrt{x^2 - x}} \frac{1}{\mu'(x)},
\]
locally uniformly with respect to $x \in \RR \setminus [0,1]$. So the value of the limit is different than in
\eqref{eq:109}.
\end{remark}

\section{Christoffel--Darboux kernel for $\calD_1$} \label{sec:CD_for_D1}
For $\calD_1$ class, we can describe the speed of convergence in Theorem \ref{thm:5} and Theorem \ref{thm:6}.
First, let us show a refined asymptotic of polynomials.

\subsection{Asymptotics of polynomials}
\begin{theorem}
	\label{thm:14}
	Let $N$ be a positive integer and $i \in \{0, 1, \ldots, N-1\}$. Suppose that $K$ is a compact interval with
	non-empty interior contained in
	\[
		\Lambda =
		\left\{
		x \in \RR :
		\lim_{j \to \infty} \discr X_{jN+i}(x) \text{ exists and is negative}
		\right\}.
	\]
	Assume that 
	\[
		\lim_{j \to \infty} \frac{a_{(j+1)N+i-1}}{a_{jN+i-1}} = 1
	\]
	and
	\[
		\big(X_{jN+i} : j \in \NN \big) \in \calD_1\big(K, \GL(2, \RR)\big).
	\]
	Let $\calX$ denote the limit of $(X_{jN+i} : j \in \NN)$. Then there is a probability measure $\nu$ such that
	$(p_n : n \in \NN_0)$ are orthonormal in $L^2(\RR, \nu)$ which is purely absolutely continuous with continuous and
	positive density $\nu'$ on $K$. Moreover, there are $M > 0$ and a real continuous function $\eta: K \rightarrow \RR$,
	such that for all $k \geq M$,
	\[
		\sqrt{a_{(k+1)N+i-1}} p_{kN+i}(x)
		=
		\sqrt{\frac{2 |[\calX(x)]_{2,1} |}
		{\pi \nu'(x) \sqrt{-\discr \calX(x)}}}
		\sin\Big(\sum_{j = M+1}^k \theta_{jN+i}(x) + \eta(x) \Big) \\
		+
		E_{kN+i}(x)
	\]
	where
	\[
		\sup_{x \in K} \big|E_{kN+i}(x) \big|
		\leq
		c \sum_{j = k}^\infty \sup_{x \in K} \big\|X_{(j+1)N+i}(x) - X_{jN+i}(x) \big\|,
	\]
	and
	\begin{equation}
		\label{eq:38}
		\theta_n(x) = \arccos\bigg(\frac{\tr X_n(x)}{2 \sqrt{\det X_n(x)}} \bigg).
	\end{equation}
\end{theorem}
\begin{proof}
	Let us fix a compact interval $K$ with non-empty interior. Since $\calX$ is the uniform limit of
	$(X_{jN+i} : j \in \NN)$, there are $\delta > 0$ and $M > 0$ such that for all $x \in K$ and $k \geq M$,
	\[
		\discr X_{kN+i}(x) \leq -\delta < 0.
	\]
	Therefore, the matrix $X_{kN+i}(x)$ has two eigenvalues $\lambda_{kN+i}$ and $\overline{\lambda_{kN+i}}$ where
	\begin{equation}
		\label{eq:26}
		\lambda_n(x) = \frac{\tr X_n(x)}{2} + \frac{i}{2} \sqrt{-\discr X_n(x)}.
	\end{equation}
	Let us next observe that for $k \geq M$,
	\[
		\Im \lambda_{kN+i}(x) = \tfrac{1}{2} \sqrt{-\discr X_{kN+i}(x)} \geq \tfrac{1}{2} \sqrt{\delta}.
	\]
	Moreover,
	\[
		X_{kN+i}(x) = C_k(x) D_k(x) C^{-1}_k(x)
	\]
	where
	\[
		C_k =
		\begin{pmatrix}
			1 & 1 \\
			\frac{\lambda_{kN+i} - [X_{kN+i}]_{1,1}}{[X_{kN+i}]_{1,2}} & 
			\frac{\overline{\lambda_{Nk+i}} - [X_{kN+i}]_{1,1}}{[X_{kN+i}]_{1,2}}
		\end{pmatrix}
		,
		\qquad
		D_k(x) = 
		\begin{pmatrix}
			\lambda_{kN+i} & 0 \\
			0 & \overline{\lambda_{kN+i}}
		\end{pmatrix},
	\]
	which is well-defined since $[X_{kN+i}(x)]_{1,2} \neq 0$ for any $x \in K$.
	Let
	\[
		\phi_{kN+i} = \frac{p_{(k+1)N+i} - \overline{\lambda_{kN+i}} \cdot p_{kN+i}}{\prod_{j = M+1}^k \lambda_{jN+i}}.
	\]
	We claim the following holds true.
	\begin{claim}
		There is $c > 0$ such that for all $m \geq n \geq M$, and $x \in K$,
		\[
			\big|\phi_{mN+i}(x) - \phi_{nN+i}(x) \big|
			\leq
			c \sum_{j = n}^\infty \big| \lambda_{(j+1)N+i}(x) - \lambda_{j N + i}(x) \big|.
		\]
	\end{claim}
	We start by writing
	\[
		p_{mN+i} = 
		\left\langle 
		C_{m-1} \Big(\prod_{j = n}^{m-1} D_j C_j^{-1} C_{j-1} \Big) C_{n-1}^{-1}
		\begin{pmatrix}
			p_{nN+i-1} \\
			p_{nN+i}
		\end{pmatrix}
		,
		\begin{pmatrix}
			0 \\
			1
		\end{pmatrix}
		\right\rangle.
	\]
	Let us now introduce two auxiliary functions
	\[
		q_m = \left\langle
		C_\infty 
		\Big(\prod_{j = n}^{m-1} D_j \Big) C_{n-1}^{-1}
        \begin{pmatrix}
            p_{nN+i-1} \\
            p_{nN+i}
        \end{pmatrix}
        ,
        \begin{pmatrix}
            0 \\
            1
        \end{pmatrix}
		\right\rangle,
	\]
	and
	\[
		\psi_m = \frac{q_{m+1} - \overline{\lambda_{mN+i}} \cdot q_m}{\prod_{j = M+1}^m \lambda_{jN+i}}.
	\]
	Notice that 
	\[
		p_{mN+i} - q_m = 
		\left\langle
		Y_m
		\begin{pmatrix}
			p_{nN+i-1} \\
			p_{nN+i}
		\end{pmatrix}
		,
		\begin{pmatrix}
			0 \\
			1
		\end{pmatrix}
		\right\rangle
	\]
	where
	\[
		Y_m = C_\infty \Big(\prod_{j = n}^{m-1} D_j C_j C_{j-1}^{-1} - \prod_{j = n}^{m-1} D_j \Big) C_{n-1}^{-1}
		+ (C_\infty - C_{n-1}) \Big( \prod_{j = n}^{m-1} D_j \Big) C_{n-1}^{-1}.
	\]
	In view of \cite[Propositon 1]{SwiderskiTrojan2019}, we have
	\begin{align*}
		\norm{Y_m} 
		&\lesssim 
		\bigg(\prod_{j = n}^{m-1} \norm{D_j} \bigg) 
		\bigg(\sum_{j = n-1}^\infty \norm{\Delta C_j} + \norm{C_\infty - C_{m-1}}\bigg) \\
		&\lesssim \bigg(\prod_{j = n}^{m-1} \norm{D_j} \bigg) \sum_{j = n-1}^\infty 
		\norm{\Delta C_j},
	\end{align*}
	thus
	\[
		\norm{Y_m} 
		\lesssim 
		\prod_{j = n}^{m-1} \abs{\lambda_{jN+i}} \cdot 
		\sum_{j = n-1}^\infty \big|\lambda_{(j+1)N+i} - \lambda_{jN +i}\big|.
	\]
	Next, by \cite[Claim 2]{SwiderskiTrojan2019}, there is $c > 0$ such that for all $n \geq M$ and $x \in K$,
	\[
		\frac{\sqrt{p^2_{nN+i}(x) + p^2_{nN+i-1}(x)}}{\prod_{j = M+1}^{n-1} \abs{\lambda_{jN+i}(x)}}
		\leq c,
	\]
	and consequently, for all $m \geq n \geq M$,
	\[
		\bigg| \frac{p_{mN+i} - q_m}{\prod_{j = M+1}^{m-1} \lambda_{jN+i}} \bigg|
		\lesssim
		\sum_{j = n-1}^\infty
		\big|\lambda_{(j+1)N+i} - \lambda_{jN +i}\big|.
	\]
	In particular, we obtain
	\begin{equation}
		\label{eq:25}
		\big| \phi_{mN+i} - \psi_m \big|
		\lesssim
		\sum_{j = n-1}^\infty
		\big|\lambda_{(j+1)N+i} - \lambda_{jN +i}\big|.
	\end{equation}
	Next, we notice that
	\[
		q_{m+1} - \overline{\lambda_{mN+i}} \cdot q_m =
		\left\langle
		C_\infty \Big(D_m - \overline{\lambda_{mN+i}} \Id\Big) \Big( \prod_{j = n}^{m-1} D_j \Big) C_{n-1}^{-1}
		\begin{pmatrix}
            p_{nN+i-1} \\
            p_{nN+i}
        \end{pmatrix},
        \begin{pmatrix}
            0 \\
            1
        \end{pmatrix}
		\right\rangle.
	\]
	Since
	\[
		\frac{1}{\prod_{j = n}^m \lambda_{jN+i}}
		\Big(D_m - \overline{\lambda_{mN+i}} \Id\Big) \Big( \prod_{j = n}^{m-1} D_j \Big)
		=
		\frac{1}{\lambda_{mN+i}}
		\begin{pmatrix}
			\lambda_{mN+i} - \overline{\lambda_{mN+i}} & 0 \\
			0 & 0 
		\end{pmatrix},
	\]
	we obtain
	\begin{align*}
		\big|\psi_m - \psi_n \big| 
		&\lesssim
		\bigg|\frac{\overline{\lambda_{mN+i}}}{\lambda_{mN+i}} 
		- \frac{\overline{\lambda_{nN+i}}}{\lambda_{nN+i}} \bigg| \\
		&\lesssim
		\sum_{j = n}^\infty \big|\lambda_{(j+1)N+i} - \lambda_{jN+i} \big|,
	\end{align*}
	which together with \eqref{eq:25} implies that for all $m \geq n > M$ and $x \in K$,
	\[
		\big|\phi_{mN+i}(x) - \phi_{nN+i}(x) \big| 
		\lesssim
		\sum_{j = n}^\infty \big|\lambda_{(j+1)N+i}(x) - \lambda_{jN+i}(x) \big|.
	\]
	In particular, the sequence $(\phi_{mN+i} : m \in \NN)$ converges. Let us denote by $\vphi$ its limit.
	Since polynomials $p_n$ have real coefficients, by taking imaginary part we arrive at
	\begin{align*}
		&
		\bigg|
		\tfrac{1}{2} \sqrt{-\discr X_{nN+i}(x)} \frac{p_{nN+i}(x)}{\prod_{j=M+1}^n \big|\lambda_{jN+i}(x)|} 
		-
		\abs{\vphi(x)} \sin\Big(\sum_{j = M+1}^n \arg \lambda_{jN+i}(x) + \arg \vphi(x) \Big)
		\bigg| \\
		&\qquad\qquad
		\lesssim
		\sum_{j = n}^\infty \big|\lambda_{(j+1)N+i} - \lambda_{jN+i} \big|.
	\end{align*}
	Because
	\[
		\det X_{jN+i} = \prod_{k = jN+i}^{(j+1)N+i-1} \det B_k = \frac{a_{jN+i-1}} {a_{(j+1)N+i-1}},
	\]
	we obtain
	\begin{align*}
		\prod_{j = M+1}^n \big|\lambda_{jN+i} \big|^2 
		&= \prod_{j = M+1}^n \det X_{jN+i} \\
		&= \frac{a_{(M+1)N+i-1}}{a_{(n+1)N+i-1}}.
	\end{align*}
	Therefore, by \cite[Theorem 6]{SwiderskiTrojan2019}, 
	\begin{align*}
		&\bigg|
		\sqrt{a_{(n+1)N+i-1}} \sqrt{-\discr X_{nN+i}(x)} p_{nN+i}(x) \\
		&\qquad\qquad
		-\sqrt[4]{-\discr \calX(x)} 
		\sqrt{\frac{2 |[\calX(x)]_{2, 1}|}{\pi \nu'(x)}}
		\sin\Big(\sum_{j = M+1}^n \arg \lambda_{j N + i}(x) + \vphi(x)\Big)
		\bigg| \\
		&\qquad\qquad\qquad\qquad
		\lesssim
		\sum_{j = n}^\infty \big|\lambda_{(j+1)N+i} - \lambda_{jN+i} \big|.
	\end{align*}
	Observe that, by \eqref{eq:26},
	\[
		\bigg|
		\frac{1}{\sqrt{-\discr X_{nN+i}(x)}} - \frac{1}{\sqrt{-\discr \calX(x)}} \bigg|
		\lesssim
		\sum_{j = n}^\infty \big|\lambda_{(j+1)N+i} - \lambda_{jN+i} \big|,
	\]
	thus
	\begin{align*}
		&
		\bigg|
		\sqrt{a_{(n+1)N+i-1}} p_{nN+i}(x) - 
		\sqrt{\frac{2 |[\calX(x)]_{2, 1}|}{\pi \nu'(x) \sqrt{-\discr \calX(x)}}}
		\sin\Big(\sum_{j = M+1}^n \arg \lambda_{j N + i}(x) + \vphi(x)\Big)
		\bigg| \\
		&\qquad\qquad
		\lesssim
		\sum_{j = n}^\infty \big|\lambda_{(j+1)N+i} - \lambda_{jN+i} \big|.
	\end{align*}
	Since
	\[
		\big| \lambda_{(j+1)N+i} - \lambda_{jN+i} \big|
		\lesssim
		\big\|X_{(j+1)N+i} - X_{jN+i} \big\|,
	\]
	we finish the proof.
\end{proof}

\begin{remark}
	\label{rem:1}
	Under the assumption of Theorem \ref{thm:14}, we have
	\[
		\theta_{kN+i}(x) = \arccos\Big(\tfrac{1}{2} \tr \calX(x)\Big) + E_{kN+i}(x)
	\]
	where
	\[
		\sup_{x \in K} \big|E_{kN+i}(x)\big|
		\leq
		c
		\sum_{j = k}^\infty 
		\sup_{x \in K} \big\|X_{(j+1)N+i}(x) - X_{jN+i}(x)\big\|.
	\]
	Indeed, since for $m \geq n \geq M$,
	\[
		\bigg|
		\frac{\tr X_{nN+i}}{2 \sqrt{\det X_{nN+i}}}
		-
		\frac{\tr X_{mN+i}}{2 \sqrt{\det X_{mN+i}}}
		\bigg|
		=
		\bigg|
		\frac{\lambda_{nN+i}}{\abs{\lambda_{nN+i}}} - \frac{\lambda_{mN+i}}{\abs{\lambda_{mN+i}}}
		\bigg|
		\lesssim
		\sum_{j = n}^\infty \big|\lambda_{(j+1)N+i} - \lambda_{jN+i} \big|,
	\]
	we get
	\[
		\Big|\theta_n(x) - \arccos \Big(\tfrac{1}{2} \tr \calX(x) \Big) \Big|
		\lesssim
		\sum_{j = n}^\infty \big|\lambda_{(j+1)N+i} - \lambda_{jN+i} \big|,
	\]
	which proves our statement.
\end{remark}

\subsection{Christoffel functions}
We are now in the position to prove the main theorem of this section. 
\begin{theorem} 
	\label{thm:2}
	Let $N$ be a positive integer and $i \in \{0, 1, \ldots, N-1\}$. Suppose that $K$ is a compact interval with
	non-empty interior and contained in
	\[
		\Lambda =
		\left\{
		x \in \RR :
		\lim_{j \to \infty} \discr X_{jN+i}(x) \text{ exists and is negative}
		\right\}.
	\]
	Assume that
	\[
	        \lim_{j \to \infty} \frac{a_{(j+1)N+i-1}}{a_{jN+i-1}} = 1
	\]
	and
	\[
		\big(X_{jN+i} : j \in \NN \big) \in \calD_1\big(K, \GL(2, \RR) \big).
	\]
	If
	\[
		\sum_{j=1}^\infty \frac{1}{a_{jN+i-1}} = \infty,
	\]
	then
	\begin{align*}
		K_{i; n}(x, x)
		=
		\frac{|[\calX(x)]_{2, 1}|}{\pi \mu'(x) \sqrt{-\discr \calX(x)}}
		\rho_{i-1; n+1}
		+
		E_{i; n}(x),
	\end{align*}
	where $\calX$ is the limit of $(X_{jN+i} : j \in \NN)$, and
	\begin{align*}
		\sup_{x \in K} \big|E_{i; n}(x)\big|
		\leq
		c
		\sum_{k = 0}^n 
		\bigg(&\bigg|\frac{1}{a_{(k+1)N+i-1}} - \frac{1}{a_{kN+i-1}} \bigg| \\
		&+
		\frac{1}{a_{(k+1)N+i-1}}
		\sum_{j \geq k}
		\sup_{x \in K}\big\|X_{(j+1)N+i}(x) - X_{jN+i}(x)\big\|
		\bigg).
	\end{align*}
\end{theorem}
\begin{proof}
	Let $K \subset \Lambda$ be a compact interval with non-empty interior. By Theorem~\ref{thm:14}, there are $c > 0$ and
	$M \in \NN$ such that for all $k \geq M$,
	\[
		a_{(k+1)N+i-1} p^2_{k N+i}(x) = 
		\frac{2 |[\calX(x)]_{2,1}|}{\pi \mu'(x) \sqrt{-\discr \calX(x)}}
		\sin^2 \Big( \eta(x) + \sum_{j=M}^k \theta_{jN+i}(x) \Big) + E_{kN+i}(x)
	\]
	where
	\[
		\sup_{x \in K}
		\big| E_{kN+i}(x) \big|
		\leq
		c
		\sum_{j \geq k} \sup_{x \in K} \big\|X_{(j+1)N+i}(x) - X_{jN+i}(x)\big\|.
	\]
	In view of the identity $2 \sin^2(x) = 1 - \cos(2x)$, we get
	\begin{align*}
		\sum_{k = M}^n p_{kN+i}^2(x)
		&=
		\frac{| [ \calX(x) ]_{2,1} | }{\pi \mu'(x) \sqrt{-\discr \calX(x)}}
		\sum_{k = M}^n \frac{1}{a_{(k+1)N+i-1}}
		\bigg(1 - \cos \Big( 2 \eta(x) + 2 \sum_{j=M}^k \theta_{j N+i}(x) \Big)\bigg) \\
		&\phantom{=}+ \sum_{k = M}^n \frac{1}{a_{(k+1)N+i-1}}
		E_{kN+i}(x).
	\end{align*}
	Since there is $c > 0$ such that
	\[
		\sup_{x \in K} \sum_{k=0}^{M-1} p^2_{kN+i}(x) \leq c,
	\]
	by Lemma \ref{lem:9} and Remark \ref{rem:1}, we obtain
	\begin{align*}
		&\bigg|
		K_{i; n}(x, x)
		-
		\frac{| [ \calX(x) ]_{2,1} | }{\pi \mu'(x) \sqrt{-\discr \calX(x)}}
		\rho_{i-1; n+1}
		\bigg| \\
		&\qquad
		\leq c
		\sum_{k = 0}^n
		\bigg(
		\bigg|
		\frac{1}{a_{(k+1)N+i-1}}- \frac{1}{a_{kN+i-1}}
		\bigg|
		+
		\frac{1}{a_{(k+1)N+i-1}}
		\sum_{j \geq k}
		\sup_{x \in K}
		\big\|
		X_{(j+1)N+i}(x) - X_{jN+i}(x)
		\big\|
		\bigg)
	\end{align*}
	which completes the proof.
\end{proof}

\begin{theorem}
	\label{thm:8}
	Let $A$ be a Jacobi matrix with $N$-periodically modulated entries. Suppose that for each
	$i \in \{0, 1, \ldots, N-1\}$,
	\[
		\bigg(\frac{a_{jN+i-1}}{a_{jN+i}} : j \in \NN \bigg),
		\bigg(\frac{b_{jN+i}}{a_{jN+i}} : j \in \NN_0 \bigg),
		\bigg(\frac{1}{a_{jN+i}} : j \in \NN_0 \bigg) \in \calD_1,
	\]
	and
	\[
		\sum_{n = 0}^\infty \frac{1}{a_n} = \infty.
	\]
	If $\abs{\tr \frakX_0(0)} < 2$ then
	\[
		K_n(x, x) = \frac{\omega'(0)}{\mu'(x)} \rho_n + E_n(x)
	\]
	where $\omega$ is the equilibrium measure corresponding to $\sigma_{\textrm{ess}}(\frakA)$ with
	$\frakA$ being the Jacobi matrix associated to $(\alpha_n : n \in \NN_0)$ and $(\beta_n : n \in \NN_0)$,
	\[
		\rho_n = \sum_{j = 0}^n \frac{\alpha_j}{a_j},
	\]
	and for each compact interval $K \subset \RR$ with non-empty interior
	\begin{align*}
		\sup_{x \in K} \big| E_n(x) \big|
		\leq
		c
		\sum_{m = 0}^{n+N} \frac{1}{a_m} 
		\sum_{j \geq 0} \sup_{x \in K} \big\|B_{m+(j+1) N}(x) - B_{m+jN}(x) \big\|.
	\end{align*}
\end{theorem}
\begin{proof}
	Let $K$ be a compact interval with non-empty interior and contained in $\RR$. By Remark \ref{rem:2},
	for each $i \in \{0, 1, \ldots, N-1\}$, the sequence $(X_{jN+i} : j \in \NN)$ belongs to
	$\calD_1\big(K, \GL(2, \RR) \big)$, thus, by Proposition \ref{prop:4}, we have
	\[
		\lim_{j \to \infty} X_{jN+i}(x) = \frakX_i(0)
	\]
	uniformly with respect to $x \in K$. Since
	\[
		\frakX_{i+1}(x) = 
		\big(\mathfrak{B}_i(x)\big) 
		\big(\frakX_i(x) \big)
		\big(\mathfrak{B}_i(x) \big)^{-1},
	\]
	we have
	\[
		\discr \frakX_i(x) = \discr \frakX_0(x).
	\]
	By Theorem \ref{thm:2},
	\[
		K_{i; n}(x, x)
		=
		\frac{|[\frakX_i(0)]_{2, 1}|}{\pi \mu'(x) \sqrt{-\discr \frakX_0(0)}}
		\rho_{i-1; n+1}
		+
		E_{i; n}(x)
	\]
	where
	\[
		\sup_{x \in K}
		\big| E_{i; n}(x) \big|
		\leq
		c
		\sum_{k = 0}^n
		\frac{1}{a_{(k+1)N+i-1}}
		\sum_{j \geq k}
		\sup_{x \in K} \big\|B_{(j+1)N+i}(x) - B_{jN+i}(x) \big\|.
	\]
	For $k \in \NN_0$ and $i \in \{0, 1, \ldots, N-1\}$ we write 
	\[
		K_{kN+i}(x, x) = \sum_{i' = 0}^{N-1} K_{i'; k}(x, x) 
		+ \sum_{i' = i+1}^{N-1} \big(K_{i'; k-1}(x, x)-K_{i'; k}(x, x)\big),
	\]
	Since
	\[
		\big|K_{i'; k-1}(x, x)-K_{i'; k}(x, x) \big| = p_{kN+i'}^2(x) \leq c \frac{1}{a_{kN+i'}},
	\]
	we obtain
	\begin{equation}
		\label{eq:80}
		K_{kN+i}(x, x) = 
		\frac{1} {\pi \mu'(x) \sqrt{-\discr \frakX_0(x)}}
		\sum_{i' = 0}^{N-1} |[\frakX_{i'}(x)]_{2,1}| \cdot \rho_{i'-1; k+1}
		+
		E_{kN+i}(x),
	\end{equation}
	where
	\[
		\sup_{x \in K} \big|E_n(x) \big| \leq c 
		\sum_{m=0}^n \frac{1}{a_{m+N-1}} \sum_{j \geq 0}
		\sup_{x \in K}\big\|B_{m+(j+1)N}(x) - B_{m+jN}(x) \big\|.
	\]
	We next claim the following holds true.
	\begin{claim}
		\label{clm:5}
		For each $i', i'' \in \{0, 1, \ldots, N-1\}$,
		\begin{equation}
			\label{eq:79}
			\Big|\alpha_{i'} \rho_{i'; k} - \alpha_{i''} \rho_{i''; k} \Big|
			\leq
			c
			\sum_{m = 0}^{(k+1)N} \frac{1}{a_m} 
			\sum_{j \geq 0} \sup_{x \in K} \big\|B_{m+jN}(x) - B_{m+(j+1)N}(x) \big\|.
		\end{equation}
	\end{claim}
	For the proof let us observe that, by Proposition \ref{prop:5}, we have
	\begin{align*}
		\bigg|\frac{\alpha_{i-1}}{\alpha_i} - \frac{a_{kN+i-1}}{a_{kN+i}}\bigg|
		&\leq
		\sum_{j \geq k} \bigg|\frac{a_{jN+i-1}}{a_{jN+i}} - \frac{a_{(j+1)N+i-1}}{a_{(j+1)N+i}}\bigg| \\
		&\leq
		c \sum_{j \geq k} \sup_{x \in K} \big\|B_{jN+i}(x) - B_{(j+1)N+i}(x) \big\|.
	\end{align*}
	Therefore, 
	\begin{align*}
		\Big|\alpha_{i'} \rho_{i'; k} - \alpha_{i''} \rho_{i''; k} \Big|
		&\leq 
		\sum_{i = 1}^{N-1} \sum_{j = 0}^k \bigg|\frac{\alpha_{i-1}}{a_{jN+i-1}} - \frac{\alpha_i}{a_{jN+i}}\bigg| \\
		&\leq
		\sum_{i = 1}^{N-1} \sum_{j = 0}^k \frac{\alpha_i}{a_{jN+i-1}} 
		\bigg|\frac{\alpha_{i-1}}{\alpha_i} - \frac{a_{jN+i-1}}{a_{jN+i}}\bigg|,
	\end{align*}
	which implies \eqref{eq:79}.

	Now, using Claim \ref{clm:5}, we can write
	\begin{align*}
		\big|
		N \alpha_i \rho_{i; k+1} - \rho_{(k+1)N}
		\big|
		&\leq
		\sum_{i' = 0}^{N-1} \big|\alpha_i \rho_{i; k+1} - \alpha_{i'} \rho_{i'; k+1} \big| \\
		&\leq
		c \sum_{m = 0}^{(k+1)N} \frac{1}{a_m}
		\sum_{j \geq 0} \sup_{x \in K} \big\|B_{m+jN}(x) - B_{m+(j+1)N}(x) \big\|.
	\end{align*}
	Hence, by \eqref{eq:80}, we obtain
	\begin{align*}
		&
		\left|
		K_{kN+i}(x, x)
		-
		\frac{1}{N \pi \mu'(x) \sqrt{-\discr \frakX_0(0)}}
		\sum_{i' = 0}^{N-1} \frac{|[\frakX_{i'}(0)]_{2,1} |}{\alpha_{i'-1}} \rho_{kN+i}
		\right| \\
		&\qquad\qquad\leq
		c \sum_{m = 0}^{(k+1)N} \frac{1}{a_m} 
		\sum_{j \geq 0} \sup_{x \in K} \big\|B_{m + jN}(x) - B_{m + (j+1)N}(x) \big\|,
	\end{align*}
	which together with \eqref{eq:4}, concludes the proof.
\end{proof}
The following theorem has analogous proof to Theorem \ref{thm:8}.
\begin{theorem}
	\label{thm:9}
	Let $A$ be a Jacobi matrix with asymptotically $N$-periodic entries. Suppose that for each
	$i \in \{0, 1, \ldots, N-1\}$,
	\[
		\bigg(\frac{a_{jN+i-1}}{a_{jN+i}} : j \in \NN \bigg),
		\bigg(\frac{b_{jN+i}}{a_{jN+i}} : j \in \NN_0 \bigg),
		\bigg(\frac{1}{a_{jN+i}} : j \in \NN_0 \bigg) \in \calD_1,
	\]
	Let $K$ be a compact interval with non-empty interior contained in 
	\[
		\Lambda = \big\{x \in \RR : \big|\tr\frakX_0(x) \big| < 2 \big\}.
	\]
	Then
	\[
		K_n(x, x) = \frac{\omega'(x)}{\mu'(x)} \rho_n + E_n(x)
	\]
	where $\omega$ is the equilibrium measure corresponding to $\sigma_{\textrm{ess}}(\frakA)$ with
	$\frakA$ being the Jacobi matrix associated to $(\alpha_n : n \in \NN_0)$ and $(\beta_n : n \in \NN_0)$,
	\[
		\rho_n = \sum_{j = 0}^n \frac{\alpha_j}{a_j},
	\]
	and
	\begin{align} \label{eq:106}
		\sup_{x \in K} \big| E_n(x) \big|
		\leq
		c
		\sum_{m = 0}^{n+N} \frac{1}{a_m} 
		\sum_{j \geq 0} \sup_{x \in K} \big\|B_{m+(j+1) N}(x) - B_{m+jN}(x) \big\|.
	\end{align}
\end{theorem}

In the following two examples we want to compare the estimate \eqref{eq:106} with some known results.

\begin{example}[Generalized Jacobi]
	Let $h$ be a real-analytic positive function on the neighborhood of $[-1, 1]$. Let $\mu$ be
	a probability measure supported on $[-1,1]$ with the density
	\[
		\mu'(x) = c \cdot h(x) (x+1)^{\gamma_1} (x-1)^{\gamma_2}, \qquad x \in (-1,1),
	\]
	where $\gamma_1, \gamma_2 > -1$, and $c$ is the normalizing constant. Then (see \cite[Theorem 1.10]{Kuijlaars2004})
	\[
		a_n = \frac{1}{2} + \frac{c_1}{n^2} + \calO\big(n^{-3}\big), \qquad
		b_n = \frac{c_2}{n^2} + \calO\big(n^{-3}\big).
	\]
	Therefore, by Theorem \ref{thm:8} we obtain
	\[
		\frac{1}{\rho_n} K_n(x, x) = 
		\frac{1}{\pi \sqrt{1-x^2}} \frac{1}{\mu'(x)} + \calO\big(n^{-1}\big).
	\]
	Hence, we obtain the same rate as in \cite[Theorem 1.1(a)]{Kuijlaars2002}.
\end{example}

\begin{example}[Pollaczek-type]
	Let $\mu$ be a probability measure supported on $[-1,1]$ with the density
	\[
    	\mu'(x) = c \cdot \exp \big( -(1-x^2)^{-\gamma} \big), \qquad x \in (-1,1),
	\]
	where $\gamma \in (0, \tfrac{1}{2})$, and $c$ is the normalizing constant. Then (see, \cite[Corollary 4]{Zhou2011})
	\[
	    a_n = \frac{1}{2} + \frac{c_1}{2} n^{-2/(1+2\gamma)} + \calO\big(n^{-2}\big), \qquad
		b_n = 0.
	\]
	Hence, Theorem \ref{thm:8} implies
	\[
		\frac{1}{\rho_n} K_n(x, x)
		=
		\frac{1}{\pi \sqrt{1-x^2}} \frac{1}{\mu'(x)} + \calO\big(n^{-1}\big).
	\]
	It should be compared with \cite[Theorem 1(i)]{Xu2011}
\end{example}

\subsection{Auxiliary results}
\begin{lemma}
	\label{cor:2}
	Let $(\gamma_k : k \geq 0)$ be a sequence of positive numbers such that
	\[
		\sum_{k=0}^\infty \gamma_k = \infty, 
		\qquad\text{and}\qquad
		\lim_{n \to \infty} \frac{\gamma_{n-1}}{\gamma_n} = 1.
	\]
	Assume that $(\theta_n : n \geq 0)$ is a sequence of continuous functions on some open set $U \subset \RR^d$ with
	values in $(0, 2\pi)$. Suppose that there is $\theta: U \rightarrow (0, 2\pi)$ such that
	\[
		\lim_{n \to \infty} \theta_n(x) = \theta(x)
	\]
	locally uniformly with respect to $x \in U$. Let $(r_n : n \in \NN)$ be a sequence of positive numbers such that
	\[
		\lim_{n \to \infty} r_n = \infty.
	\]
	For $x \in U$, and $a, b \in \RR$, we set
	\[
		x_n = x + \frac{a}{r_n}, \qquad\text{and}\qquad
		y_n = x + \frac{b}{r_n}.
	\]
	Then for each compact subset $K \subset U$, $L > 0$, and any function $\sigma: U \rightarrow \RR$,
	\begin{equation}
		\label{eq:17}
		\lim_{n \to \infty}
		\sum_{k=0}^n \frac{\gamma_k}{\sum_{j=0}^n \gamma_j}
        \cos \Big(\sigma(x_n) + \sigma(y_n) + \sum_{j=0}^k \big( \theta_j(x_n) + \theta_j(y_n) \big) \Big)
		=0
	\end{equation}
	uniformly with respect to $x \in K$, and $a, b \in [-L, L]$.
\end{lemma}
\begin{proof}
	Let us fix $K$ a compact subset of $U$ and $L > 0$. Select $R > 0$ so that
	\[
		\bigg(x - \frac{2L}{R}, x + \frac{2L}{R} \bigg) \subset U
	\]
	for all $x \in K$, and let $N \in \NN$ be such that $r_n \geq R$ for all $n \geq N$.
	For $(x, a, b) \in U \times (-2L, 2L)^2$, we set
	\[
		\tilde{\theta}_j(x, a, b) = 
		\theta_j \bigg( x + \frac{a}{R} \bigg) + 
		\theta_j \bigg( x + \frac{b}{R} \bigg),
		\qquad\text{and}\qquad
		\tilde{\theta}(x, a, b) = 
		\theta \bigg( x + \frac{a}{R} \bigg) +
		\theta \bigg( x + \frac{b}{R} \bigg).
	\]
	Then
	\[
		\lim_{j \to \infty} \tilde{\theta}_j(x, a, b) = \tilde{\theta}(x, a, b)
	\]
	uniformly with respect to $(x, a, b) \in K \times [-L, L]^2$. By Lemma \ref{lem:9}, there is $c > 0$ such that
	\begin{equation}
		\label{eq:28}
		\bigg| 
		\sum_{k = 0}^n \gamma_k 
		\exp \Big( i \sum_{j=0}^k \tilde{\theta}_j(x, a, b) \Big)
		\bigg|
		\leq
		c
		\sum_{k = 0}^{n-1} 
		\big|\gamma_{k+1} - \gamma_{k} \big|+
		\gamma_{k+1} \big|\tilde{\theta}_{k+1}(x, a, b) - \tilde{\theta}(x, a, b) \big|
	\end{equation}
	for all $x \in K$, $a, b \in [-L, L]$, and $n \in \NN$. Since
	\[
		\tilde{\theta}_j\bigg(x, \frac{Ra}{r_n}, \frac{Rb}{r_n}\bigg) = \theta_j(x_n) + \theta_j(y_n),
	\]
	by \eqref{eq:28},
	\begin{equation}
		\label{eq:40}
		\begin{aligned}
		&\bigg|
		\sum_{k = 0}^n \gamma_{k}
		\exp \Big( i \sum_{j=0}^k \theta_j(x_n) + \theta_j(y_n) \Big)
		\bigg| \\
		&\qquad\qquad
		\leq
		c
		\sum_{k = 0}^{n-1} 
		\big|\gamma_{k+1} - \gamma_{k} \big|
		+
		\gamma_{k+1} \Big(
		\big|\theta_{k+1}(x_n) - \theta(x_n)\big|
		+
		\big|\theta_{k+1}(y_n) - \theta(y_n)\big|
		\Big)
		\end{aligned}
	\end{equation}
	for all $x \in K$, $a, b \in [-L, L]$, and $n \geq N$. Finally,
	\begin{align*}
		&\cos \Big( \sum_{j=0}^k \big( \theta_j(x_n) + \theta_j(y_n) \big) + 
		\sigma(x_n) + \sigma(y_n) \Big) \\
		&\qquad\qquad=
		\cos \Big( \sum_{j=0}^k \big( \theta_j(x_n) + \theta_j(y_n) \big) \Big)
		\cos \Big( \sigma(x_n) + \sigma(y_n) \Big) \\
		&\qquad\qquad\phantom{=}-
		\sin \Big( \sum_{j=0}^k \big( \theta_j(x_n) + \theta_j(y_n) \big) \Big)
		\sin \Big( \sigma (x_n) + \sigma(y_n) \Big),
	\end{align*}
	which together with \eqref{eq:40} implies that there are $c > 0$ and $N \in \NN$, such that for any
	function $\sigma: U \rightarrow \RR$ and all $x \in K$, $a, b \in [-L, L]$ and $n \geq N$,
	\begin{equation}
		\label{eq:92}
		\begin{aligned}
		&
		\bigg|
		\sum_{k=0}^n \gamma_k
		\cos \Big(\sigma(x_n) + \sigma(y_n) + \sum_{j=0}^k \big( \theta_j(x_n) + \theta_j(y_n) \big) \Big)
		\bigg| \\
		&\qquad\qquad
		\leq
		c
		\sum_{k = 0}^{n-1}
		\big|\gamma_{k+1} - \gamma_k \big|
		+
		\gamma_{k+1} 
		\Big(
		\big|\theta_{k+1}(x_n) - \theta(x_n)\big|
		+
		\big|\theta_{k+1}(y_n) - \theta(y_n)\big|
		\Big).
		\end{aligned}
	\end{equation}
	Finally, \eqref{eq:17} follows from \eqref{eq:92} by the Stolz--Ces\`aro theorem, since
	\[
		\lim_{n \to \infty} \frac{1}{\sum_{k=0}^n \gamma_k} \sum_{k = 0}^{n-1} \big|\gamma_{k+1} - \gamma_k \big|
		=
		\lim_{n \to \infty} \frac{\big|\gamma_n - \gamma_{n-1}\big|}{\gamma_n} = 0,
	\]
	and
	\begin{align*}
		&\lim_{n \to \infty} \frac{1}{\sum_{k = 0}^n \gamma_k} 
		\sum_{k=0}^n \gamma_{k+1} \Big(
		\big|\theta_{k+1}(x_n) - \theta(x_n)\big|
		+
		\big|\theta_{k+1}(y_n) - \theta(y_n)\big|
		\Big) \\
		&\qquad\qquad=
		\lim_{n \to \infty}
		\big|\theta_n(x_n) - \theta(x_n)\big| + \big|\theta_n(y_n) - \theta(y_n)\big| = 0. \qedhere
	\end{align*}
\end{proof}

\begin{theorem}
	\label{thm:3}
	Let $U$ be an open subset of $\RR$. Let $(\gamma_k : k \geq 0)$ be a sequence of positive 
	numbers such that
	\begin{equation}
		\label{thm:3:eq:1}
		\sum_{k=0}^\infty \gamma_k = \infty, \qquad\text{and}\qquad \lim_{k \to \infty} \frac{\gamma_{k-1}}{\gamma_k} = 1.
	\end{equation}
	Assume that $(\theta_k : k \geq 0)$ is a sequence of $\calC^2(U)$ functions with values in $(0, 2 \pi)$
	such that for each compact set $K \subset U$ there are functions $\theta: K \rightarrow (0, 2\pi)$ and
	$\psi: K \rightarrow (0, \infty)$, and $c > 0$ such that
	\begin{enumerate}[(a)]
	\item 
		$\begin{aligned}[b] 
			\lim_{n \to \infty} \sup_{x \in K}{\big|\theta_n(x) - \theta(x)\big|} = 0,
		\end{aligned}$
	\item
		\label{eq:35}
		$\begin{aligned}[b]
			\lim_{n \to \infty} \sup_{x \in K}{\big| \gamma_n^{-1} \cdot \theta_n'(x) - \psi(x) \big|} =0,
		\end{aligned}$
	\item
		\label{eq:34}
		$\begin{aligned}[b]
			\sup_{n \in \NN} \sup_{x \in K}{\big|\gamma_n^{-2} \cdot \theta_n''(x)\big|} \leq c.
		\end{aligned}$
	\end{enumerate}
	For $x \in U$ and $a, b \in \RR$ we set
	\[
		x_n = x + \frac{a}{\sum_{k=0}^n \gamma_k}, \qquad\text{and}\qquad
		y_n = x + \frac{b}{\sum_{k=0}^n \gamma_k}.
	\]
	Then for any continuous function $\sigma: U \rightarrow \RR$,
	\begin{equation}
		\label{eq:44}
		\begin{aligned}
		&\lim_{n \to \infty} \sum_{k=0}^n \frac{\gamma_k}{\sum_{j=0}^n \gamma_j}
		\sin \Big( \sum_{j=0}^n \theta_j(x_n) + \sigma(x_n) \Big)
		\sin \Big( \sum_{j=0}^n \theta_j(y_n) + \sigma(y_n) \Big) \\
		&\qquad\qquad=
		\frac{1}{2} \sinc \big( (b-a) \psi(x) \big)
		\end{aligned}
	\end{equation}
	locally uniformly with respect to $x \in U$, and $a, b \in \RR$.
\end{theorem}
\begin{proof}
	Let us fix a compact set $K \subset U$ and $L > 0$. We write
	\begin{equation}
		\label{eq:10}
		\begin{aligned}
		&2\cdot \sin \Big( \sum_{j=0}^n \theta_j(x_n) + \sigma(x_n) \Big)
		\sin \Big( \sum_{j=0}^n \theta_j(y_n) + \sigma(y_n) \Big)\\
		&\qquad\qquad=
		\cos \Big( \sum_{j=0}^n \big( \theta_j(x_n) - \theta_j(y_n) \big) + 
		\sigma(x_n) - \sigma(y_n) \Big) \\
		&\qquad\qquad\phantom{=}
		-
		\cos \Big( \sum_{j=0}^n \big( \theta_j(x_n) + \theta_j(y_n) \big) + 
		\sigma(x_n) + \sigma(y_n) \Big). 
		\end{aligned}
	\end{equation}
	In view of Corollary \ref{cor:2}, the second term has no contribution to the limit \eqref{eq:44}. To deal with
	the first term in \eqref{eq:10}, we write
	\begin{equation}
		\label{eq:43}
		\begin{aligned}
		&\cos \Big( \sum_{j=0}^n \big( \theta_j(x_n) - \theta_j(y_n) \big) + 
		\sigma(x_n) - \sigma(y_n) \Big) \\
		&\qquad\qquad=
		\cos \Big( \sum_{j=0}^n \theta_j(x_n) - \theta_j(y_n) \Big) 
		\cos \Big( \sigma(x_n) - \sigma(y_n) \Big) \\
		&\qquad\qquad\phantom{=}-
		\sin \Big( \sum_{j=0}^n \theta_j(x_n) - \theta_j(y_n) \Big) 
		\sin \Big( \sigma(x_n) - \sigma(y_n) \Big).
		\end{aligned}
	\end{equation}
	Now, by the continuity of $\sigma$, the second term in \eqref{eq:43} has no contribution to the limit \eqref{eq:44}.
	Hence, it is enough to show that
	\begin{align*}
		&
		\lim_{n \to \infty}
		\sum_{k=0}^n \frac{\gamma_k}{\sum_{j=0}^n \gamma_j}
		\cos \Big( \sum_{j=0}^n \big( \theta_j(x_n) - \theta_j(y_n) \big) \Big)
		=
		\sinc \big( (b-a) \psi(x) \big)
	\end{align*}
	uniformly with respect to $x \in K$, and $a, b \in [-L, L]$. We first prove the following claim.
	\begin{claim}
		\label{clm:1}
		There is $c > 0$ such that for all $j \in \NN$ and $n \in \NN$,
		\begin{equation}
			\label{eq:8}
			\bigg|
			\theta_j(y_n) - \theta_j(x_n) - (b-a) \theta_j'(x) \Big(\sum_{\ell = 0}^n \gamma_\ell\Big)^{-1}
			\bigg|
			\leq
			c \Big( \sum_{\ell = 0}^n \gamma_\ell\Big)^{-2} \sup_{x \in K} \abs{\theta''_j(x)}.
		\end{equation}
	\end{claim}
	For the proof, let us write Taylor's polynomial for $\theta_j$ centered at $x$, that is,
	\[
		\theta_j(y) = \theta_j(x) + \theta_j'(x)(y-x) + E_j(x; y)
	\]
	where
	\[
		\abs{E_j(x; y)} \leq \frac{1}{2} \abs{y - x}^2 \sup_{w \in [x, y]}{\big|\theta_j''(w)\big|}.
	\]
	Therefore,
	\[
		\theta_j(y_n) - \theta_j(x_n) = \theta_j'(x)(y_n-x_n) + E_j(x; y_n) - E_j(x; x_n),
	\]
	which leads to \eqref{eq:8}.

	Let us now observe that, by the mean value theorem and Claim \ref{clm:1}, we have
	\begin{align*}
		&\Bigg|
		\cos\Big(\sum_{j = 0}^k \theta_j(y_n) - \theta_j(x_n) \Big)
		- \cos\Big((b-a) \Big(\sum_{\ell = 0}^n \gamma_\ell\Big)^{-1} \sum_{j = 0}^k \theta_j'(x) \Big)
		\Bigg|\\
		&\qquad\qquad
		\leq \sum_{j = 0}^k \Big|\theta_j(y_n) - \theta_j(x_n)
  	 	- (b-a) \Big(\sum_{\ell = 0}^n \gamma_\ell\Big)^{-1} \theta_j'(x) 
		\Big| \\
		&\qquad\qquad
		\leq
		c \Big(\sum_{\ell = 0}^n \gamma_\ell\Big)^{-2} 
		\sum_{j = 0}^k \sup_{x \in K} {\abs{\theta''_j(x)}}.
	\end{align*}
	Hence,
	\begin{align*}
		&\left|
		\sum_{k = 0}^n \frac{\gamma_k}{\sum_{j = 0}^n \gamma_j}
		\Bigg(
		\cos\Big(\sum_{j = 0}^k \theta_j(x_n) - \theta_j(y_n) \Big) 
		- \cos\Big((b-a) \Big(\sum_{\ell = 0}^n \gamma_\ell\Big)^{-1}
		\sum_{j = 0}^k \theta_j'(x) \Big)
		\Bigg)
		\right|\\
		&\qquad\qquad
		\leq
		c
		\Big(\sum_{\ell = 0}^n \gamma_\ell\Big)^{-3} \sum_{k = 0}^n \gamma_k 
		\sum_{j = 0}^k \sup_{x \in K}{ \abs{\theta''_j(x)}}.
	\end{align*}
	Now, by the Stolz--Ces\`aro theorem, we have
	\[
		\lim_{n \to \infty} \frac{\gamma_n}{\sum_{j = 0}^n \gamma_j} 
		=
		\lim_{n \to \infty} \frac{\gamma_n-\gamma_{n-1}}{\gamma_n} = 0,
	\]
	thus, by repeated application of the Stolz--Ces\`aro theorem we obtain
	\begin{align*}
		\lim_{n \to \infty} \Big(\sum_{\ell = 0}^n \gamma_\ell\Big)^{-3} \sum_{k = 0}^n \gamma_k
		\sum_{j = 0}^k \sup_{x \in K}{ \abs{\theta''_j(x)}}
		&=
		\frac{1}{3} 
		\lim_{n \to \infty}
		\Big(\sum_{\ell = 0}^n \gamma_\ell\Big)^{-2}
		\sum_{j = 0}^n \sup_{x \in K}{ \abs{\theta''_j(x)}} \\
		&=
		\frac{1}{6}
		\lim_{n \to \infty}
		\gamma_n^{-1} \Big(\sum_{\ell = 0}^n \gamma_\ell\Big)^{-1} 
		\sup_{x \in K}{ \abs{\theta''_n(x)}}.
	\end{align*}
	In view of \eqref{eq:34}, it is enough to show that
	\[
		\lim_{n \to \infty}
		\sum_{k = 0}^n 
		\frac{\gamma_k}{\sum_{j = 0}^n \gamma_j}
		\cos\Big((b-a) \Big(\sum_{\ell = 0}^n \gamma_\ell\Big)^{-1}
	    \sum_{j = 0}^k \theta_j'(x) \Big)
		=
		\sinc \big( (b-a) \psi(x) \big).
	\]
	For the proof, we write
	\begin{align*}
		&
		\sum_{k=0}^n \frac{\gamma_k}{\sum_{j =0}^n \gamma_j} \cos\Big((b-a) \Big(\sum_{\ell=0}^n \gamma_\ell \Big)^{-1}
	    \sum_{j = 0}^k \theta'_j(x) \Big) \\
		&\qquad\qquad=
		\sum_{k = 0}^n
		\sum_{m = 0}^\infty
		\frac{(-1)^m}{(2m)!} (b-a)^{2m} \gamma_k \Big(\sum_{\ell=0}^n \gamma_\ell \Big)^{-2m-1}
	    \Big(\sum_{j = 0}^k \theta'_j(x) \Big)^{2m} \\
		&\qquad\qquad
		=
		\sum_{m = 0}^\infty
		\frac{(-1)^m}{(2m)!} (b-a)^{2m}
		\Big(\sum_{\ell=0}^n \gamma_\ell \Big)^{-2m-1}
		\sum_{k = 0}^n \gamma_k
		\Big(\sum_{j = 0}^k \theta'_j(x) \Big)^{2m}.
	\end{align*}
	We now claim the following.
	\begin{claim}
		For each $m \in \NN$,
		\label{clm:3}
		\begin{equation}
			\label{eq:2}
			\lim_{n \to \infty}
			\Big(\sum_{\ell=0}^n \gamma_\ell \Big)^{-2m-1}
   			\sum_{k = 0}^n \gamma_k
   			 \Big(\sum_{j = 0}^k \theta'_j(x) \Big)^{2m}
			=
			\frac{1}{2m+1} \big( \psi(x) \big)^{2m}.
		\end{equation}
	\end{claim}
	By \eqref{eq:35} and the Stolz--Ces\`aro theorem, we get
	\[
		\lim_{n \to \infty}
		\frac{\sum_{j = 0}^n \theta'_j(x)}{\sum_{\ell=0}^n \gamma_\ell} = \psi(x).
	\]
	Since
	\begin{align*}
		\frac{1}{2m+1}
		\bigg(
		\frac{\sum_{j = 0}^n \theta'_j(x)}{\sum_{\ell=0}^n \gamma_\ell}\bigg)^{2m}
		&\leq
		\frac{\gamma_n \big(\sum_{j = 0}^n \theta'_j(x)\big)^{2m}}
		{\big(\sum_{\ell=0}^n \gamma_\ell\big)^{2m+1} - \big(\sum_{\ell=0}^{n-1} \gamma_\ell\big)^{2m+1}} \\
		&\leq
		\frac{1}{2m+1}
		\bigg(
		\frac{\sum_{j = 0}^n \theta'_j(x)}{\sum_{\ell=0}^{n-1} \gamma_\ell}
		\bigg)^{2m},
	\end{align*}
	we get
	\[
		\lim_{n \to \infty}
		\frac{\gamma_n \big(\sum_{j = 0}^n \theta'_j(x)\big)^{2m}}
   		{\big(\sum_{\ell=0}^n \gamma_\ell\big)^{2m+1} - \big(\sum_{\ell=0}^{n-1} \gamma_\ell\big)^{2m+1}}
		=
		\frac{1}{2m+1} \big( \psi(x) \big)^{2m}.
	\]
	Therefore, another application of the Stolz--Ces\`aro theorem leads to \eqref{eq:2}.
	
	Let us notice that for some $c > 0$,
	\[
		\sum_{k = 0}^n \gamma_k \Big(\sum_{j = 0}^k \theta'_j(x) \Big)^{2m}
		\leq
		c^{2m} \Big(\sum_{k = 0}^n \gamma_k\Big)^{2m+1},
	\]
	thus, we have the estimate
	\[
		\bigg|
		\frac{(-1)^m}{(2m)!} (b-a)^{2m}
		\Big(\sum_{\ell=0}^n \gamma_\ell \Big)^{-2m-1}
    	\sum_{k = 0}^n \gamma_k
    	\Big(\sum_{j = 0}^k \theta'_j(x) \Big)^{2m}
		\bigg|
		\leq
		\frac{1}{(2m)!} (b-a)^{2m} c^{2m}.
	\]
	Hence, by the dominated convergence theorem and Claim \ref{clm:3}, we can compute
	\begin{align*}
		&\lim_{n \to \infty}
		\sum_{m = 0}^\infty
	    \frac{(-1)^m}{(2m)!} (b-a)^{2m}
 		\Big(\sum_{\ell=0}^n \gamma_\ell \Big)^{-2m-1}
		\sum_{k = 0}^n \gamma_k
		\Big(\sum_{j = 0}^k \theta'_j(x) \Big)^{2m} \\
		&\qquad\qquad=
		1 + 
		\sum_{m = 1}^\infty \frac{(-1)^m}{(2m+1)!} \big((b-a) \psi(x) \big)^{2m}  \\
		&\qquad\qquad=
		\sinc \big((b-a) \psi(x) \big),
	\end{align*}
	which finishes the proof of the theorem.
\end{proof}

\subsection{Christoffel--Darboux kernel}
Let us recall the definition 
\[
	\theta_n(x) = \arccos\bigg(\frac{\tr X_n(x)}{2 \sqrt{\det X_n(x)}} \bigg).
\]
\begin{proposition} 
	\label{prop:6}
	Let $A$ be a Jacobi matrix with $N$-periodically modulated entries. Then for every compact subset $K \subset \RR$, 
	we have
	\begin{equation}
		\label{eq:18}
		\lim_{n \to \infty} \sup_{x \in K}
		\left| \frac{a_{n}}{\alpha_{n}} \theta_n'(x) + 
		\frac{\tr \frakX_0'(0)}{N \sqrt{-\discr \frakX_0(0)}} \right| = 0,
	\end{equation}
	and
	\begin{equation}
		\label{eq:19}
		\sup_{x \in K} |\theta_n''(x)| \leq c \Big( \frac{\alpha_n}{a_{n}} \Big)^2
	\end{equation}
	for some $c > 0$.
\end{proposition}
\begin{proof}
	Let us fix $i \in \{ 0, 1, \ldots, N-1 \}$. Since
	\begin{equation} \label{eq:107}
		\det X_{kN+i}(x) = \frac{a_{kN+i-1}}{a_{(k+1)N+i-1}},
	\end{equation}
	by Proposition~\ref{prop:5}, we conclude that
	\begin{equation}
		\label{eq:95}
		\lim_{k \to \infty} \det X_{kN+i}(x) = 1.
	\end{equation}
	The chain rule applied to \eqref{eq:38} leads to
	\begin{align*}
		\theta_{kN+i}'(x) 
		&=
		-\Bigg(4 - \bigg( \frac{\tr X_{kN+i}(x)}{\sqrt{\det X_{kN+i}(x)}} \bigg)^2 \Bigg)^{-1/2}
		\frac{\tr X_{kN+i}'(x)}{\sqrt{\det X_{kN+i}(x)}} \\
		&=
		-\frac{\tr X_{kN+i}'(x)}{\sqrt{-\discr X_{kN+i}(x)}},
	\end{align*}
	thus, by \eqref{eq:95} and Corollary \ref{cor:5}, we obtain \eqref{eq:18}. Consequently, in view of \eqref{eq:107},
	\[
		\theta_{kN+i}''(x) 
		= 
		-\frac{\tr X_{kN+i}''(x)}{\sqrt{-\discr X_{kN+i}(x)}} - 
		\frac{\big(\tr X_{kN+i}'(x)\big)^2 \tr X_{kN+i}(x)}{\big( -\discr X_{kN+i}(x) \big)^{3/2}}.
	\]
	Therefore, the estimate \eqref{eq:19} is a consequence of Corollary~\ref{cor:5}. 
\end{proof}

\begin{theorem}
	\label{thm:1}
	Let $A$ be a Jacobi matrix with $N$-periodically modulated entries. Suppose that for each
	$i \in \{0, 1, \ldots, N-1\}$,
	\begin{equation}
		\label{eq:41}
		\bigg(\frac{a_{jN+i-1}}{a_{jN+i}} : j \in \NN \bigg),
		\bigg(\frac{b_{jN+i}}{a_{jN+i}} : j \in \NN_0 \bigg),
		\bigg(\frac{1}{a_{jN+i}} : j \in \NN_0 \bigg) \in \calD_1
	\end{equation}
	and
	\begin{equation}
		\label{eq:50}
		\sum_{j = 0}^\infty \frac{1}{a_j} = \infty.
	\end{equation}
	If $\abs{\tr \frakX_0(0)} < 2$, then
	\[
		\lim_{n \to \infty}
		\frac{1}{\rho_n} K_n\bigg(x+\frac{u}{\rho_n}, x+\frac{v}{\rho_n} \bigg)
		=
		\frac{\omega'(0)}{\mu'(x)} \cdot
		\sinc \big((u-v) \pi \omega'(0)\big)
	\]
	locally uniformly with respect to $x, u, v \in \RR$, where 
	\[
		\rho_n = \sum_{j = 0}^n \frac{\alpha_j}{a_j},
	\]
	and $\omega$ is the equilibrium measure corresponding to
	$\sigma_{\textrm{ess}}(\frakA)$ with $\frakA$ being the Jacobi matrix associated to
	$(\alpha_n : n \in \NN_0)$ and $(\beta_n : n \in \NN_0)$.
\end{theorem}
\begin{proof}
	Let us fix a compact set $K \subset \Lambda$ with non-empty interior and $L > 0$. Let $\tilde{K} \subset \Lambda$ be a
	compact set containing $K$ in its interior. There is $n_0 > 0$ such that for all $x \in K$, $n \geq n_0$,
	$u \in [-L, L]$, and $i \in \{0, 1, \ldots, N-1\}$,
	\[
		x + \frac{u}{N \alpha_i \rho_{i; n}}, 
		x + \frac{u}{\rho_{nN+i}} \in \tilde{K}.
	\]
	Given $u, v \in [-L, L]$, we set
	\begin{align*}
		x_{i; n} &= x + \frac{u}{N \alpha_i \rho_{i; n}}, &\qquad x_{nN+i} &= x + \frac{u}{\rho_{nN+i}}, \\
		y_{i; n} &= x + \frac{v}{N \alpha_i \rho_{i; n}}, &\qquad y_{nN+i} &= x + \frac{v}{\rho_{nN+i}}.
	\end{align*}
	Remark \ref{rem:2} together with \eqref{eq:41} entails that $(X_{jN+i} : j \in \NN)$ belongs to $\calD_1$. Moreover,
	\[
		\lim_{k \to \infty} X_{kN+i}(x) = \frakX_i(0)
	\]
	uniformly with respect to $x \in K$. Let us recall that
	\[
		\discr \frakX_i(x) = \discr \frakX_0(x).
	\]
	In view of \eqref{eq:6}, the Carleman condition \eqref{eq:50} implies that
	\[
		\lim_{k \to \infty} \rho_{i; k} = \infty.
	\]
	Hence, by Theorem \ref{thm:14}, there are $c > 0$ and $M \in \NN$ such that for all $x, y \in \tilde{K}$ and 
	$k \geq M$,
	\begin{equation}
		\label{eq:47}
		\begin{aligned}
		&p_{kN+i}(x) p_{kN+i}(y)
		=
		\frac{|[\frakX_i(0) ]_{2,1}|}{\pi \sqrt{-\discr \frakX_0(y)}}
		\cdot
		\frac{2}{\sqrt{\mu'(x) \mu'(y)}} \\
		&\qquad\phantom{=}
		\times
		\frac{1}{a_{(k+1)N+i-1}}
		\sin \Big( \sum_{j=M+1}^k \theta_{j N+i}(x) + \eta(x) \Big)
		\sin \Big( \sum_{j=M+1}^k \theta_{j N+i}(y) + \eta(y) \Big) \\
		&\qquad\phantom{=}
		+ E_{kN+i}(x, y)
		\end{aligned}
	\end{equation}
	where
	\[
		\sup_{x, y \in \tilde{K}} \big| E_{kN+i}(x, y) \big| \leq
		c \sum_{j \geq k} \sup_{x \in \tilde{K}}{\big\|X_{(j+1)N+i}(x) - X_{jN+i}(x) \big\|}.
	\]
	Hence, we obtain
	\begin{align*}
		&\sum_{k = M}^n p_{kN+i}(x) p_{kN+i}(y)
		=
		\frac{|[\frakX_i(0)]_{2,1}|}{\pi \sqrt{-\discr \frakX_0(x)}} \cdot
		\frac{2}{\sqrt{\mu'(x) \mu'(y)}} \\
		&\qquad\phantom{=}
		\times
		\sum_{k = M}^n
		\frac{1}{a_{(k+1)N+i-1}} 
		\sin \Big( \sum_{j=M+1}^k \theta_{j N+i}(x) + \eta(x) \Big)
	    \sin \Big( \sum_{j=M+1}^k \theta_{j N+i}(y) + \eta(y) \Big) \\
		&\qquad\phantom{=}+ 
		\sum_{k = M}^n \frac{1}{a_{(k+1)N+i-1}} E_{kN+i}(x, y).
	\end{align*}
	Let
	\[
		\gamma_k = \frac{N\alpha_{i-1}}{a_{kN+i-1}},
		\qquad\text{and}\qquad
		\psi(x) = -\frac{\tr \frakX_0'(0)}{N \sqrt{-\discr \frakX_0(0)}}.
	\]
	By Proposition \ref{prop:6}
	\[
		\lim_{n \to \infty} \sup_{x \in \tilde{K}}{\bigg|\frac{a_{kN+i-1}}
		{\alpha_{i-1}} \theta'_{kN+i} (x) - \psi(x) \bigg|} = 0.
	\]
	Moreover, by \eqref{eq:4} and \eqref{eq:82} 
	\[
		|\psi(x)| = \pi \omega'(0).
	\]
	Therefore, by Theorem \ref{thm:3} we get 
	\begin{align*}
		&
		\lim_{n \to \infty}
		\frac{1}{N \alpha_i \rho_{i; n}}
		\sum_{k = M}^n \frac{N \alpha_{i-1}}{a_{kN+i-1}} 
		\sin \Big( \sum_{j=M+1}^k \theta_{j N+i}(x_{i-1; n}) + \eta(x_{i-1; n} ) \Big)
		\sin \Big( \sum_{j=M+1}^k \theta_{j N+i}(y_{i-1; n}) + \eta(y_{i-1; n} ) \Big) \\
		&\qquad\qquad=
		\sinc \big((u-v) \pi \omega'(0) \big).
	\end{align*}
	Now, by uniformness and \eqref{eq:6}, for any $i' \in \{0, 1, \ldots, N-1\}$, we obtain
	\begin{align}
		\nonumber
		&\lim_{n \to \infty} \frac{N \alpha_{i-1}}{\rho_{nN+i}} K_{i; n}(x_{nN+i'}, y_{nN+i'}) \\
		\nonumber
		&\qquad\qquad=
		\frac{|[\frakX_i(0) ]_{2,1}|}{\pi \mu'(x) \sqrt{-\discr \frakX_0(0)}}
		\cdot
		\lim_{n \to \infty} 
		\sinc\bigg((u-v) \frac{\rho_{nN+i'}}{N\alpha_{i-1}\rho_{i-1; n}} \pi \omega'(0) \bigg) \\
		\label{eq:42}
		&\qquad\qquad=
		\frac{|[\frakX_i(0) ]_{2,1}|}{\pi \mu'(x) \sqrt{-\discr \frakX_0(0)}}
		\cdot
		\sinc \big((u-v) \pi \omega'(0) \big)
	\end{align}
	uniformly with respect to $x \in K$ and $u, v \in [-L, L]$. Here, we have also used that
	\begin{equation}
		\label{eq:39}
		\sup_{m \in \NN}{\sup_{x \in \tilde{K}}{\abs{p_m(x)}}} \leq c.
	\end{equation}
	Since
	\[
		K_{nN+i}(x, y) = \sum_{i' = 0}^{N-1} K_{i'; n}(x, y) 
		+ \sum_{i' = i+1}^{N-1} \big(K_{i'; n-1}(x, y) - K_{i'; n}(x, y) \big),
	\]
	by \eqref{eq:39} and \eqref{eq:42}, we obtain
	\begin{align*}
		&\lim_{n \to \infty} \frac{1}{\rho_{nN+i}} K_{nN+i}(x_{nN+i}, y_{nN+i}) \\
		&\qquad\qquad=
		\frac{1}{N}
		\sum_{i' = 0}^{N-1} \frac{N \alpha_{i'-1}}{\rho_{nN+i}} K_{i'; n}(x_{nN+i}, y_{nN+i}) \frac{1}{\alpha_{i'-1}} \\
		&\qquad\qquad=
		\sinc \big((u-v) \pi \omega'(0) \big)
		\cdot
		\frac{1}{N \pi \mu'(x) \sqrt{-\discr \frakX_0(0)}}
		\sum_{i' = 0}^{N-1}
		\frac{|[\frakX_{i'}(0) ]_{2,1}|}{\alpha_{i'-1}}
	\end{align*}
	which together with \eqref{eq:4} concludes the proof.
\end{proof}

\begin{proposition}
	\label{prop:10}
	Let $A$ be a Jacobi matrix that is $N$-periodic blend. Let $K$ be a non-empty compact interval contained in
	\[
		\Lambda = \big\{x \in \RR : \abs{\tr \calX_1(x)} < 2 \big\}
	\]
	where $\calX_1$ is the limit of $(X_{j(N+2)+1} : j \in \NN_0)$. Then for each $i \in \{1, 2, \ldots, N\}$,
	\begin{equation}
		\label{eq:83}
		\lim_{n \to \infty} 
		\sup_{x \in K}{
		\bigg|
		\frac{a_{n(N+2)+i}}{\alpha_i} \theta'_{n(N+2)+i}(x) + \frac{\tr \calX_1'(x)}{N \sqrt{-\discr \calX_1(x)}}
		\bigg|}
		=0,
	\end{equation}
	and
	\begin{equation}
		\label{eq:84}
		\sup_{x \in K}{\abs{\theta''_{n(N+2)+i}(x)}} \leq c \bigg(\frac{\alpha_i}{a_{n(N+2)+i}}\bigg)^2
	\end{equation}
	for some $c > 0$.
\end{proposition}
\begin{proof}
	Let us fix $i \in \{1, 2, \ldots, N\}$. Since
	\begin{equation} \label{eq:108}
		\det X_{k(N+2)+i}(x) = \frac{a_{k(N+2)+i-1}}{a_{(k+1)(N+2)+i-1}},
	\end{equation}
	we conclude that
	\begin{equation}
		\label{eq:15}
		\lim_{k \to \infty} \det X_{k(N+2)+i}(x) = 1.
	\end{equation}
	The chain rule applied to \eqref{eq:38} leads to
	\begin{align*}
		\theta_{k(N+2)+i}'(x) &=
		-\Bigg(4 - \bigg( \frac{\tr X_{k(N+2)+i}(x)}{\sqrt{\det X_{k(N+2)+i}(x)}} \bigg)^2 \Bigg)^{-1/2}
		\frac{\tr X_{k(N+2)+i}'(x)}{\sqrt{\det X_{k(N+2)+i}(x)}} \\
		&=
		-\frac{\tr X_{k(N+2)+i}'(x)}{\sqrt{-\discr X_{k(N+2)+i}(x)}}
	\end{align*}
	thus, by \eqref{eq:15} and Corollary \ref{cor:3}, we obtain \eqref{eq:83}. Consequently, in view of \eqref{eq:108},
	\[
		\theta_{k(N+2)+i}''(x) 
		= 
		- \frac{\tr X_{k(N+2)+i}''(x)}{\sqrt{-\discr X_{k(N+2)+i}(x)}}
		- \frac{\big(\tr X_{k(N+2)+i}'(x)\big)^2 \tr X_{k(N+2)+i}(x)}{\big( -\discr X_{k(N+2)+i}(x) \big)^{3/2}}.
	\]
	Therefore, the estimate \eqref{eq:84} is a consequence of Corollary~\ref{cor:3}. 
\end{proof}

\begin{theorem}
	\label{thm:15}
	Let $A$ be a Jacobi matrix that is $N$-periodic blend. Suppose that for each $i \in \{0, 1, \ldots, N-1\}$,
	\[
		\bigg(\frac{1}{a_{j(N+2)+i}} : j \in \NN_0\bigg),
		\bigg(\frac{b_{j(N+2)+i}}{a_{j(N+2)+i}} : j \in \NN_0\bigg)
		\in \calD_1
	\]
	and
	\[
		\bigg(\frac{1}{a_{j(N+2)+N}} : j \in \NN_0 \bigg), 
		\bigg(\frac{1}{a_{j(N+2)+N+1}} : j \in \NN_0\bigg),
		\bigg(\frac{a_{j(N+2)+N}}{a_{j(N+2)+N+1}} : j \in \NN_0\bigg) \in \calD_1.
	\]
	Let $K$ be a compact interval with non-empty interior contained in
	\[
		\Lambda = \big\{x \in \RR : \abs{\tr \calX_1 (x)} < 2\big\}
	\]
	where $\calX_1$ is the limit of $(X_{j(N+2)+1} : j \in \NN_0)$. Then
	\[
		\lim_{n \to \infty} 
		\frac{1}{\rho_n} K_n\bigg(x+\frac{u}{\rho_n}, x + \frac{v}{\rho_n}\bigg)
		=
		\frac{\omega'(x)}{\mu'(x)} \cdot \sinc \big((u-v) \pi \omega'(x)\big)
	\]
	locally uniformly with respect to $x \in K$, and $u, v \in \RR$, where $\omega$ is the equilibrium measure
	corresponding to $\overline{\Lambda}$, and
	\[
		\rho_n = \sum_{i = 0}^{N-1} \sum_{\stackrel{m = 0}{m \equiv i \bmod{(N+2)}}}^n \frac{\alpha_m}{a_m}.
	\]
\end{theorem}
\begin{proof}
	Let $K \subset \Lambda$ be a compact interval with non-empty interior and let $L > 0$. Let $\tilde{K} \subset \Lambda$
	be a compact set containing $K$ in its interior. There is $n_0 > 0$ such that for all $x \in K$, $n \geq n_0$, 
	$i \in \{0, 1, \ldots, N+1\}$, and $u \in [-L, L]$,
	\[
		x + \frac{u}{\rho_{n(N+2)+i}} \in \tilde{K}.
	\]
	Given $x \in K$ and $u, v \in [-L, L]$, we set
	\[
		x_n = x + \frac{u}{\rho_n},
		\qquad\text{and}\qquad
		y_n = x + \frac{v}{\rho_n}.
	\]
	For each $i \in \{1, 2, \ldots, N\}$ and $i' \in \{0, 1, \ldots, N+1\}$, by the reasoning analogous to the proof of
	Theorem \ref{thm:1} one can show that
	\begin{equation}
		\label{eq:88}
		\begin{aligned}
		&
		\lim_{n \to \infty}
		\frac{N \alpha_{i-1}}{\rho_{i; n}} K_{i; n}(x_{n(N+2)+i'}, y_{n(N+2)+i'}) \\
		&\qquad\qquad=
		\frac{|[\calX_i(x)]_{2,1}|}{\pi \mu'(x) \sqrt{-\discr \calX_i(x)}} 
		\cdot 
		\sinc \big((u-v) \pi \omega'(x) \big)
		\end{aligned}
	\end{equation}
	uniformly with respect to $x \in K$ and $u, v \in [-L, L]$. By Claim \ref{clm:2}, for each
	$i \in \{1, 2, \ldots, N\}$ the sequence $(X_{j(N+2)+i} : j \in \NN_0)$ belongs to $\calD_1\big(K, \GL(2, \RR) \big)$
	and converges to $\calX_i$ satisfying
	\[
		\discr \calX_i = \discr \calX_1.
	\]
	For $i = 0$, by \eqref{eq:73} and Claim \ref{clm:4},
	\begin{align*}
		&\big| K_{0; n}(x, y) - K_{N; n}(x, y) \big|\\
		&\qquad\qquad\leq
		c\Big(\sum_{j = 0}^n \big|p_{j(N+2)}(x) + p_{j(N+2)+N}(x) \big| + \big|p_{j(N+2)}(y) + p_{j(N+2)+N}(y)\big|\Big)\\
		&\qquad\qquad\leq
		c \sum_{j = 0}^k \frac{1}{a_{j(N+2)+N+1}^2} + \bigg|1 - \frac{a_{j(N+2)+N}}{a_{j(N+2)+N+1}} \bigg|.
	\end{align*}
	Therefore,
	\begin{equation}
		\label{eq:86}
		\lim_{n \to \infty} 
		\frac{1}{\rho_{N-1; n}} \sup_{x, y \in \tilde{K}}{\big|K_{0; n}(x, y) - K_{N; n}(x, y)\big|} = 0.
	\end{equation}
	Similarly, for $i = N+1$, one can show
	\begin{equation}
		\label{eq:87}
		\lim_{n \to \infty} \frac{1}{\rho_{N-1; n}} \sup_{x, y \in \tilde{K}}{\big| K_{N+1; n}(x, y) \big|} = 0.
	\end{equation}
	Now, let $i \in \{0, 1, \ldots, N+1\}$. We write
	\[
		K_{n(N+2)+i}(x, y) = \sum_{i' = 0}^{N+1} K_{i'; n}(x, y) 
		+ \sum_{i' = i+1}^{N+1} \big(K_{i'; n}(x, y) - K_{i'; n-1}(x, y)\big).
	\]
	By \eqref{eq:73}, \eqref{eq:86} and \eqref{eq:87}, we obtain
	\[
		K_{n(N+2)+i}(x, y) = \sum_{i' = 1}^{N-1} K_{i'; n}(x, y) + 2 K_{N; n}(x, y) + o\big(\rho_{N-1; n}\big)
	\]
	uniformly with respect to $x, y \in \tilde{K}$. Therefore, by \eqref{eq:88} and \eqref{eq:6},
	\begin{align*}
		&\lim_{n \to \infty} \frac{1}{\rho_{n(N+2)+i}} K_{n(N+2)+i}(x_{n(N+2)+i}, y_{n(N+2)+i}) \\
		&=
		\sum_{i' = 0}^{N-1} \lim_{n \to \infty} 
		\frac{N \alpha_{i'-1}}{\rho_{i'; n}} K_{i'; n}(x_{n(N+2)+i}, y_{n(N+2)+i})
		+2 \lim_{n \to \infty} \frac{N\alpha_{N-1}}{\rho_{N-1; n}} K_{N-1; n}(x_{n(N+2)+i}, y_{n(N+2)+i}) \\
		&=
		\sinc \big((u-v) \pi \omega'(x) \big)
		\cdot
		\frac{1}{\pi \mu'(x) \sqrt{-\discr \calX_1(x)}} 
		\bigg(\sum_{i' = 1}^{N-1} \frac{|[\calX_{i'}(x)]_{2,1}|}{\alpha_{i'-1}} + 
		2 \frac{|[\calX_{N}(x)]_{2,1}|}{\alpha_{N-1}}\bigg)
	\end{align*}
	which together with Theorem \ref{thm:4} finishes the proof.
\end{proof}

\begin{bibliography}{jacobi}
\bibliographystyle{amsplain}

\providecommand{\bysame}{\leavevmode\hbox to3em{\hrulefill}\thinspace}
\providecommand{\MR}{\relax\ifhmode\unskip\space\fi MR }
\providecommand{\MRhref}[2]{%
  \href{http://www.ams.org/mathscinet-getitem?mr=#1}{#2}
}
\providecommand{\href}[2]{#2}
\begin{thebibliography}{10}

\bibitem{Janas2011}
A.~Boutet~de Monvel, J.~Janas, and S.~Naboko, \emph{Unbounded {J}acobi matrices
  with a few gaps in the essential spectrum: constructive examples}, Integ.
  Equat. Oper. Th. \textbf{69} (2011), no.~2, 151--170.

\bibitem{Dombrowski1978}
J.~Dombrowski, \emph{Quasitriangular matrices}, P. Am. Math. Soc. \textbf{69}
  (1978), no.~1, 95--96.

\bibitem{Ignjatovic2016}
A.~Ignjatovi\'c, \emph{Asymptotic behaviour of some families of orthonormal
  polynomials and an associated {H}ilbert space}, J. Approx. Theory
  \textbf{210} (2016), 41--79.

\bibitem{Ignjatovic2017}
A.~Ignjatovic and D.S. Lubinsky, \emph{On an asymptotic equality for
  reproducing kernels and sums of squares of orthonormal polynomials}, Progress
  in approximation theory and applicable complex analysis, Springer Optim.
  Appl., vol. 117, Springer, Cham, 2017, pp.~129--144.

\bibitem{JanasNaboko2002}
J.~Janas and S.~Naboko, \emph{Spectral analysis of selfadjoint {J}acobi
  matrices with periodically modulated entries}, J. Funct. Anal. \textbf{191}
  (2002), no.~2, 318--342.

\bibitem{Kuijlaars2004}
A.B.J. Kuijlaars, K.T.-R. McLaughlin, W.~Van~Assche, and M.~Vanlessen,
  \emph{The {R}iemann--{H}ilbert approach to strong asymptotics for orthogonal
  polynomials on {$[-1,1]$}}, Adv. Math. \textbf{188} (2004), no.~2, 337--398.

\bibitem{Kuijlaars2002}
A.B.J. Kuijlaars and M.~Vanlessen, \emph{Universality for eigenvalue
  correlations from the modified {J}acobi unitary ensemble}, Int. Math. Res.
  Not. \textbf{2002} (2002), no.~30, 1575--1600.

\bibitem{Levin2001}
E.~Levin and D.S. Lubinsky, \emph{Orthogonal polynomials for exponential
  weights}, CMS Books in Mathematics/Ouvrages de Math\'{e}matiques de la SMC,
  vol.~4, Springer-Verlag, New York, 2001.

\bibitem{Levin2008}
\bysame, \emph{Universality limits in the bulk for varying measures}, Adv.
  Math. \textbf{219} (2008), no.~3, 743--779.

\bibitem{Levin2009}
\bysame, \emph{Universality limits for exponential weights}, Constr. Approx.
  \textbf{29} (2009), no.~2, 247--275.

\bibitem{Lubinsky2009}
D.S. Lubinsky, \emph{A new approach to universality limits involving orthogonal
  polynomials}, Ann. Math. \textbf{170} (2009), no.~2, 915--939.

\bibitem{Lubinsky2016}
\bysame, \emph{An update on local universality limits for correlation functions
  generated by unitary ensembles}, SIGMA Symmetry Integrability Geom. Methods
  Appl. \textbf{12} (2016), Paper No. 078, 36.

\bibitem{Mason2002}
J.C. Mason and D.C. Handscomb, \emph{Chebyshev polynomials}, Chapman and
  Hall/CRC, 2002.

\bibitem{Mate1991}
A.~M\'{a}t\'{e}, P.~Nevai, and V.~Totik, \emph{Szeg\"{o}'s extremum problem on
  the unit circle}, Ann. of Math. (2) \textbf{134} (1991), no.~2, 433--453.

\bibitem{Muresan2009}
M.~Mure\c{s}an, \emph{A concrete approach to classical analysis}, CMS Books in
  Mathematics/Ouvrages de Math\'{e}matiques de la SMC, Springer, New York,
  2009.

\bibitem{Saff1997}
E.B. Saff and V.~Totik, \emph{Logarithmic {P}otentials with {E}xternal
  {F}ields}, vol. 316, Springer-Verlag, 1997.

\bibitem{Schmudgen2012}
K.~Schm\"{u}dgen, \emph{Unbounded {S}elf-adjoint {O}perators on {H}ilbert
  {S}pace}, Graduate Texts in Mathematics, vol. 265, Springer, Cham, 2012.

\bibitem{Schmudgen2017}
\bysame, \emph{The moment problem}, Graduate Texts in Mathematics, vol. 277,
  Springer, Cham, 2017.

\bibitem{Simon2008}
B.~Simon, \emph{The {C}hristoffel-{D}arboux kernel}, Perspectives in partial
  differential equations, harmonic analysis and applications, Proc. Sympos.
  Pure Math., vol.~79, Amer. Math. Soc., Providence, RI, 2008, pp.~295--335.

\bibitem{Simon2010Book}
\bysame, \emph{{S}zegő's theorem and its descendants: Spectral theory for
  {$L^2$} perturbations of orthogonal polynomials}, Princeton University Press,
  2010.

\bibitem{StahlTotik1992}
H.~Stahl and V.~Totik, \emph{General orthogonal polynomials}, Encyclopedia of
  Mathematics and its Applications, vol.~43, Cambridge University Press,
  Cambridge, 1992.

\bibitem{Stolz1994}
G.~Stolz, \emph{Spectral theory for slowly oscillating potentials {I}. {J}acobi
  matrices}, Manuscripta Math. (1994), no.~84, 245--260.

\bibitem{PeriodicII}
G.~\'{S}widerski, \emph{Periodic perturbations of unbounded {J}acobi matrices
  {II}: {F}ormulas for density}, J. Approx. Theory \textbf{216} (2017), 67--85.

\bibitem{PeriodicIII}
\bysame, \emph{Periodic perturbations of unbounded {J}acobi matrices {III}:
  {T}he soft edge regime}, J. Approx. Theory \textbf{233} (2018), 1--36.

\bibitem{Discrete}
G.~\'{S}widerski and B.~Trojan, \emph{About essential spectra of unbounded
  {J}acobi matrices}, arXiv: 2006.07959, 2020.

\bibitem{ChristoffelII}
\bysame, \emph{Asymptotic behaviour of {C}hristoffel--{D}arboux kernel via
  three-term recurrence relation {II}}, arXiv: 2004.07826, 2020.

\bibitem{SwiderskiTrojan2019}
\bysame, \emph{Asymptotics of orthogonal polynomials with slowly oscillating
  recurrence coefficients}, J. Funct. Anal. \textbf{278} (2020), no.~3, 108326,
  55.

\bibitem{Totik2000}
V.~Totik, \emph{Asymptotics of {C}hristoffel functions for general measures on
  the real line}, J. Anal. Math. \textbf{81} (2000), 283--303.

\bibitem{Totik2009}
\bysame, \emph{Universality and fine zero spacing on general sets}, Ark. Mat.
  \textbf{47} (2009), no.~2, 361--391.

\bibitem{VanAssche1993}
W.~Van~Assche, \emph{Christoffel functions and {T}urán determinants on several
  intervals}, J. Comput. Appl. Math. \textbf{48} (1993), no.~1, 207--223.

\bibitem{Xu2011}
Sh.-X. Xu, Y.-Q. Zhao, and J.-R. Zhou, \emph{Universality for eigenvalue
  correlations from the unitary ensemble associated with a family of singular
  weights}, J. Math. Phys. \textbf{52} (2011), no.~9, 093302.

\bibitem{Zhou2011}
J.-R. Zhou, Sh.-X. Xu, and Y.-Q. Zhao, \emph{Uniform asymptotics of a system of
  {S}zeg\"o class polynomials via the {R}iemann–{H}ilbert approach}, Anal.
  Appl. \textbf{09} (2011), no.~04, 447--480.

\end{thebibliography}
\end{bibliography}

\end{document}